\documentclass[a4paper,11pt]{amsart}
\usepackage{amscd}

\usepackage{amsmath,amsfonts,amsthm,epsfig,latexsym,graphicx,amssymb}
\usepackage{tikz}
\usepackage{enumitem}
\usepackage[english]{babel}
\usepackage{xcolor}
\usepackage[colorlinks=true, citecolor=red]{hyperref}
\usepackage{verbatim}
\usepackage{fullpage}

\newtheorem{theorem}{Theorem}[section]
\newtheorem{corollary}[theorem]{Corollary}
\newtheorem{proposition}[theorem]{Proposition}
\newtheorem{definition}[theorem]{Definition}
\newtheorem{lemma}[theorem]{Lemma}
\newtheorem{claim}[theorem]{Claim}

\newtheorem*{theorem*}{Theorem}
\newtheorem*{proposition*}{Proposition}
\newtheorem*{definition*}{Definition}
\newtheorem*{lemma*}{Lemma}
\newtheorem*{claim*}{Claim}
\newtheorem*{corollary*}{Corollary}
\newtheorem*{convention*}{Convention}
\newtheorem{observation}[theorem]{Observation}

\newtheorem{thmintro}{Theorem}

\newtheorem{corintro}[thmintro]{Corollary}

\theoremstyle{definition}
\newtheorem{convention}[theorem]{Convention}
\newtheorem{example}{Example}

\newtheorem*{notation}{Notation}

\theoremstyle{remark}
\newtheorem{rem}[theorem]{Remark}
\newtheorem*{rem*}{Remark}
\newtheorem*{acknowledgement}{Acknowledgement}

\newcommand\bR{\mathbb R}

\newcommand\bZ{\mathbb Z}
\newcommand\bN{\mathbb N}
\newcommand\bQ{\mathbb Q}

\newcommand{\R}{\mathbb R}

\newcommand{\bD}{\mathbb D}

\newcommand{\graph}{\Gamma}

\newcommand{\cC}{\mathcal{C}}

\newcommand{\cF}{\mathcal{F}}

\newcommand{\cI}{\mathcal{I}}
\newcommand{\cJ}{\mathcal{J}}
\newcommand{\cL}{\mathcal{L}}

\newcommand{\cP}{\mathcal{P}}
\newcommand{\cR}{\mathcal{R}}

\newcommand{\cT}{\mathcal{T}}
\newcommand{\cE}{\mathcal{E}}

\newcommand{\cX}{\mathcal{X}}

\newcommand{\Pbound}{\partial P}

\newcommand{\Homeo}{\mathrm{Homeo}}

\newcommand{\geo}{\mathrm{Geo}}

\newcommand{\nonsing}{\mathrm{ns}} 

\newcounter{notes}
\title{Completing prelaminations}

\author[Thomas Barthelm\'e]{Thomas Barthelm\'e}
\address{Queen's University, Kingston, Ontario}
\email{thomas.barthelme@queensu.ca}
\urladdr{sites.google.com/site/thomasbarthelme}

\author[Christian Bonatti]{Christian Bonatti}
 \address{Universit\'e de Bourgogne, Dijon, France}
 \email{bonatti@u-bourgogne.fr}

\author[Kathryn Mann]{Kathryn Mann}
 \address{Cornell University, Ithaca, NY}
 \email{k.mann@cornell.edu}
\urladdr{https://e.math.cornell.edu/people/mann}

\begin{document}
\begin{abstract} 
Motivated by problems in the study of Anosov and pseudo-Anosov flows on 3-manifolds, we characterize when a pair $L^+, L^-$ of subsets of laminations of the circle can be {\em completed} to a pair of transverse foliations of the plane or, separately, realized as the endpoints of such a {\em bifoliation} of the plane.  (We allow also singular bifoliations with simple prongs, such as arise in pseudo-Anosov flows).   This program is carried out at a level of generality applicable to bifoliations coming from pseudo-Anosov flows with and without perfect fits, as well as many other examples, and is natural with respect to group actions preserving these structures.    
\end{abstract}

\maketitle

\section{Introduction} 
A pseudo-Anosov flow on a $3$-manifold on $M$ gives rise to an action of $\pi_1(M)$ on an open disc (a topological plane) equipped with two transverse, possibly singular, foliations \cite{Fen_Anosov_flow_3_manifolds}. Such a bifoliated plane can be compactified by adding a topological circle at infinity to which the action of the group extends in a natural way  \cite{Bonatti_boundary,Fen_ideal_boundaries}.\footnote{see also \cite{Mather} for the case of one foliation or \cite{Fra_Mobiuslike} for cirles at infinity associated to even more general decompositions of the plane.} Leaves of each foliation (or faces of singular leaves) can be associated to pairs of points in the circle, giving pairs of points which define a subset of a lamination.  

Here we show that one can uniquely reconstruct a bifoliated plane, and a group action, given relatively sparse data from a ``circle at infinity" alone.  This data may be as sparse as a countable set of leaves.  
We do this in a very broad topological context, but one directly applicable to flows. Our Theorem \ref{thm_completion} below (specifically, uniquely reconstructing a bifoliated plane from {\em countable} data in a group-equivariant way) is used in \cite{BBM} to show that pseudo-Anosov flows on compact 3-manifolds can be {\em classified up to orbit equivalence by actions on the circle} -- precisely, a flow on $M$ is determined by the action of $\pi_1(M)$ on the circle at infinity of the orbit space of the flow.  Reconstructing the plane is a procedure we call ``completion'':

\begin{definition}
If $L^+$ and $L^-$ are subsets of laminations of $S^1$, a {\em planar completion} is a pair of transverse (possibly singular) foliations $(\cF^+, \cF^-)$ of the disc such that $L^\pm$ correspond to endpoints in the circle at infinity of a dense subset of leaves of $\cF^\pm$.  
\end{definition} 

Our main result gives easily checkable combinatorial conditions describing exactly when subsets of laminations can be ``completed" in this way.   A pair of points on $S^1$ can be represented by a Euclidean straight line in the unit disc.  If $L$ is a collection of such pairs, we let $\geo(L)$ denote the union of the corresponding geodesics.  Our conditions simply encode the intersections of these geodesics.  

We first present a simple case that is already relevant to the question of flows. Call a subset of a lamination {\em regular} if no two geodesic of $\geo(L)$ share a common endpoint and each geodesic of $\geo(L)$ is accumulated on both sides; or equivalently, no component of $\bD^2 \setminus \geo(L)$ has a geodesic of $\geo(L)$ in its closure.  

\begin{theorem}[Special case of main theorem]\label{thm_special_case_regular}
Let $L^+, L^-$ be regular subsets of laminations of $S^1$, with no shared leaves.  This pair can be {\em completed} to a non-singular bifoliation of the plane if and only if the following two conditions hold: 
\begin{enumerate}
\item $\geo(L^+) \cup \geo(L^-)$ is connected through segments of geodesics and dense in $S^1$.
\item Each connected component of $\bD^2 \setminus (\geo(L^+) \cup \geo(L^-))$ is a quadrilateral bounded by alternating segments of $\geo(L^+)$ and $\geo(L^-)$.
\end{enumerate} 
Moreover, this completion is {\em unique} up to foliation-preserving homeomorphisms restricting to the identity on $S^1$, and any group acting on $S^1$ preserving $L^+, L^-$ extends to act (uniquely) by foliation-preserving homeomorphisms of the plane.  
\end{theorem} 
Examples satisfying these hypotheses include the (countable) laminations induced by the sets of leaves of an Anosov flow which contain periodic orbits unique in their homotopy class (see \cite{BBM}). In these examples, $\geo(L)$ is relatively sparse and has many complementary regions in $\bD^2$ with nonempty interior, as in Figure \ref{fig:intro} left.   In particular, even for these examples the ``completion" procedure is nontrivial, and not simply taking a ``closure" in the disc.  

As a consequence even of this simple case, the orbit equivalence class of an Anosov flow is determined by the intersection pattern of this countable set of periodic leaves.   

For the general case applicable to even sparser subsets of laminations, and to foliations with singularities, we replace the quadrilateral condition by instead describing the combinatorics of which sides of complementary regions of $\bD^2 \setminus \geo(L^\pm)$ are crossed by geodesics of $\geo(L^\mp)$.  This information is encoded in a graph, called the {\em linkage graph} of the region.  Instead of specifying quadrilaterals, we restrict the behavior of cycles and high-valence vertices in this graph (conditions \ref{item_simple_cycle} and \ref{item_no_high_valence} below).  We defer precise definitions to Sections \ref{sec:details} and \ref{subsec:linkage_graphs}.  

\begin{thmintro}[Main theorem -- completions] \label{thm_completion}
Let $L^+, L^-$ be two subsets of laminations of $S^1$ with no shared leaves.  The pair $(L^+, L^-)$ has a  {\em planar completion} if and only if the following 
conditions are satisfied: 
\begin{enumerate}[label=(\roman*)]
 \item $\geo(L^+) \cup \geo(L^-)$ is connected through segments of geodesics and is dense in $S^1$.  
 \item\label{item_simple_cycle} $(L^+,L^-)$ satisfies the simple cycle condition for linkage graphs,
 \item \label{item_no_high_valence} $(L^+,L^-)$ has no high valence leaves in the linkage graphs, and
 \item\label{item_same_endpoint_condition} No three $L^\pm$-leaves with a common endpoint cross a common $L^\mp$-leaf.
\end{enumerate}
Moreover, the planar completion is unique up to foliation-preserving homeomorphisms restricting to the identity on $S^1$. 
\end{thmintro} 

As in the special case, planar completions respect group actions:
\begin{thmintro}[Extending group actions]\label{cor:action_extends_2}
If $G$ is a group acting on $S^1$ preserving a pair $L^+, L^-$ satisfying the hypotheses of Theorem \ref{thm_completion}, then the action of $G$ extends uniquely to an action by homeomorphisms of the compactified plane, preserving the foliations of the completion.  
\end{thmintro}

Along the way to the proof, we also characterize exactly what subsets of laminations arise as those defined ``at infinity" by leaves of a pair of transverse singular foliations such as those arising from a pseudo-Anosov flow. 
To motivate this and other related results, we introduce the reader to the important objects of study.  

%
%

\subsection{An introduction to bifoliations and (pre)-laminations} \label{sec:details}
Much of the terminology in this section is standard or slight variations on standard terms adapted to a broad context allowing for the study of pseudo-Anosov flows, including those with ``perfect fits".  To keep this paper self contained, and to disambiguate from related work such as that from the study of pseudo-Anosov flows without perfect fits (a special case) treated by Baik, Jung, and Kim in \cite{BJK}, we give a detailed exposition. 

\begin{definition} 
A pair of foliations $\cF^+, \cF^-$ of a plane $P$ is called a {\em pA-bifoliation}  if $\cF^+$ and $\cF^-$ are transverse, have only isolated prong singularities, and each leaf contains at most one singularity. 
\end{definition} 
\noindent This class of bifoliations contains all examples induced by pseudo-Anosov flows on their orbit spaces, and many more beyond this.  

A subset of a singular leaf $l$ that bounds a connected component of $P \smallsetminus l$  is called a {\em face} of $l$.  
As remarked above, a plane with a bifoliation has a natural compactification to a closed disc by a circle boundary (we recall the construction in Section \ref{sec:circle_infty}).  This circle is denote $S^1_\infty({\cF^+, \cF^-})$, or $S^1_\infty$ when the bifoliation is unambiguous.    
Nonsingular leaves of $\cF^\pm$ and faces of singular leaves define pairs of points, called their ``endpoints" in $S^1_\infty$.  We call these sets of pairs the {\em prelaminations induced by $\cF^\pm$}.  More generally: 

\begin{definition}
A {\em prelamination} is a set of pairs of distinct pairs of points of $S^1$ such that no two pairs cross.  Equivalently, this is just a subset of a {\em lamination} in the sense of \cite[Definition 2.2]{Calegari_book}. 
\end{definition} 
Recall that, if $a_i,b_i$ are points in $S^1$, we say the pairs $\{a_1,a_2\}$ and $\{b_1,b_2\}$ {\em cross} if $a_1$ and $a_2$ lie in different connected components of $S^1\setminus \{b_1,b_2\}$. 
An element $\{a_1, a_2\}$ of a prelamination $L$ is called a {\em leaf}, and the points $a_1$ and $a_2$ are its {\em endpoints}.  

Prelaminations induced by a pA-bifoliation satisfy additional properties.  The most easy to check is the following: 

\begin{definition}\label{def_FT} 
Two prelaminations $L^+, L^-$ are \emph{fully transverse (FT)}\footnote{Note that FT is much weaker than the notions of strongly transverse or quite full of Baik-Junk-Kim.  ``Quite full" requires each complementary region to be either a finite-sided polygon or countable-sided polygon with vertices accumulating on a single point, ``strongly transverse" requires no leaf of $L^+$ to share an endpoint with a leaf of $L^-$. Moreover, both require (closed) {\em laminations} rather than prelaminations.}
if the following properties hold
\begin{enumerate}
 \item {\em Transversality:} $L^-\cap L^+=\emptyset$
 \item {\em Density:} the set $\{a \in S^1 : \exists  b \text{ with } \{a,b\}\in L^-\cup L^+\}$ is dense.  
 \item {\em Connectedness:} For any $\alpha, \beta \in L^-\cup L^+$, there exists a sequence $\alpha = \alpha_1, \ldots, \alpha_k = \beta$ such that $\alpha_i$ crosses $\alpha_{i+1}$ for all $i=1, \ldots, {k-1}$.  
\end{enumerate}
\end{definition}

\noindent This is a concise, precise rephrasing of the ``no shared leaves, density, and connectedness through segments of geodesics" assumptions that appeared in Theorem \ref{thm_completion} and Theorem \ref{thm_special_case_regular}, and we will use the shorthand FT going forward.  
The fact that the induced laminations of a pA-bifoliation are FT is shown in Proposition \ref{prop:converse_A}.

\subsection*{Geodesic realizations} 
It is often convenient to record the data of a prelamination by connecting pairs of points defining leaves by geodesics in the disc.  
The {\em geodesic realization} $\geo(L)$ of a prelamination $L$ is the set of Euclidean straight lines $\gamma$  in the unit disc $\bD^2$
such that the pair of endpoints of $\gamma$ is an element of $L$ in $S^1 = \partial \bD^2$. \footnote{One could of course equivalently use hyperbolic geodesics in the Poincaré disc model (and we shall do that in figures) instead of Euclidean straight lines, but the affine structure given by straight lines will be convenient to work with at times, particularly the result proved in the Appendix.}

For fully transverse laminations (but not in general! --- see Example \ref{ex:non_natural}), we show in the Appendix that the geodesic realization is natural in the sense that changing a pair of FT prelaminations by a homeomorphism of $S^1$ changes the geodesic realization by a homeomorphism of the disc. See Theorem \ref{t.naturality}. 
Thus, we can speak meaningfully of any properties of a pair of FT laminations invariant under homeomorphism by using their geodesic realizations.

\begin{definition}\label{def.complementary_region}
A {\em complementary region} to $L$ is the {\em closure} 
of a connected component of $\bD^2 \smallsetminus \geo(L)$.  A complementary region is {\em trivial} if it is a single geodesic line. 

The {\em boundary $\partial R$} of a complementary region $R$ consists of geodesic segments and (possibly degenerate) intervals of $S^1$.  
A geodesic boundary component of a complementary region is called a {\em geodesic side} and a maximal nondegenerate component of $\partial R \cap S^1$ is called an {\em ideal side}.  
\end{definition} 
For convenience, we have defined complementary regions to be closed sets.  This property will be used in Section \ref{sec_preliminaries}.  

Two types of complementary regions, shown in Figure \ref{fig:intro},  play a special role.  
\begin{definition}
A nontrivial complementary region $C$ of $L^\pm$ is called {\em one-root} if there exists a unique geodesic side $s$ (the ``root") such that 
\begin{itemize}
\item every leaf of $L^\mp$ which intersects $C$ intersects $s$, and
\item $s$ is not a leaf of $L^\pm$, but every other geodesic side of $C$ is. 
\end{itemize} 
\end{definition}

\begin{definition}
An {\em ideal polygon} is a complementary region with finitely many sides, all of which are leaves. 
Two ideal polygons $P^+, P^-$ of prelaminations $L^+, L^-$ (respectively) are called a {\em coupled pair} if their vertices are disjoint and any two vertices of $P^+$ are separated in $S^1$ by vertices of $P^-$, and vice versa.  
\end{definition} 

   \begin{figure}[h]
     \centering
     \includegraphics[width=9cm]{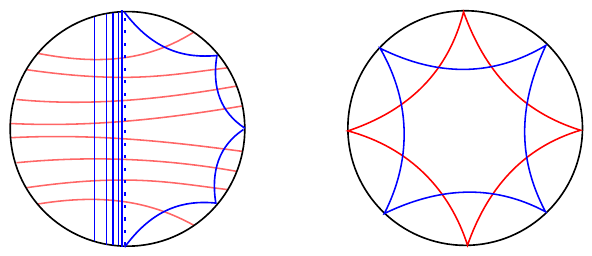}
     \caption{A one-root region and a coupled pair of ideal polygons}
     \label{fig:intro}
   \end{figure}
   
The {\em linkage graph} of a complementary region $R$ to $L^+$ encodes how $R$ interacts with $L^-$.  
When $R$ has only geodesic sides, the definition is simple to state:  vertices are geodesic sides of $R$, with an edge between two vertices if the sides are crossed by a common leaf of $L^-$.  
For ideal sides, we refer the reader to Definition  \ref{d.linkage-graph}. 

For example, the linkage graph of an ideal polygon is a cycle, and the linkage graph of a 1-root region is a {\em star} (a tree where all but one vertex, the ``root" has degree 1).

\subsection{Results} 
To prove Theorem \ref{thm_completion}, we first need to be able to characterizes the prelaminations that are induced from a pA-bifoliation.  We show the necessary conditions described above are in fact sufficient: 

\begin{thmintro}[Characterization of induced prelaminations] \label{thm:realization}
Let $L^+, L^-$ be prelaminations.  There exists a pA-bifoliation~on $\bR^2$ inducing the pair $(L^+, L^-)$ if and only if the following hold:
\begin{enumerate}[label=(\roman*)]
\item  \label{item.FT} $(L^+, L^-)$ is fully transverse 
 \item \label{item.countably_many} For any $a\in S^1$, the set $\{b : \{a,b\}\in L^-\cup L^+\}$ is countable
 \item\label{item.one_root_ideal_polygon}  Each complementary region of $L^\pm$ is an ideal polygon or one-root region.  Ideal polygons come in coupled pairs, and no leaf lies in the boundary of two ideal polygons.  
 \end{enumerate} 
Furthermore, the pA-bifoliation inducing $(L^+, L^-)$ is unique up to homeomorphism of $\bD^2$ which is the identity on the circle $S^1$. 
\end{thmintro}
A consequence of uniqueness is the following: 
\begin{thmintro}\label{cor:action_extends_1}
If $G$ is a group acting on $S^1$ preserving prelamination $L^+, L^-$ satisfying the hypotheses of Theorem \ref{thm:realization}, then the action of $G$ extends uniquely to $\bD^2$ preserving the pA-bifoliation inducing $(L^+, L^-)$.  
\end{thmintro}

Results in the vein of Theorem \ref{thm:realization} have been obtained by other authors.  For example, in the unpublished work \cite{Fra_preprint}, Frankel also obtained a sufficient condition for the existence of a pA-bifoliation of the plane associated to the ``spidery decompositions" he introduced in \cite{Fra_coarse_hyperbolicity}, towards the study of quasi-geodesic flows.  It is however not always clear to us how one passes from our conditions on prelaminations to Frankel's conditions on his spidery decompositions.   Baik--Jung--Kim \cite{BJK} give (much stronger) conditions on pairs of circle laminations invariant under the action of a group that allow one to reconstruct a bifoliated plane without perfect fits, and then apply work of Frankel, Schleimer and Segerman \cite{FSS22} to produce a veering triangulation.  We recently learned also of work in progress of Baik, Wu and Zhao generalizing this approach. 

Thus, to avoid problems of translation and generalization, we give a self-contained and elementary proof.  Unlike the authors above, we do not even need to use Moore's theorem on planar decompositions.

\subsection*{Naturality with respect to homeomorphism} 
Finally, we state the corollary of Theorem \ref{thm_completion} that we use in \cite{BBM}, which should be of general interest: 
\begin{corintro} \label{cor:for_nontransitive}
Let $(P_i,\cF_i^+,\cF_i^-)$, $i=1,2$ be pA-bifoliations. Let $\cL_i^\pm$ be subsets of $\cF_i^\pm$ and call $L_i^\pm$ the prelaminations induced by $\cL_i^\pm$ in $\Pbound_i:=S^1_\infty(\cF^+_i, \cF^-_i)$.   

Let $h\colon \Pbound_1 \to \Pbound_2$ be a homeomorphism; this induces an obvious map $\hat{h}\{x,y\} = \{h(x), h(y)\}$ on pairs of points in $\Pbound_1$.  
Suppose that:
\begin{enumerate}[label=(\roman*)]
\item The subsets $\cL_i^\pm \subset \cF_i^\pm$ are dense in the plane, and 
\item $\hat{h}(L_1^+\cup L_1^-) = L_2^+\cup L_2^-$. 
\end{enumerate}
Then there exists a unique homeomorphism $H\colon P_1 \to P_2$ such that $H|_{\Pbound_1} = h$, $H(\cF^+_1) = \cF^\pm_2$, and $H(\cF^-_1) = \cF^\mp_2$.
\end{corintro}

\subsection*{Embedding versus completion} 
In Theorem \ref{thm_completion}, we show that a planar completion is \emph{unique} if it exists. 
In section \ref{sec:uniqueness}, we treat the problem of {\em embedding} (rather than completing) a pair of fully transverse prelaminations in bifoliated planes.   We give examples of fully transverse prelaminations that arise as subsets of distinct (non-homeomorphic) pairs of pA-bifoliations, and give a necessary and sufficient condition to have a unique embedding. See Theorem \ref{thm_uniqueness_extension}.

\subsection*{The need for {\em pre}-laminations} 
It is important in this work to really consider \emph{pre}laminations instead of laminations (i.e.,~closed prelaminations) as is more usually done. Indeed, the induced prelamination from a single foliation (i.e. its boundary data) is a genuine lamination if and only if there are no nonseparated (branching) leaves, in other words, the leaf space of the foliation is hausdorff.  
A direct application of Theorem \ref{thm:realization} also shows that very few laminations themselves come from the boundary data of pA-bifoliations; specifically, one can show the following:  \\  
Let $(L^+, L^-)$ be a pair of FT \emph{laminations}.  Then  $(L^+, L^-)$ is induced by a pA-bifoliation if and only if  each complementary region of $L^\pm$ is an ideal polygon, and for each $a$, the set $\{b : \{a,b\}\in L^-\cup L^+\}$ is countable.

\subsection*{Further motivation and related work}
There are many reasons beyond pseudo-Anosov flows which motivate the study of group actions on bifoliated planes and their associated circles.  For instance, several authors have attempted to develop a theory of group actions on $S^1$ that  admit an analog of a Markov partition (see, e.g.,~\cite{ping_pong}). We expect that such a notion would be more easily defined and natural for the class of actions which extend to a bifoliated plane (see for instance \cite{Iakovoglou}).

Within the context of flows and 3-manifolds, there are many examples of the rich interaction of foliations, laminations, and other topolo-geometric structures.  We have already mentioned the work of \cite{FSS22}, who build a circle at infinity and a bi-foliated plane called the {\em link space} equipped with an action of $\pi_1(M)$ out of the data of a veering triangulation on a 3-manifold $M$; and conversely, show that the veering triangulation can be recovered from the dynamics of the action of $\pi_1(M)$ on the link space.  This is used also by \cite{BJK}, who broaden the picture by reconstructing link spaces out of certain circle laminations.   

In addition to the unpublished preprint mentioned above, Frankel's published work on quasigeodesic flows \cite{Fra_Mobiuslike,Fra_coarse_hyperbolicity} constructs a compactification of the orbit space of a quasigeodesic flow on a hyperbolic 3-manifold by an ideal circle, and studies the dynamics of the induced action of $\pi_1(M)$ (in particular, showing the existence of closed orbits of the flow). In much earlier work, which inspired several of the threads followed above, Thurston \cite{Thurston:3MFC} and Calegari--Dunfield \cite{CalDun_laminations}
showed that an atoroidal 3-manifold $M$ with either a taut foliation $\mathcal{F}$ or certain type of lamination has a {\em universal circle} relating the boundaries of leaves of the induced foliation in the universal cover, on which $\pi_1(M)$ acts by homeomorphisms.
Calegari \cite{Calegari_promoting} constructs a {\em pair} of $\pi_1(M)$-invariant laminations on this circle, and shows that (unless $\mathcal{F}$ is the weak-stable folation of an Anosov flow) these laminations of $S^1$ can be promoted to a pair of laminations on $M$ transverse to $\mathcal{F}$.

In another direction,  \cite{Gabai_Kazez}, Gabai and Kazez
show that an essential lamination of $\R^2$ can always be realized as a subset of a foliation of $\R^2$.  More precisely, they show that a lamination uniquely determines a one-dimensional object called a {\em cyclically ordered $\R$ order tree} up to isomorphism, and that the lamination, and a completion to a foliation, can be constructed out of the data given by this tree.  This is similar in spirit to our work here, but the setting is very different as they deal with only a single lamination, and are concerned with leaf spaces rather than ideal boundaries.

This is just a small sampling of related results, and far from exhaustive.  Further references and a detailed introduction can be found in Calegari's book \cite{Calegari_book}, or the recent survey in \cite{BaikKim}.

\subsection{Outline}
Since Theorem \ref{thm:realization} is used in the proof of Theorem \ref{thm_completion}, we prove it (and Corollary \ref{cor:action_extends_1}) in 
Section \ref{sec.plane_construction}.  This is elementary and can be read independently from the rest of the paper.

Section \ref{sec_preliminaries} develops a general toolkit to describe how the combinatorics of a pair of transverse prelaminations changes when leaves are added to one or both.  

Section \ref{sec:proof_completion_thm} gives the proof of the {\em existence} of a planar completion, proving half of Theorem \ref{thm_completion}.   The other half (necessity and uniqueness) is proved in Section \ref{sec:necessity_uniqueness}, along with Corollary \ref{cor:for_nontransitive} (see Corollary \ref{cor:for_nontransitive_inthepaper}), Theorem \ref{thm_special_case_regular} (see Corollary \ref{cor:restatmentofregularformintro}) and Theorem \ref{cor:action_extends_1}. 

Section \ref{sec:uniqueness} gives examples to show that the question of {\em embedding} a prelamination into a foliation is fundamentally different from that of completing, and gives a necessary and sufficient condition for a completable prelamination to admit only one embedding into a foliation (necessarily its completion).  

Finally, the appendix contains the proof of {\em naturality} of the geodesic realization of FT prelaminations (Theorem \ref{t.naturality}) and a cautionary example to show that naturality may not hold for general pairs of transverse prelaminations.

\begin{acknowledgement}
TB was partially supported by the NSERC (Funding reference number RGPIN-2017-04592). 
CB thanks the Banach center Bedlewo (for the conference ''Beyond Uniform Hyperbolicity'' in 2024) and the Simons Foundation Award No. 663281 granted to the Institute of Mathematics of the Polish Academy of Sciences for the years 2021-2023, where some of the ideas for this article were developed. 
KM was partially supported by NSF Career grant DMS-1933598, a Simons Fellowship in mathematics, and thanks the Institut Henri Poincar\'e Institut Henri Poincaré (UAR 839 CNRS-Sorbonne Université) and LabEx CARMIN (ANR-10-LABX-59-01). 
\end{acknowledgement}

\section{Proof of Theorem \ref{thm:realization}: Constructing the plane} \label{sec.plane_construction}

Most of this section is devoted to proving the existence part of Theorem \ref{thm:realization}. 
Let $L^+, L^-$ be prelaminations satisfying the hypotheses of Theorem \ref{thm:realization}.  We will use these to build an abstract space $\cP = \cP(L^+, L^-)$ and equip it with a natural topology.  We then show that this space is homeomorphic to the plane, has a natural bifoliation, and that this bifoliation induces $L^+, L^-$ on its circle at infinity.   Finally, we prove uniqueness in Proposition \ref{prop:uniqueness}.  
Necessity of the conditions of Theorem \ref{thm:realization} is shown in section \ref{sec:easy_direction}. 

We start by making the following obvious observation: Item \ref{item.one_root_ideal_polygon} of Theorem \ref{thm:realization} implies that, if $\alpha^+\in L^+$ crosses two sides of a complementary region of $L^-$, then either exactly one of these sides are leaves of $L^-$, or both sides are leaves and the region is an ideal polygon. We will use this simple fact many times in our construction.

By Theorem \ref{t.naturality}, there is a 
natural correspondence between $L^{\pm}$ and $\geo(L^{\pm})$ for fully transverse laminations, so we often work with $\geo(L^{\pm})$ instead of $L^{\pm}$ in the construction of $\cP$.

\subsection{The crossing space} 

\begin{notation}
Let $\cX = \{(\alpha^+, \alpha^-) \in L^+ \times L^- : \alpha^+ \text{ crosses } \alpha^- \}$.
\end{notation}
Define a relation on $\cX$ by $(\alpha^+, \alpha^-) \sim (\beta^+, \beta^-)$ if for each $\epsilon \in \{-, +\}$, either $\alpha^\epsilon = \beta^\epsilon$ or $\alpha^\epsilon$ and $\beta^\epsilon$ are two sides of the same ideal polygon.
In particular, this relation is trivial if there are no ideal polygons.

By identifying a pair $(\alpha^+, \alpha^-)$ with the intersection point of the geodesic representatives $\alpha^+$ and $\alpha^-$ in $\bD^2$, we can realize $\cX$ as a subset of $\bD^2$.  We often make this identification, so that (for example) 
when we write $I\cap \cX$ for some subset $I$ in $\bD^2$, we mean all the points $(\alpha^+, \alpha^-)$ such that $\alpha^+\cap\alpha^-\in I$.

\begin{lemma} 
$\sim$ is an equivalence relation
\end{lemma} 

\begin{proof}
Symmetry and reflexivity are immediate.  Transitivity follows from the assumption that no ideal polygons share a leaf.  
\end{proof} 

One can also explicitly describe the types of equivalence classes of $\sim$, as follows 
\begin{lemma}  \label{lem:equiv_classes}
Each equivalence class of $\sim$ is either 
\begin{enumerate} 
\item A singleton $(\alpha^+, \alpha^-)$ where no sides are boundaries of ideal polygons
\item Two points.  Up to reversing the roles of $+$ and $-$, these points are of the form $(\alpha^+, \alpha^-), (\alpha^+, \beta^-)$ where $\alpha^+$ is not a boundary of an ideal polygon but $\alpha^-$ and $\beta^-$ are boundaries of the same ideal polygon.
\item Four points, consisting of the crossing leaves of sides of two uncoupled ideal polygons
\item The $2k$ pairs of crossing leaves of two coupled ideal polygons, for $k>2$.
\end{enumerate} 
\end{lemma} 

The proof follows easily from the definition, and is left to the reader.  
We note in particular that any $L^\pm$-leaf $\alpha$ meets a class of $\sim$ in at most $2$ points.

\begin{notation} 
We denote the quotient space $\cX/\sim$ by $\cP$; and $q\colon \cX \to \cP$ the quotient map.  
\end{notation}

\begin{definition} \label{def:singular} 
A point $p \in \cP$ is {\em singular} if it is the equivalence class consisting of the $2k$ pairs of crossing leaves of two coupled ideal polygons, for $k>2$, and {\em nonsingular} otherwise. 
\end{definition} 

Choosing an orientation of a leaf $\alpha$ induces a total order $<_\alpha$ on the set of leaves crossing $\alpha$. 
We now define sets that will play the role of nonsingular leaves, or faces of singular leaves in $\cP$, and then show that the order described above descends to these sets.  

\begin{notation} 
For $\alpha \in L^+$, let $\cP(\alpha) := \{ q(\alpha, \alpha^-) : (\alpha, \alpha^-)\in \cX  \} \subset \cP$.  If $\alpha \in L^-$, we similarly define $\cP(\alpha) := \{ q(\alpha^+, \alpha) : (\alpha^+, \alpha)\in \cX \}$.
\end{notation}

\begin{lemma}  \label{lem:order_to_quotient}
Fix an orientation on $\alpha \in L^+$.  Let $\beta, \beta', \gamma, \gamma' \in L^-$ be four elements that all cross $\alpha$, and suppose $(\alpha, \beta) \sim (\alpha, \beta')$ and $(\alpha, \gamma) \sim (\alpha, \gamma') \not\sim(\alpha, \beta)$.   If $\beta <_\alpha \gamma$, then $\beta' <_\alpha \gamma'$.  

The analogous statement holds for $\alpha \in L^-$. 
\end{lemma} 

\begin{proof} 
For concreteness, assume $\alpha \in L^+$.
Let $\beta, \beta', \gamma, \gamma' \in L^-$ be as in the statement of the lemma.  
If $\beta = \beta'$ and $\gamma = \gamma'$ there is nothing to prove.  So suppose $\beta \neq \beta'$. 
Then $\beta$ and $\beta'$ are sides of the same ideal polygon, and thus adjacent points in the ordering along $\alpha$.  
In the case $\beta'<_\alpha \beta$, we have $\gamma' \geq_\alpha \beta >_\alpha \beta'$, since $\gamma'$ is also either equal to $\gamma$ or an adjacent point in the ordering.  Now consider the case $\beta' >_\alpha \beta$.  Since $(\alpha, \gamma') \not\sim(\alpha, \beta')$, we know that $\beta' \neq \gamma$ and $\beta' \neq \gamma'$.  Since no two ideal polygons share a side (in particular cannot share the side $\beta'$), we therefore have that $\gamma$ and $\gamma'$ are sides of a disjoint ideal polygon, hence $\beta' <_\alpha \gamma'$.  
\end{proof} 
Thus, $<_\alpha$ induces an order on $\cP(\alpha)$.  We show the following.  

\begin{lemma}\label{l.crossing_order}
There is a strictly increasing bijection from $\cP(\alpha), <_\alpha$ to the real line $\R$. 
\end{lemma}

For the proof, we call a subset $S$ of a totally ordered set $(Z, <)$ {\em a separating set} if for every $s_1 \neq s_2 \in S$ there exists $s_3 \in S$ with $s_1 < s_3 < s_2$.  

\begin{proof}[Proof of Lemme \ref{l.crossing_order}]
Fix $\alpha$. For concreteness we assume $\alpha$ lies in $L^+$.  Choose an orientation on $\alpha$ and let $<$ denote the order on $\cP(\alpha)$ induced by $<_\alpha$.  
We show first that $<$ has a countable separating set and no extremal elements.   Having done this, there is a unique (up to order-preserving homeomorphism) continuous map from this countable set to $\bQ$, and we show that it extends continuously to a surjection to $\R$.   

Working in $\bD^2$ with $\geo(L^\pm)$, consider the set of points $I:= \{ \alpha \cap \beta : \beta \in \geo(L^-)\}$; this is a subset of the geodesic $\alpha$; and $<_\alpha$ is simply the natural order on $I$. 

If $I$ has nonempty interior, let $S_1$ be a countable dense subset of the interior of $I$ (and let $S_1 = \emptyset$ otherwise).   Each component of $\alpha \smallsetminus I$ is the intersection of $\alpha$ with a complementary region of $L^-$, which are either ideal polygons or one-root regions.  In the first case, both endpoints belong to $I$ and in the second case exactly one endpoint belongs to $I$.  Let $S_2$ denote the union of all endpoints of components of $\alpha \smallsetminus I$ that lie in $I$; this is obviously countable.  Each point of $S_1 \cup S_2$ corresponds to a unique crossing pair $(\alpha, \beta)$ so we may identify it with a countable subset of $\cX$.  We will now show that $q(S_1 \cup S_2)$ is a separating set.

Given $p_1, p_2 \in \cP(\alpha)$ with $p_1 < p_2$, fix $\beta_i \in I$ such that $q(\alpha, \beta_i) = p_i$.  Since $p_1 \neq p_2$, the leaves $\beta_i$ are not the sides of a common ideal polygon.  Thus, (working again in the geodesic realization) there is some $\beta$ that intersects $\alpha$ between $\beta_1$ and $\beta_2$. Since there are no adjacent ideal polygons we can find such a $\beta$ which satisfy $q(\alpha, \beta) \notin \{p_1, p_2\}$.  If $\beta$ lies in the interior of $I$, then we may without loss of generality choose it so that $\alpha \cap \beta \in S_1$.  If not, we automatically have $\alpha \cap \beta \in S_2$.  Thus, $p_1 <  q(\alpha, \beta) <  p_2$ in $\cP(\alpha)$, as desired. 

Finally we show that $I$ has no maximum or minimum point.  Assume for a contradiction that $\beta$ represented a maximal point (the minimum case is similar).  Then, by definition, up to replacing $\beta$ with some $\beta' \sim \beta$ that is the side of an ideal polygon, the segment $J$ of $\alpha$ in $\geo(L^\pm)$ consisting of points greater than $\beta$ would be contained in a complementary region to $L^-$ bounded by $\beta$.  This complementary region cannot be an ideal polygon because the coupled condition forces $J$ to intersect another side of the region.  Similarly, it cannot be a one-root region because $J$ would need to cross also the non-leaf side, which is by definition accumulated by leaves of $L^-$ crossing $\alpha$, and so we would have another point of intersection.     
 
Thus, there is an order-preserving bijection $\phi$ from $\left( q(S_1 \cup S_2), <\right)$ to $\bQ$. As $q(S_1 \cup S_2)$ is separating for $\cP(\alpha)$, the bijection $\phi$ extends uniquely to an order preserving injective map from $\cP(\alpha)$ to $\bR$.  Abusing notation slightly we use $\phi$ to denote this map with domain $\cP(\alpha)$. We now check that it is surjective onto $\bR$.

Lift $\phi$ to a map $\Phi$ defined on $I\subset \alpha$ such that $\phi \circ q= \Phi$. 
Let $\tilde I\subset \alpha$ be the smallest (open) interval containing $I$.  The map $\Phi$ extends in a unique way as an orientation-preserving, nondecreasing map $\tilde \Phi\colon \tilde I\to\bR$, constant on each complementary component of $I$ in $\tilde I$.  Thus, $\tilde \Phi$ is continuous and surjective onto $\bR$. 

Now, consider any $t\in\bR$ and consider $\tilde \Phi^{-1}(t)$. 
Let $x \in \tilde \Phi^{-1}(t)$. If $x \in I$, then $t$ lies in the image of $\phi$.  If $x \notin I$, then $x$ lies in a complementary component $C$, and $\tilde \Phi$ is constant on the closure of $C$. By condition \ref{item.one_root_ideal_polygon} in Theorem \ref{thm:realization}, at least one endpoint, say $y$ of $C$ lies in $I$.  Thus, $\phi(q(y))=t$ proving that $\phi$ is surjective, which is what we needed to show.  
\end{proof}

\subsection{The polygon topology}

By Lemma \ref{l.crossing_order}, if we endow $\cP(\alpha)$ with the order topology, it is homeomorphic to $\R$.  In this section we define a topology on $\cP$ by specifying a (sub)-basis of open sets.  These are the {\em regular rectangles} and {\em regular prong polygons} defined below.   
The situation is simpler when $L^\pm$ have no ideal polygons, in which case one needs only to work with rectangles, and the bifoliated plane will be nonsingular.  A reader interested only in this setting can skip any mention of regular prong polygons below.  

\begin{definition} \label{def:regular}
A leaf $\beta\in L^\pm$ is \emph{regular} if it is accumulated on both sides by leaves of $L^\pm$.  
\end{definition} 

By definition, any non-regular leaf is the boundary of a complementary region.  Thus, there are at most countably many non-regular leaves.  As a consequence of this, we have
\begin{lemma} \label{lem:regular_dense}
Suppose $\beta_1, \beta_2, \beta_3$ are leaves crossing $\alpha$, with $\beta_2$ between $\beta_1$ and $\beta_3$.   Then there are uncountably many regular leaves crossing $\alpha$ between $\beta_1$ and $\beta_3$.  
\end{lemma} 

\begin{proof} 
As observed above, the set of regular leaves is countable.  Thus, the projection to $\cP(\alpha)$ of all regular leaves is a countable subset.  Since there are no adjacent ideal polygons, $(\alpha, \beta_1)$ and $(\alpha, \beta_3)$ project to distinct points in $\cP(\alpha)$.  By Lemma \ref{l.crossing_order}, $\cP(\alpha)$ has the order type of $\bR$, and hence there are uncountably many regular leaves crossing $\alpha$ between $\beta_1$ and $\beta_3$.  
\end{proof}

\begin{definition} 
A \emph{polygonal path in $\bD^2$} is a simple curve $c$ in $\bD^2$ consisting of finitely many leaf segments of $\geo(L^\pm)$.  These segments are called the {\em sides}, and they alternate between $L^+$ and $L^-$.  The endpoints of the sides are the \emph{a vertices} of $c$.

Such a path is called {\em regular} if all its sides are segments of regular leaves.  
\end{definition}

\begin{lemma}\label{l.polygon-with-more-sides}
Let $c$ be a closed polygonal path with more than $5$ sides and bounding a convex region $P$ in $\bD^2$.  Then the interior of $P$ contains the intersection of two coupled ideal polygons.
\end{lemma}

\begin{proof} 
Each leaf crossing $P$ splits $P$ into two convex polygons of which at least one has more than $5$ sides. We consider a maximal decreasing chain (i.e., a totally ordered set with respect to inclusion) of regular convex polygons with more than $5$ sides. The intersection is a convex polygon $C$ with more than $5$ sides,  whose sides are geodesic segment contained either in $L^\pm$-leaves or accumulated by $L^\pm$-leaves.

The interior of this polygon is disjoint from $L^-$ and $L^+$ and thus it is the intersection of complementary components of $L^+$ and $L^-$.  Neither can be a one-root region, since the one root property would force $C$ to have only four sides.  
Thus $C$ is the intersection of two coupled ideal polygons.
\end{proof}

\begin{definition}
A {\em regular polygonal path in $\cP$} is the projection to $\cP$ of the union of the sides of a regular polygonal path in $\bD^2$. 
\end{definition} 

\begin{definition}Let $c$ be a closed, regular polygonal path.  One says that a point $p\in\cP$ \emph{lies  in the interior of $c$} if its class $q^{-1}(p)$ is contained in the interior of $c$ (or equivalently, contains a point in the interior of $c$).
\end{definition}

We highlight two special kinds of polygonal paths.  
\begin{definition} 
A {\em regular rectangle} is a 4-sided, closed, regular polygonal path (in either $\bD^2$ or $\cP$).  A {\em regular prong polygon in $\bD^2$} is a closed regular polygonal path in $\bD^2$ with $2k$ sides whose interior contains the intersection of two coupled ideal $k$-gon polygons; its projection to $\cP$ is called a {\em regular prong polygon in $\cP$}.
\end{definition} 

\begin{lemma} \label{obs:one_prong}
If $R$ is a regular prong polygon in $\bD^2$, then there is a {\em unique} pair of coupled ideal polygon complementary regions whose intersection is contained in the interior of $R$.
\end{lemma} 

\begin{proof}  This is an easy consequence of Poincar\'e-Hopf, but one can also argue directly from the definition as follows: 
Containing the intersection of a pair of coupled $k$-gons forces a polygonal path to have {\em at least} $2k$ sides, since each side of the path can intersect at most one pair of  leaves of one ideal polygon (necessarily adjacent with a common endpoint).  If the interior of such a path also contained a (necessarily disjoint) intersection of ideal $l$-gons for some $l$, the path would need strictly more than $2k$ sides.   
\end{proof}

\begin{lemma} \label{l.nonsingular_points_in_rectangle}
Every nonsingular point (see Definition \ref{def:singular}) lies in the interior of a regular rectangle in $\cP$.
\end{lemma} 

\begin{proof}
Let $x=q(\alpha^+,\alpha^-)$ be a nonsingular point in $\cP$.\\
\noindent \textbf{Case 1: $\alpha^\pm$ not an ideal polygon side. }
As a first case, suppose that $\alpha^+$ is not a side of an ideal polygon.  Since each non-regular leaf is adjacent to a complementary region, there are at most countably many such leaves.  Thus, by Lemma \ref{l.crossing_order}, we can find two regular leaves $\alpha^-_1,\alpha^-_2$ with $q(\alpha^+,\alpha_1^-)$ and $q(\alpha^+,\alpha_2^-)$ on either sides of $x$ on $\cP(\alpha^+)$.

Consider one side of $\alpha^+$. 
If  $\alpha_+$ is accumulated by leaves on this side, Lemma \ref{lem:regular_dense} says that we can approximate $\alpha_+$ on this side by regular leaves, and find such a regular leaf crossing both $\alpha_1^-$ and $\alpha_2^-$.
If not, $\alpha^+$ bounds a one-root region $C$ on this side.  By definition, both $\alpha_1^-$ and $\alpha_2^-$ cross the unique side of $C$ that is not in $L^+$, and so cross the regular $L^+$-leaves accumulating on this side.  

This shows that there are regular $L^+$-leaves on both sides of $\alpha^+$ crossing  both $\alpha_1^-$ and $\alpha_2^-$. This proves that $x$ belongs to the interior of a rectangle.
An analogous proof works when $\alpha^-$ is not the side of an ideal polygon, so it remains to consider the following.  

\noindent \textbf{Case 2: ideal polygon sides. }
Say that $\gamma_1, \gamma_2$ are {\em successive sides} of an ideal polygon if they share a vertex.  
Suppose  
\[ q^{-1}(x)=\{(\alpha^+,\alpha^-),(\alpha^+,\beta^-), (\beta^+,\alpha^-), (\beta^+,\beta^-)\}
\] where $\alpha^+,\beta^+$ are successive 
$L^+$-sides of an ideal polygon $P^+$ and $\alpha^-,\beta^-$ are successive $L^-$-sides of an ideal polygon $P^-$.

By assumption, there are no adjacent ideal polygons, so one side of each of these four leaves $\alpha^\pm$, $\beta^\pm$ (the side opposite $P^\pm$) is either accumulated by regular leaves, or bounds a one-root region whose root is accumulated by regular leaves.  

As in the previous case, one can easily show that the regular $L^+$-leaves accumulating the leaves $\alpha^+$ and/or $\beta^+$ or the corresponding root 
all cross the regular $L^-$-leaves accumulating the leaves $\alpha^-$ and/or $\beta^-$ or the corresponding root of the adjacent one-root region. 
This shows that $x$ belongs to the interior of a rectangle.
\end{proof}

\begin{lemma} \label{l.singular_points_in_polygon}
Every singular point in $\cP$ lies in the interior of a regular prong polygon in $\cP$.  
\end{lemma} 

\begin{proof}
Let $x\in \cP$ be a singular point. So there exists $(\alpha_i^+,\alpha^-_i)\in \cX$, $i=1,\dots ,p$, such that $x = q(\alpha_i^+,\alpha^-_i)$, and $(\alpha_i^+,\alpha_{i+1}^-)\in \cX$ (with index considered modulo $p$).
Then as in the proof of Lemma \ref{l.nonsingular_points_in_rectangle}, either $\alpha^\pm_i$  is accumulated by regular $L^\pm$-leaves $\alpha^\pm_{i,n}$, or $\alpha^\pm_i$ bounds a complementary component whose geodesic root $\tilde \alpha^\pm_i$ is accumulated by regular leaves $\alpha^\pm_{i,n}$.

Then, again as in the proof of Lemma~\ref{l.nonsingular_points_in_rectangle}, one checks that, for $n$ large enough, the regular $L^+$-leaves $\alpha^+_{i,n}$ cross the regular $L^-$-leaves  $\alpha^-_{i-1,n}$ and  $\alpha^-_{i,n}$, building a convex regular polygon with $2p$-sides around the (non regular) intersection of the two ideal polygons whose corners define the class of $x$.
\end{proof}

\begin{notation}
For a regular rectangle $R$ we denote by $\mathring{R}$ the set of points of $\cP$ that lie in the interior of $R$.
\end{notation} 

\begin{lemma} \label{l.rectangles_are_rectangles}
For a regular rectangle $R$, there is a bijection  $\psi\colon\mathring{R}\to (0,1)^2$, such that for any $\alpha^\pm \in L^\pm$ 
\begin{itemize}
\item If $\cP(\alpha^+)\cap \mathring{R}\neq \emptyset$, then 
$\psi$ maps $\cP(\alpha^+)\cap\mathring{R}$ monotonically to a horizontal segment $(0,1)\times \{r\}$. 
 \item If $\cP(\alpha^-)\cap \mathring{R}\neq \emptyset$, then $\psi$ maps $\cP(\alpha^-)\cap\mathring{R}$ monotonically to a vertical segment $\{s\}\times(0,1)$.
\end{itemize}
\end{lemma}

\begin{proof}
Let $\alpha_1^\pm,\alpha_2^\pm$ denote the regular leaves defining $R$.  
Let $I_i^\pm$ be the interior of the side of $R$ that lies on $\alpha_i^\pm$.

First, define a map $(\psi^-, \psi^+) \colon \mathring{R} \to (\cX \cap I_1^-) \times (\cX \cap I_1^-)$
by sending $(\alpha^+, \alpha^-)$ to $((\alpha^+ \cap I_1^-), (\alpha^+ \cap I_1^-))$.  Note that this is a bijection. 

From Lemma \ref{lem:equiv_classes}, it follows easily that two points $x,y\in\cX \cap \mathring{R}$ satisfy $x \sim y$ if and only if $\psi^-(x)\sim\psi^-(y)$ and $\psi^+(x)\sim\psi^+(y)$. 
Thus the map $(\psi^-,\psi^+)$ descends to a bijection $\psi$ from the quotient $\mathring{R}/\sim$ to 
$I^-/\sim\times I^+/\sim$.  By Lemma \ref{l.crossing_order}, $I^\pm/\sim$ are in strictly increasing bijection with an open interval of $\bR$.  
\end{proof}

Similarly, we have a standard model for regular prong polygons.  
\begin{definition} 
\emph{A standard bifoliated $k$-prong polygon} is a bifoliated region with exactly one $k$-prong singularity, obtained by taking $k$ copies $R_i$, $i\in \bZ/k\bZ$, of the rectangle $[0,1) \times (-1, 1)$ with the standard product foliation, and gluing the half-side $\{0\} \times [0, 1)$ in $R_i$ to the half-side $\{0\} \times (-1, 0]$ in $R_{i+1}$ isometrically.
\end{definition} 

\begin{lemma} \label{lem:standard_model_prong}
For a regular prong polygon $R$ in $\cP$, there is a natural bijection from the set of points that lie in the interior of $R$ in $\cP$ to a standard bifoliated prong polygon, via a map sending segments of $\cP(\alpha)$, $\alpha \in L^+$ to vertical leaves (of the form $\{t\} \times (-1, 1)$) monotonically, and segments of $\cP(\beta)$, for $\beta \in L^-$ to the transverse foliation monotonically.
\end{lemma} 

The proof is similar to that of Lemma \ref{l.rectangles_are_rectangles}, the only significant modification is to use 
Observation \ref{obs:one_prong}, which implies that a regular prong polygon $R$ contains a unique singular point.  
Then, one shows just as in the previous lemma that if $x= q(\alpha_i^+,\alpha^-_i)$, $i=1, \dots, p$ is the singular point inside a standard prong polygon, then each ``half space'' bounded by $\alpha_i^+$ and the appropriate sides of the polygon is in bijective correspondence with $[0,1) \times (-1,1)$ in a way that sends leaves of $L^\pm$ to horizontal and vertical leaves, and the singular point is sent to the point $(0,0)$.  After identifying the sides that are equivalent in $\cP$, one obtains the standard model.  Further details are left to the reader.  

With this, we define a topology on $\cP$.

\begin{convention}
We endow $\cP$ with the topology generated by the subbase consisting of interiors of regular rectangles and regular prong polygons.
\end{convention}

\begin{lemma} 
With this convention, the bijections defined in Lemmas \ref{l.rectangles_are_rectangles} and \ref{lem:standard_model_prong} are homeomorphisms.  
\end{lemma} 

\begin{proof} 
We treat the case for rectangles, the prong polygon case is obtained by applying a similar argument to the standard prong model. 
 
Let $R$ be a regular rectangle and let $\psi$ be the bijection to $(0,1)^2$ defined in Lemma \ref{l.rectangles_are_rectangles}.  Then, by definition, $\psi$ maps the interior of regular sub-rectangles of $\mathring{R}$ onto open subrectangles of $(0,1)^2$. 
As a consequence of Lemma \ref{lem:regular_dense}, the regular leaves of $L^+$ and of $L^-$ that cross $\mathring{R}$ project to dense subsets (for the order topology) under $\psi^+$ and $\psi^-$.  Thus, the set of open subrectangles of $(0,1)^2$ that arise as image of regular sub-rectangles under $\psi$ form a subbase for the topology of $(0,1)^2$, which implies that $\psi$ is a homeomorphism.
\end{proof} 

\begin{proposition} \label{prop:intersections_rectangles}
The intersection of two regular rectangles or prong polygons is either empty, a regular rectangle, or in the case of two prong polygons that contain the same singular point, then their intersection is again a regular prong polygon.  
\end{proposition} 
Thus, what we originally specified as a subbase for the topology is in fact a basis for the topology.  

\begin{proof} 
We give the proof first for the intersection of rectangles.  Let $R_i$ $i = 1,2$  be regular rectangles, bounded by leaves $\alpha_i^\pm, \beta_i^\pm$.  Let $x \in R_1 \cap R_2$.  
Consider the leaves $\gamma^\pm(x)$ of $L^\pm$ through $x$.   We have that $\alpha_i^\pm \cap \gamma^\mp(x) \neq \emptyset$, and these are ordered along $\gamma^\mp$.   Then $R_1 \cap R_2$ is the rectangle bounded on two sides by two middle leaves of $\{\alpha_i^+, \beta_i^+\}$ (according to the order along $\gamma^-(x)$), and on the other two sides by the middle leaves of $\{\alpha_i^-, \beta_i^-\}$. 

The proof for intersections of prong polygons with rectangles is essentially the same; except in the case where the prong polygons contain a common singular point.  In this case, one simply considers instead the induced orders along the rays from the singular point.  
\end{proof}

\begin{proposition} \label{p.Hausdorff}
$\cP$ is Hausdorff. 
\end{proposition} 

\begin{proof} 
We need to show that any two points $p \neq  p' \in \cP$ can be put in disjoint open sets.  Consider regular rectangles $R, R'$ (or regular prong polygons if $p$ or $p'$ is singular) containing each in their interior.  Since there are only countably many non-regular leaves, the regular leaves have dense image and so, up to shrinking $R$ and/or $R'$ slightly, we may assume that leaves carrying their boundary do not pass through $p$ or $p'$.  
We consider several cases.  

First, if $R$ and $R'$ are disjoint, we are done.

Secondly, if $p$ and $p'$ belong to a same rectangle or prong polygon, one may assume then $R = R'$.  Then we can find smaller regular sub-polygons containing each using the fact that regular leaves are dense in each of the product foliations of $R$, see Lemma~\ref{l.rectangles_are_rectangles}.

As a final case, suppose $p'\notin R$ and $p\notin R'$ but the interior $\mathring{R}$ and $\mathring{R'}$ have nontrivial intersection. Since $p\neq p'$, $R$ and $R'$ are not prong polygons associated to the same singular point, it follows from Proposition \ref{prop:intersections_rectangles} that $\mathring{R}\cap \mathring{R'}$ is the interior of a rectangle, whose side are segments of sides of $R$ and  $R'$.   

Cut $R$ along the (regular) geodesic carrying the sides of $\mathring{R}\cap \mathring{R'}$ to obtain new polygonal regions. Since the geodesic sides of $\mathring{R}\cap \mathring{R'}$ do not contain $p$ or $p'$, so these points belong to  the interiors of these new regions (and they lie in distinct regions).  The region containing $p$ has no more sides than $R$ does, so is either a rectangle or a regular prong polygon, depending on whether $p$ is singular or not.  
Either way, this component and $R'$ are disjoint open sets containing $p$ and $p'$, respectively, which is what we needed to show.  
\end{proof}

\subsection{$\cP$ is a bifoliated plane} 
In the previous section, we showed that  $\cP$ is a Hausdorff surface.  
Moreover, we showed it admits an atlas of bifoliated charts, defining a pair of transverse singular foliations.  Denote these foliations by $\cF^+,\cF^-$. The curves $\cP(\alpha^\pm)$ for $\alpha^\pm$ regular leaves of $L^\pm$ are leaves of these foliations.  We call these \emph{regular leaves} of $\cF^\pm$, and they form a dense subset of $\cP$.  We choose the signs $+, -$ such that regular leaves of $\cF^\pm$ are the images of regular leaves of $L^\pm$. 
The connectedness property of fully transverse prelaminations (see Definition \ref{def_FT}) means that $\cP$ is connected and this global choice is well defined.  
  
In this section, we will show that $\cP$ is simply connected and that $(\cF^+,\cF^-)$ is a pA-bifoliation.
The proof of simple connectedness is somewhat technical, so we break it into several lemmas.  

\begin{lemma}  \label{l.homotop_regular}
Any loop $\gamma$ in $\cP$ is homotopic to a loop formed by a finite union of segments of regular leaves.  
\end{lemma} 

\begin{proof} 
Let $\gamma$ be a loop in $\cP$. 
By Lemmas \ref{l.nonsingular_points_in_rectangle} and \ref{l.singular_points_in_polygon}, the image of $\gamma$ in $\cP$ can be covered by finitely many regular polygons.  By Lemmas \ref{l.rectangles_are_rectangles} and \ref{lem:standard_model_prong}, these are homeomorphic to standard models, and so we can perform an isotopy of $\gamma$, inductively in each rectangle (relative to boundary), to produce a loop whose image consists of a finite union of horizontal and vertical segments.  Since regular leaves are dense (Lemma \ref{lem:regular_dense}), we can further perturb this loop so that it only lies on regular leaves. 
\end{proof} 

\begin{definition}[Lifts]
Consider a segment of a regular $\cF^\pm$-leaf bounded by two regular points, $p,p'$. By definition $q^{-1}(p)$ and $q^{-1}(p')$ are each a singleton, respectively $x$ and $x'$, and are on the same regular $L^\pm$-leaf. We call the geodesic segment $[x,x']$ the \emph{lift of $[p,p']$}.

If $\rho$ is a path consisting of a union of such geodesic segments, we call the union of their lifts the {\em lift of $\rho$}. 
\end{definition}

Note that each point of $\cX$ in the $L^\pm$-segment $[x,x']$ is either a singleton for the equivalence relation $\sim$, or its class is two points, which are adjacent along $[x,x']$.  Thus,  $q$ induces a monotone bijection from the  classes contained in $[x,x']$ onto the segment $[p,p']$.

Borrowing terminology from $\cX$, we call a simple loop formed by a finite union of segments of regular leaves a {\em regular polygonal path in $\cP$}, and the segments the {\em sides} of the path.  

\begin{lemma}\label{c.minimal-is-convex} 
Suppose that $\rho$ is a noncontractible polygonal path that has a {\em minimum} number of sides (where the minimum is taken over all noncontractible polygonal paths in $\cP$).  Then the lift of $\rho$ to $\bD^2$ bounds a convex disc.
\end{lemma} 

\begin{proof} 
Let $\hat \rho$ be the lift of $\rho$ to $\bD^2$ and $\Delta$ the disc bounded by $\hat \rho$. 
Assume for a contradiction that $\hat{\rho}$ is not convex. Thus there is a closed segment $J$ of a geodesic such that $J$ contains a side $I$ of $\hat \rho$, one endpoint of $J$ is an endpoint of $I$, the other endpoint of $J$ is also on $\hat \rho$, and $\mathring J\smallsetminus I$ is contained in the interior of the disc $\Delta$.

We cut $\hat \rho$ along  $J$. We get two components. The boundary of each component is a regular polygon with strictly fewer sides than $\hat \rho$.  They induce on $\cP$ two regular polygonal paths, whose concatenation is $\rho$.  Thus one of them is not homotopic to a point, but has strictly fewer sides than $\rho$, which contradicts the minimality.
\end{proof}

With these tools we can now finish the proof of simple connectedness.  
\begin{proposition} 
$\cP$ is simply connected.  
\end{proposition} 

\begin{proof} 
Suppose for a contradiction that $\cP$ is not simply connected.  
By Lemma \ref{l.homotop_regular}, there exists a regular noncontractible loop.  Choose such a loop $\rho$, as in Lemma \ref{c.minimal-is-convex}, that has a minimum number of sides.   Then its lift $\hat{\rho}$ to $\bD^2$ bounds a disc $\Delta$.  Denote by $I_1, \dots I_{2n}$ the sides of $\hat \rho$ ordered cyclically, so that $I_{2i-1}$ is on (the geodesic realization of) a $L^+$ leaf.  If $n=2$, then $\rho$ is a rectangle, which is contractible by definition.  Thus, we assume $n>2$.  

Consider all the $L^-$ leaves intersecting $I_1$. Since $\hat{\rho}$ bounds a topological disc $\Delta$ in $\bD^2$, each such leaf must intersect another side of $\hat\rho$. 
Assume first that there exists a regular leaf $\beta^-$ that intersects both $I_1$ and $I_{2i-1}$ with $2<i<n$. Then we can cut $\rho$ into two distinct regular polygonal paths using $\beta^-$. These paths will have respectively $2i$ and $2n - 2i +4$ sides. In particular, both paths have strictly less than $2n$ sides and at least one of them must be non-contractible when projected to $\cP$. This contradicts the fact that $\rho$ has the minimal number of sides amongst non contractible curves.

Therefore, all regular leaves  of $L^-$ intersecting $I_1$ must intersect either $I_3$ or $I_{2n-1}$.
Since regular leaves are dense in $\cP$, this shows that all leaves intersecting $I_1$ and some other side $l_k$ must intersect either $I_3$ or $I_{2n-1}$ 

The above argument is independent of the choice of starting side $I_1$. Thus we deduce that any leaf of $L^{\pm}$ intersecting a side $I_j$ must intersect either $I_{j+2}$ or $I_{j-2}$.  This implies that every (regular or not) convex polygon contained in $\Delta$ is either a rectangle or has the same number of sides as $\Delta$.
By Lemma~\ref{l.polygon-with-more-sides}, $\Delta$ contains the intersection of a coupled pair of ideal polygons.  As stated above, this intersection of coupled pairs is a polygon with $2n$ sides.  Thus the polygon bounded by $\rho$ in $\cP$ is homeomorphic to a standard prong polygon.  This contradicts the fact the $\rho$ is assumed to be noncontractible, which ends the proof.
 \end{proof}

To show $(\cP, \cF^+, \cF^-)$ is a pA-bifoliation, it remains only to prove the following.

\begin{proposition}
Each leaf of $(\cP, \cF^+, \cF^-)$ has at most one singularity.  \end{proposition}

\begin{proof} 
Let $x$ be a singular point in $\cP$.  Then $q^{-1}(x)$ consists of the 2k pairs of crossing leaves of two coupled ideal polygons.  Denote these ideal polygons by $P^+$ and $P^-$, and label the sides of $P^+$ by $\alpha_1, \ldots \alpha_k$ in cyclic order.  

Let $r$ be any half-leaf of $\cF^+(x)$, i.e., an infinite ray based at $x$.  Then $q^{-1}(r)$ contains infinite rays of $\alpha_i$ and $\alpha_{i+1}$ for some $i$ (indices taken cyclically).  Since there are no adjacent ideal polygons, $q^{-1}(r)$ does not contain points on any other leaves of $L^+$.  In particular, no other point on $r$ is singular.  
The same argument works for $\cF^-(x)$.  
\end{proof}

\subsection{Circle at infinity} \label{sec:circle_infty}
Finally, in this section we show that there is a homeomorphism $h\colon S^1 \to S^1_\infty(\cF^+, \cF^-)$ mapping each leaf of $L^\pm$ to the pair of endpoints of a leaf or face
 of $\cF^\pm$; and so that for any regular leaf $\alpha = \{a_1, a_2\}$, the images $h(a_i)$ are the endpoints of $q(\alpha)$.   
We refer the reader to \cite{Bonatti_boundary} for the formal definition of $S^1_\infty(\cF^+, \cF^-)$; the idea of the construction is to take the set of endpoints of rays of leaves as a cyclically ordered set, pass to a quotient by identifying any two rays with only a finite number in between them, and then ``complete" the resulting circularly ordered set to a circle in a manner analogous to taking Dedekind cuts.  

\begin{definition} A {\em ray} of $L^+ \cup L^-$ is a connected subset of $l\smallsetminus\{x\}$ where $l$ is a leaf  of $\geo(L^+ \cup L^-)$ and $x\in \mathring\bD^2$ is a point in $l$. 
Two rays define the same \emph{germ of ray} if they coincide outside of a compact set.
\end{definition} 
In order to avoid an heavy terminology, we generally say ``ray" instead of ``germs of ray" when this abuse of terminology does not create confusion.

There is a natural circular order on rays in $\R^2 \cong \bD^2 \smallsetminus S^1$, defined as follows.  

\begin{definition}  \label{def:order_rays}
Given three disjoint rays $r_1, r_2, r_3$, consider a simple closed curve $\gamma$ 
crossing each $r_i$ in exactly one point, and orient $\gamma$ as the boundary of its compact complementary component.   
We say the triple $r_1, r_2, r_3$ has positive orientation if the points $r_1\cap \gamma, r_2 \cap \gamma$ and $r_3 \cap \gamma$ are in positive (counterclockwise) cyclic order along $\gamma$.  This is independent of the choice of $\gamma$.  
We write $[r_1, r_2]$ for the set of rays between $r_1$ and $r_2$ in the positive sense in this order.
\end{definition} 
Note that 
the map associating its endpoint to each ray is a monotone map from the set of rays to $S^1$.

We first prove the following characterization of rays sharing an endpoint.
\begin{lemma} \label{lem:same_endpoint}
Two rays $r_1, r_2$ of $L^+ \cup L^-$ share the same endpoint if and only if either $[r_1, r_2]$ or $[r_2, r_1]$ is countable.  
\end{lemma}

\begin{proof}[Proof of Lemma~\ref{lem:same_endpoint}]
If $r_1, r_2$ have the same endpoint on $S^1$ then all rays between them (in one direction) share that endpoint, and by assumption \ref{item.countably_many} from Theorem \ref{thm:realization} this set is countable.

Conversely, suppose that $r_1, r_2$ do not share an endpoint.  We will show $[r_1, r_2]$ is uncountable, the other case being symmetric. Let $I$ be the interval of $S^1$ corresponding to $[r_1,r_2]$.  Consider a simple polygonal path $\gamma$ joining $r_1$ to $r_2$ and whose interior is disjoint from $r_1\cup r_2$.

By the density assumption of fully transverse prelaminations, there are infinitely many leaves of $L^+$ or $L^-$ having an endpoint in the interior of $I$.  Each of these leaves either crosses $\gamma$ or has both its endpoints in $I$.

If some $L^\mp$ leaf has both endpoints in $I$ then every $L^\pm$-leaf crossing this leaf has an endpoint in $I$.  So uncountably many regular leaves of $L^\pm$ have one endpoint in $I$. Since distinct regular leaves
have distinct endpoints, one deduces that uncountably many rays of regular $L^\pm$-leaves have an endpoint in the interior of $I$ and hence belong to $[r_1,r_2]$, finishing the proof in this case.  

Otherwise, no leaf has both endpoints in $I$, so there is some leaf segment $\gamma_i$ of $\gamma$ which is crossed by infinitely many $L^\pm$-leaves that have endpoint in $I$.   By Lemma \ref{lem:regular_dense}, 
 there are uncountably many regular $L^\pm$-leaves crossing $\gamma_i$ between two of these leaves. These regular leaves each have an endpoint in $I$, so define distinct rays in $[r_1,r_2]$, showing this set is uncountable.
\end{proof}

The same characterization as above also holds for the circle at infinity of $(\cF^+,\cF^-)$:
\begin{proposition}[See \cite{Bonatti_boundary} Theorem 2] \label{o.same}
Two rays $r_1,r_2$ of leaves or faces of leaves of $\cF^+\cup\cF^-$ have the same endpoint on $S^1_{\cF^{\pm}}$ if and only if one of the two intervals $(r_1,r_2)$ or $(r_2,r_1)$ (for the cyclic order on the rays of leaves) is countable.
\end{proposition}

We will now build a natural map from the rays of $L^+\cup L^-$ to those of $\cF^+\cup \cF^-$:
Let $\cR(L)$ and $\cR(\cF)$ be the sets of (germs of) rays in $L^+\cup L^-$ and $\cF^+\cup\cF^-$ respectively.

Given a ray $r$ inside a leaf $\alpha$ of $L^\pm$, we let $\cP(r)$ be the corresponding subset of $\cP(\alpha)$. Then $\cP(r)$ is a ray of $\cF^\pm$ contained in $\cP(\alpha)$. Moreover, if two rays $r,r'$ of $L^\pm$ define the same germ of ray, then $\cP(r)$ and $\cP(r')$ define the same germ of ray of $\cF^\pm$.
This gives a natural map
\[
\cP\colon \cR(L) \to \cR(\cF)
\]

If $\alpha \in L^\pm$ is not a side of an ideal polygon, then $\cP(\alpha)$ is a non-singular leaf of $\cF^\pm$ and  $q^{-1}(\cP(\alpha))=\alpha\cap \cX$.  If $\alpha$ is the sides of an ideal polygon, then $\cP(\alpha)$ is a face  of a singular leaf of $\cF^\pm$ and $q^{-1}(\cP(\alpha))$ consists of $\alpha\cap \cX$ together with all the points of $\cX$ in the equivalence class of the prong of $\cP(\alpha)$, as well as the two rays contained in that same ideal polygon, sharing an endpoint with $\alpha$, and starting at the corresponding preimage of the singularity.  
 
To avoid this problem, it is thus natural to restrict the map $\cP$ to rays not contained in ideal polygon.
Let $\cR_{\nonsing}(L)\subset\cR(\cF)$ be the subset of rays not contained in an ideal polygon of $L^+\cup L^-$, and $\cR_{\nonsing}(\cF)\subset \cR(\cF)$ the subset of rays contained in non-singular leaves of $\cF^+\cup\cF^-$.
By the paragraph above, 
$$\cP_{\nonsing}\colon \cR_{\nonsing}(L)\to \cR_{\nonsing}(\cF)$$
is a bijection.

Recall that the set $\cR_{\nonsing}(L)$ is endowed with a natural complete cyclic order. Similarly, the set $\cR_{\nonsing}(\cF)$ is also endowed with a natural cyclic order, up to a choice of an orientation of the bifoliated plane $\cP$.
We can therefore consider whether or not the map $\cP_{\nonsing}$ is monotone, as monotonicity is independent of the choice of orientation.

The main step for building the homeomorphisms $h\colon S^1\to S^1_\infty(\cF^+, \cF^-)$ is the following.
\begin{proposition}\label{p.P-is-monotonous} 
The bijection $\cP_{\nonsing}\colon \cR_{\nonsing}(L)\to \cR_{\nonsing}(\cF)$ is strictly monotone.

The map $\cP\colon \cR(L) \to \cR(\cF)$ is (posibly non strictly) monotone.
\end{proposition}

\begin{proof} Let $r_1,\dots,r_k$, $k>2$, be $k$ nonsingular rays in $\cR_{\nonsing}(L)$ and let $s_i=\cP_{\nonsing}(r_i)$, $i=1,\dots, k$, be the corresponding non-singular rays in $\cP$. 

Recall from Definition \ref{def:order_rays} that the cyclic order on rays is defined by intersections with simple closed curves.   Thus, we may find a curve $\gamma$ intersecting each ray $s_i$ exactly once at a point $y_i$, and thus specifying the cyclic order.  
Given any such curve, we may cover it with a finite number of rectangles and prong polygons and deform it slightly to be a simple, polygonal path with regular leaves as sides.   Thus, there is a 
well defined lift $\tilde \gamma$ to $\bD^2$.  The rays $r_i$ are disjoint outside of the disc bounded by $\tilde \gamma$ and they cross $\tilde \gamma$ in exactly one point, since the preimages $q^{-1}(y_i)$ are singletons $x_i$ on $\tilde\gamma$. Thus the $x_i$ are cyclically ordered on $\tilde \gamma$ as the $r_i$.

The projection $q\colon \tilde\gamma\cap\cX\to \gamma$ preserves (or reverses) the cyclic order between the segment constituting $\tilde \gamma$ and $\gamma$, is strictly monotone on the corners of the regular polygonal path, and is  monotonous on each side. Thus  the map $q$ is strictly monotone from $\{x_i\}$ to $\{y_i\}$.
Therefore, the map $r_i\mapsto s_i=\cP_{\nonsing}(r_i)$ is strictly monotone for the cyclic orders.

As this holds for any finite set of ray, one gets that the map $\cP_{\nonsing}$ is strictly monotone,  ending the first part of Proposition~\ref{p.P-is-monotonous}. 

Now, suppose that $r_1,\dots,r_k$, $k>2$, are  rays in $\cR(L)$ and $s_i=\cP(r_i)$, $i=1,\dots,k$, are the corresponding (possibly singular) rays.  We can apply the same argument as above, getting again a regular polygonal path that intersects each $s_i$ at $y_i$ with the correct cyclic order, whose lift $\tilde\gamma$ intersects each $r_i$ in $x_i$. Now, the preimages $q^{-1}(y_i)$, which are equivalence classes for $\sim$, contain $x_i$ but may also contain other $x_j$. However, if $x_j\in q^{-1}(y_i)$ then $y_i=y_j$ and  $s_i=s_j$.

Since the map $q$ is strictly monotone on the classes contained in each side, we deduce that the map $x_i\mapsto y_i$ is constant if $y_i=y_j$, and strictly monotone on distinct classes for $\sim$. Therefore, the map $r_i\mapsto s_i=\cP(r_i)$ is monotone, which is what we needed to show.  
\end{proof}

We are now ready to build the claimed homoemorphism between circles.  Let $E(L) \subset S^1$ be the set of endpoints of rays of $L^+ \cup L^-$, and $E(\cF)\subset S^1_\infty(\cF^-,\cF^+)$ the set of the endpoints of rays in $\cF^-\cup \cF^+$.

\begin{proposition}
The map $E(L) \to E(\cF)$ that sends the end $a\in E$ of a ray $r_a$ of $L^+ \cup L^-$ to the end of the ray $\cP(r_a)$ in $S^1_\infty(\cF^+, \cF^-)$ is well-defined and extends to a homeomorphism $h\colon S^1\to S^1_\infty(\cF^+, \cF^-)$.
\end{proposition}

\begin{proof}
Since the foliations are fixed, we use the notation $S^1_\infty$ for $S^1_\infty(\cF^+, \cF^-)$ in the proof.  
Let $E_{\nonsing}(L)\subset E$ be the set of endpoints of rays in $L^-\cup L^+$ which are not on ideal polygons, and $E_{\nonsing}(\cF)$ the set of the endpoints of rays on nonsingular leaves in $\cF^-\cup \cF^+$.
Call $e_L \colon \cR(L) \to E(L)$ and $e_\cF\colon \cR(\cF) \to E(\cF)$ then endpoint map which associate to a ray in $L^+\cup L^-$ (resp.~$\cF^+\cup\cF^-$) its endpoint in $S^1$ (resp.~$S^1_\infty$). 

By Lemma~\ref{lem:same_endpoint} and Observation~\ref{o.same}, two rays $r_1,r_2\in \cR_\nonsing(L)$ have the same endpoint on $S^1$ if and only if the rays
$\cP_{\nonsing}(r_1),\cP_{\nonsing}(r_2)$ have the same endpoint on $S^1_\infty$.
Therefore, the map $\cP_{\nonsing}\colon \cR_{\nonsing}(L)\to \cR_{\nonsing}(\cF)$ descends to a \emph{bijective} map $h_\nonsing = e_\cF \circ \cP_\nonsing \circ e_L^{-1} \colon E_\nonsing(L) \to E_\nonsing(\cF)$.

By definition of the cyclic order, the maps $e_L$ and $e_\cF$ are both monotone, and $\cP_{\nonsing}\colon \cR_{\nonsing}(L) \to \cR_{\nonsing}(\cF)$ is a strictly monotone bijection by Proposition \ref{p.P-is-monotonous}.
Therefore, $h_\nonsing$ is a monotone bijection. 
Since $E_{\nonsing}(L)$ and $E_{\nonsing}(\cF)$ are dense in $S^1$ and $S^1_\infty$ respectively, $h_{\nonsing}$ extends continuously to a homeomorphism $h\colon S^1\to S^1_\infty$.   
By definition, if $e_L(r_a)=a\in E_\nonsing(L)$, then $h(a) = e_\cF(\cP(r_a))$. To finish the proof, we just need to show that this equality still holds when if $e_L(r_a)=a\in E(L) \smallsetminus E_\nonsing(L)$.

If $r$ is a ray for $L^+\cup L^-$, it is accumulated on both sides, for the cyclic order, by rays $r_n$ which are not contained in ideal polygons. As $\cP\colon \cR(L)\to\cR(\cF)$ is monotonous and bijective from $\cR_{\nonsing}(L)$ to $\cR_{\nonsing}(\cF)$ one deduces that $\cP(r)$ is accumulated on both sides by the $\cP(r_n)$.
Using again the monotonicity of $e_L$ and $e_\cF$, we deduce that $e_L(r) = \lim_{n\to \infty} e_L(r_n)$ and $e_\cF(\cP(r)) = \lim_{n\to \infty} e_\cF(\cP(r_n))$. So continuity of $h$ allows us to conclude.
\end{proof}

\subsection{Uniqueness}
It remains to prove the uniqueness of the bifoliated plane obtained in the construction above.  This is an immediate consequence of the following, which says that the induced lamination on the boundary circle determines a pA-bifoliation .

The statement also gives the proof of Theorem \ref{cor:action_extends_1} from the introduction, which said that the construction of the bifoliated plane was natural with respect to group actions preserving the laminations.  

\begin{proposition} \label{prop:uniqueness}
Suppose that $(\cF_1^+, \cF_1^-)$ and $(\cF_2^+, \cF_2^-)$ are pA-bifoliations on the open disc, and there exists a homeomorphism $h\colon S^1_\infty(\cF^+_1, \cF^-_1) \to S^1_\infty(\cF^+_2, \cF^-_2)$ taking leaves of the induced lamination of $\cF^\pm_1$ to leaves of the induced lamination of $\cF^\pm_2$.  Then there is a unique homeomorphism $H\colon \bD^2 \to \bD^2$ taking leaves of  $\cF^\pm_1$ to leaves of  $\cF^\pm_2$ and restricting to $h$ on the boundary.

Furthermore, if $h_i: S^1_\infty(\cF^+_i, \cF^-_i) \to S^1_\infty(\cF^+_{i+1}, \cF^-_{i+1})$, $i = 1, 2$ are two such maps, with induced homeomorphisms $H_i: \bD^2 \to \bD^2$, then the map induced by $h_2 \circ h_1$ is $H_2 \circ H_1$.  
\end{proposition}

\begin{proof} 
Each point $x$ in the open disc can be written uniquely as an intersection $\alpha^+ \cap \alpha^-$ where $\alpha^\pm$ is a leaf of $\cF_1^\pm$. 
Let $\{a_1^\pm, a_2^\pm\}$ denote the endpoints of $\alpha^\pm$ in $S^1_\infty(\cF^-_1,\cF^+_1)$.  This set has two elements, except when $\alpha^\pm$ is a singular leaf, where it has $k$ elements for some $k>2$.  
By assumption, $\{h(a_i^+)\}$ and $\{h(a_i^-)\}$ are endpoints of leaves of $\cF_2^+$ and $\cF_2^-$, respectively, and since $h$ is a homeomorphism these leaves cross.  Define $H(x)$ to be the unique intersection point of these crossing leaves.    This gives an injective map (which is surjective by considering the obvious inverse) mapping each leaf of $\cF^\pm_1$ to a leaf of $\cF^\pm_2$ whose endpoints are the images by $h$ of the endpoints of the initial leaf.  We need only show that $H$ is continuous and that its union with $h$ on the boundary circle is continuous as well.  

The order on the set of leaves crossing a given $\cF^\pm$-leaf can easily be read from the endpoints. This implies that $H$ is a monotone bijection on each $\cF^\pm_1$ leaf to its image, and thus the restriction of $H$ to each leaf is a homeomorphism. 

If $x$ is nonsingular, then it has a foliated rectangle neighborhood, bounded by four leaves; the interior of this will map to a foliated rectangle neighborhood under $H$. By the discussion above, $H$ maps horizontal (resp.~vertical) segments of this rectangle homeomorphically on horizontal (resp.~vertical) segments of the image rectangle.  Thus $H$ is continuous at $x$.

Similarly, if $x$ is singular, it has a neighborhood homeomorphic (in a foliation-preserving way) to the standard bifoliated $k$-prong polygon, which will be mapped by $H$ to a corresponding neighborhood of $H(x)$.  This shows that $H$ is a homeomorphism on  the open disc.

Finally, we check continuity at the points of the circle. A neighborhood basis of a point
$t\in S^1_\infty(\cF^-_1,\cF^+_1)$ is obtained by considering proper embeddings $\gamma_n$ of $\bR$ in the disc obtained by concatenation of a ray of regular leaf of
$\cF^-_1\cup\cF^+_1$ with endpoint $a_n$, a ray   of regular leaf of
$\cF^-_1\cup\cF^+_1$ with endpoint $b_n$ and a segment $\sigma_n$ joining this two rays, so that the sequences  $a_n$ and $b_n$ converge to $t$ on both sides and in a monotone way.  We choose the segment $\sigma_n$ so that the half compact discs containing $t$, and bounded in $\mathring{\bD}^2$  by $\gamma_n$,
are decreasing for the inclusion, and their intersection is $t$.

The image by $H$ of such a sequence is a similar sequence for $\cF^-_2,\cF^+_2$, proving the continuity of $H$ at $t$.  Thus $H$ is a homeomorphism extending $h$. Uniqueness is immediate from the construction, as is the fact that this construction respects composition.  
\end{proof}

\subsection{Converse} \label{sec:easy_direction}

Having proved the hard direction, we now prove the easy direction of Theorem \ref{thm:realization}.  This is the following statement.  

\begin{proposition} \label{prop:converse_A}
Suppose that $(L^+, L^-)$ is the pair of prelaminations induced by a pA-bifoliation  $(\cF^+,\cF^-)$.  Then
\begin{enumerate}[label=(\roman*)]
\item  $L^+, L^-$ are fully transverse
 \item  For any $a\in S^1$, the set $\{b : \{a,b\}\in L^-\cup L^+\}$ is countable
 \item  Each complementary region of $L^\pm$ is an ideal polygon or one-root region.  Ideal polygons come in coupled pairs, and no leaf lies in the boundary of two ideal polygons.  
 \end{enumerate} 
\end{proposition}

\begin{proof}
We check each property; the first two are easy.  

\smallskip
\noindent \textbf{Fully transverse property \ref{item.FT} and countably many leaves with shared endpoints \ref{item.countably_many}.} 
Transversality is immediate, and density follows from the definition of $S^1_\infty(\cF^+, \cF^-)$.  Connectedness follows from the connectedness of the plane and the fact that any path between leaves of $\cF^+ \cup \cF^-$ can be deformed to one that follows alternating segments of $\cF^+$ and $\cF^-$.  

The fact that sets of the form $\{b : \{a,b\}\in L^-\cup L^+\}$ are countable is an immediate consequence of Proposition \ref{o.same}.  

\smallskip
\noindent \textbf{Properties of complementary regions \ref{item.one_root_ideal_polygon}.} 
We now prove the desired properties hold for complementary regions to $L^+$;  the same proof works for $L^-$.

Let $C$ be a complementary region of $L^+$.  If $C$ is an ideal polygon, then the sides of $C$ define a singular leaf $l$ of $\cF^+$.   There is a coupled ideal polygon given by the faces in $\cF^-$ of the prong singularity in $l$.   Since there are no saddle connections in the bifoliated plane $\cF^+, \cF^-$, there are no adjacent ideal polygons to $C$.  
\smallskip 

Consider now  a complementary region $C$ of $L^+$ which is not an ideal polygon. We need to show $C$ is a one-root region.   
Denote by $\alpha^k$, $k=1, 2, ...$, the (finitely or countably many) geodesic sides of $C$. For each $k$, take $\alpha_n^k$ a monotone sequence of leaves of $L^+$ converging to $\alpha$ from the side opposite $C$. If $\alpha^k\in L^+$, we take $\alpha_n^k=\alpha^k$ for all $n$.
Let $l_n^k = \cP(\alpha^k_n)$ be the corresponding non-singular leaf or face of singular leaf in $\cF^+$. 

For each $n$, let $C_n$ be the convex hull in $\bD^2$ of the geodesics $\alpha^k_n$; $k = 1, 2, \ldots$. 
There is a unique connected component of $\bD^2 \smallsetminus (\cup_k l_n^k)$ whose boundary contains $\cup_k l_n^k$, we let $K_n$ denote this connected component.  Note that a leaf $\alpha$ of $L^+$ lies in the interior of $C_n$ iff $\cP(\alpha)$ lies in $K_n$.  
We have that $C_{n+1}\subset C_n$ and $K_{n+1}\subset K_n$. Moreover, the interior of each $K_n$ is saturated by leaves of $\cF^+$.  

By definition, $C= \cap_n C_n$. Let $K= \cap_n \overline{K_n}$. Since the $K_n$ is a decreasing sequence of compact  simply connected regions saturated by $\cF^+$ in $\bD^2$, the set $K$ is connected, simply connected, and $\cF^+$-saturated. Hence, $\partial K \cap \mathring{\bD}^2$ consists of a union of leaves or faces of leaves of $\cF^+$.

Moreover, $K$ has empty interior: Otherwise, there would be a leaf $l^+$ that (by density) we can choose to be non-singular, in the interior of $K$, whose corresponding leaf $\alpha$ of $L^+$ would have to be in the interior of $C$, which is impossible. We conclude $K \cap \mathring{\bD}^2$ is a connected union of leaves or faces of leaves of $\cF^+$ with empty interior.

\begin{claim}
There exists at least one $k_0\in J$ such that $\alpha^{k_0}$ is \emph{not} a leaf of $L^+$.
\end{claim}

\begin{proof}
 If all the geodesic sides of $C$ are leaves of $L^+$, then by our construction, $\alpha^k = \alpha^k_n$ and thus $l^k = l^k_n$, for all $n$, and the union of the leaves $l^k$ bounds $K$.  We just showed $K$ has empty interior. Thus, $K$ must be a singular leaf whose faces are the $l^k$.  We conclude $C$ is an ideal polygon, which contradicts our assumption.
 \end{proof}

Up to renaming the sides of $C$, we now assume that $\alpha^0$ is not a leaf of $L^+$. We next want to show the other sides are leaves.  

\begin{claim}\label{claim_leaves_in_K_are_geodesic_sides}
Every leaf or face in $K$ corresponds to a geodesic side $\alpha^k$ of $C$ such that $\alpha^k\in L^+$.
\end{claim}

\begin{proof}
Let $l$ be a leaf or face of $\cF^+$ in $K$.
Let $\alpha$ be the leaf of $L^+$ such that $\cP(\alpha) = l$.  Then $\alpha$ must be a side of $C$, since otherwise, for large enough $n$, $\alpha$ is not contained in $C_n$, so $l$ would not be in $K_n$, a contradiction.
\end{proof}

\begin{claim}\label{claim_every_leaf_in_limit}
Every point in $K$ is accumulated upon by the sequence $l_n^0$.
More precisely, every leaf $l$ of $\cF^+$ in $K$  is in the limit of $l_n^0$, and so is every point in $K\cap S^1$.
\end{claim}

\begin{proof}
We start by proving that every leaf in $K$ is accumulated by $l_n^0$. 
Suppose that a leaf $l$ of $\cF^+$ is in $K$ but is not in the limit of $l_n^0$. Then there exists a leaf or face $f\in \lim l_n^0$ such that $f$ separates $l$ from the $l_n^0$. By Claim \ref{claim_leaves_in_K_are_geodesic_sides}, $l$ and $f$ corresponds to geodesic sides $\alpha^{k_0}$ and $\alpha^{k_1}$, respectively, of $C$, and the fact that $f$ separates $l$ from the $l_n^0$ implies that $\alpha^{k_1}$ separates $\alpha^{k_0}$ from $\alpha^0$. We conclude $l$ is not in $K$.

Now we prove the second statement. If $x\in K\cap S^1$, either $x$ is an endpoint of a leaf $l$ in $K$, so the first part of the proof yields the result, or $x$ is contained in an open interval of $K\cap S^1$. This open interval corresponds to an ideal side $I$ of $C$ (since $K_n\cap S^1 = C_n \cap S^1$). So as in the first part of the proof, if $x$ is not accumulated by $l_n^0$, then there would be a side of $C$ that separates $I$ from $\alpha^0$, which once again is impossible.
\end{proof}

Using this we can show:
\begin{claim}
$\alpha^0$ is the unique geodesic side of $C$ that is not a leaf of $L^+$.
\end{claim}

\begin{proof}
If $\alpha^0$ is the unique geodesic side of $C$, then, for any fixed $l$ intersecting $C$, the set of leaves meeting $l$ is bounded on one side by $\alpha^0$, but is homeomorphic as an ordered set to $\bR$, so cannot include $\alpha^0$, which is what we needed to show. 
 So we assume that $C$ has at least two geodesic sides, and hence must have at least a third side possibly an ideal one. 
 Let $\alpha^1$ be another side and $l_n^1$ the corresponding sequence of leaves of $\cF^+$.
 By Claim \ref{claim_every_leaf_in_limit}, if $\alpha^1$ is not in $L^+$, then both $l_n^0$ and $l_n^1$ must accumulate onto the whole of $K$, which can happen only if $K$ reduces to one leaf of $L^+$ or one ideal segment. But this is not the case by our assumption.
\end{proof}

Thus, we have proven that there exists a unique geodesic side $\alpha^0$ of $C$ that is not a leaf of $L^+$, and Claim \ref{claim_every_leaf_in_limit} directly imply that every leaf of $L^-$ that crosses $C$ must intersect $l_n^0$ for $n$ large enough, hence must intersect $\alpha^0$. That is, we proved that $C$ is a one-root region.
\end{proof}


\section{A toolkit for completing prelaminations to bifoliations}\label{sec_preliminaries}
Theorem \ref{thm:realization} described the behavior of complementary regions of induced prelaminations.  In order to complete sparser prelaminations to pA-bifoliations, we need an efficient way to encode the configurations of complementary regions, so we can describe how these change when one adds or removes leaves.  This data is captured by the {\em linkage graphs}.  

We treat this in full generality, but the situation is simpler in certain cases (see Remark \ref{rem:easy_cases}) such as nonsingular bifoliations.  The reader may wish to follow the advice of Remark \ref{rem:easy_cases} on a first pass.  

\subsection{Linkage graphs and their basic properties}  \label{subsec:linkage_graphs}
Recall from Definition \ref{def.complementary_region} that the boundary $\partial C$ of a complementary region $C$ consists of {\em geodesic sides} contained in the interior of the disc and possibly {\em ideal sides} on the circle. We further subdivide ideal sides of $C$ as follows. 

\begin{definition} \label{d.ideal-segments}
Let $(L^+, L^-)$ be a pair of FT prelaminations, and $C^\pm$ a nontrivial complementary region of $L^\pm$. 
An  \emph{an ideal segment of $C$} is a
a compact segment $I\subset S^1$ with non-empty interior $\mathring I$  
such that:
\begin{itemize}
\item There exists a single geodesic side $\alpha$ of $C$ such that the endpoints of $L^\mp$-leaves crossing $\alpha$ are dense in $\mathring I$, and 
 \item $I$ is maximal for the property above.
\end{itemize}
\end{definition}
\begin{rem}\label{rem_ideal-segments} 
One easily checks that if $I$ is an ideal segment, then every 
 $L^\mp$-leaf having an endpoint in $\mathring I$ crosses the same geodesic side $\alpha$ of $C$.  As a consequence,  distinct ideal segments always have disjoint interiors.
\end{rem}
See Figure \ref{fig:linkage_graph} for an illustration. 

\begin{definition}[Linkage graph] \label{d.linkage-graph}
Let $(L^+, L^-)$ be  a pair of FT prelaminations, and $C^\pm$ a complementary region of $L^\pm$. The {\em linkage graph $\graph(C^\pm;L^\mp)$} is defined as follows:
\begin{itemize}
\item The vertices are:
\begin{enumerate}[label=(\roman*)]
\item Geodesic sides of $C^\pm$, which we call \emph{geodesic vertices} and
\item The ideal segments of $C^\pm$ (see Definition~\ref{d.ideal-segments}), which we call the \emph{ideal vertices}.
\end{enumerate}
\item Edges:  Two distinct vertices (geodesic or ideal) are joined by an edge if there exists a leaf $\alpha$ of $L^\mp$ \emph{intersecting} both.
\end{itemize}
\emph{Intersecting a geodesic vertex} means the leaf crosses the geodesic, and \emph{intersecting an ideal vertex} means an endpoint of the leaf on $S^1$ lies in the interior of the corresponding ideal segment.
Any $L^\mp$-leaf $\alpha$ with this property is said to {\em define} the edge $e$.
\end{definition} 

   \begin{figure}[h]
     \centering
     \includegraphics[width=8cm]{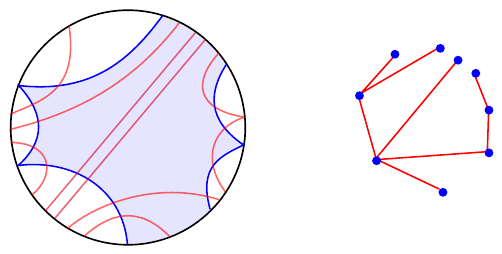}
     \caption{A complementary region (in blue) and its linkage graph.  Note the ideal side divided into three ideal segments giving three vertices.  See figures \ref{fig:condition_4} or \ref{fig:crossing} for examples with a cycle.}
     \label{fig:linkage_graph}
   \end{figure}
   
\begin{rem} 
If $C$ is a trivial complementary region, then its linkage graph consists of a single vertex and no edges.  We will mostly ignore trivial complementary regions. 
\end{rem}

Note that the linkage graph depends both on the complementary region $C^\pm$ of $L^\pm$, and on  its position relative to $L^\mp$. When the prelamination $L^\mp$ is clear from context, we will just write $\graph(C^\pm)$ for the linkage graph with respect to $L^\mp$.

We often think of $\graph(C)$ as a {\em labelled graph}, with vertices labelled by the geodesic sides or ideal segments defining them. That is, we make the following convention:

\begin{convention}
Going forward, we always identify a vertex of a linkage graph with its corresponding geodesic side, or ideal segment. Thus, when we say ``let $s$ be a vertex of $\graph(C^\pm;L^\mp)$" we mean that $s$ is a geodesic side or an ideal segment of $C^\pm$.
\end{convention}

Later, we will also keep track of the cyclic order on vertices induced by the order on $\partial C$.

\begin{definition}[Simple cycle condition]\label{d.simple-cycle} 
A linkage graph $\graph(C^\pm;L^\mp)$ satisfies the {\em simple cycle} condition if each edge $e$ is contained in at most one cycle.\footnote{Recall a {\em cycle} in a graph is a simple closed curve contained in the graph.}

A pair of FT prelaminations $(L^+,L^-)$ satisfies the simple cycle condition if the linkage graph of every complementary region does.
\end{definition} 

\begin{definition}[High valence]\label{d.high-valence}  
A {\em high-valence side} of $C^\pm$ is a side $s$ satisfying {\em at least one} of the following conditions:
\begin{enumerate}[label=(\roman*)]
\item\label{item_high_val_cycle_greater_3} $s$ has valence $\geq 3$ on a cycle of $\graph(C^\pm; L^\mp)$,
\item\label{item_high_val_greater_2} $s$ has valence $\geq 2$ and is not on any cycle of $\graph(C^\pm; L^\mp)$,
 \item\label{item_high_val_degenerate}  $\graph(C^\pm; L^\mp)$ has only two vertices (one being $s$); in that case we say that $\graph(C^\pm; L^\mp)$ is degenerate.
 \item\label{item_high_val_two_polygons} $s$ is the side of two complementary regions $C^\pm_0$ and $C_1^\pm$  and has valence $\geq 2$ on both graphs $\graph(C_i^\pm; L^\mp)$. (See e.g. Figure \ref{fig:condition_4})
\end{enumerate}
We say $\alpha \in L^\pm$ is a {\em high valence leaf} if it is a high-valence side of some complementary region. 
\end{definition}

   \begin{figure}[h]
     \centering
     \includegraphics[width=8cm]{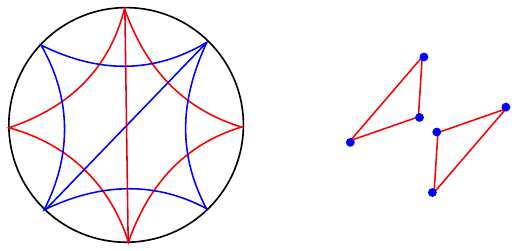}
     \caption{An example of a high valence leaf (blue, diagonal) satisfying \ref{item_high_val_two_polygons} but none of the other conditions, with the linkage graphs for the two complementary regions of the blue foliation shown.} 
     \label{fig:condition_4}
   \end{figure}

\begin{rem} \label{rem:easy_cases}
There are conditions that can simplify the definition and properties of linkage graphs that the reader may want to assume on a first reading:
\begin{itemize}
\item If the two sets $\{a\in S^1 : \{a,b\}\in L^+\}$ and $\{a\in S^1 : \{a,b\}\in L^-\}$ are \emph{both} dense, then the complementary regions of $L^\pm$ have no ideal segments in their boundary. Thus, the reader who wants to add this assumption may ignore every argument involving ideal vertices of linkage graphs.
\item For the readers who are only interested in \emph{non-singular} bifoliations, they may consider pairs of prelaminations such that their linkage graphs have \emph{no} cycles.
\end{itemize}
\end{rem}

We will now establish some basic properties about the structure of the linkage graphs $\graph(C^\pm;L^\mp)$ for FT prelaminations. For instance, we show linkage graphs are connected (Lemma \ref{lem:linkage_connected}), and ideal segments have valence $1$ (Lemma \ref{l.no-ideal-connection}).

\begin{convention}
For the remainder of this section, $(L^+, L^-)$ denotes a pair of FT prelaminations and $C$ is a fixed, {\em nontrivial} complementary region of $L^\pm$. 
\end{convention} 

\begin{lemma}\label{l.no-ideal-connection} 
Any $L^\mp$-leaf with non-empty intersection with $C$ intersects a geodesic boundary component of $C$.  Consequently, no edges of $\Gamma(C)$ joins two ideal vertices, and every ideal vertex has valence equal to $1$. 
\end{lemma}

\begin{proof} If a leaf $\beta$ of $L^\mp$ with non empty intersection with $C$ does not cross a geodesic boundary component of $C$, then it is contained in $C$. Thus $\beta$ does not cross any $L^\pm$-leaf.  This contradicts the density and connectedness property of FT prelaminations.

Since complementary regions are closed and convex, if a $L^\mp$-leaf has both endpoints in ideal boundary components of $C$, it is therefore contained in of $C$. Thus, it is either a geodesic boundary component (which is impossible) or is contained in $C$ and one concludes as before.
An ideal vertex  is by definition joined to a unique geodesic vertex by an edge (see Definition \ref{d.ideal-segments} and Remark \ref{rem_ideal-segments}). Thus ideal vertices have valence $1$.
\end{proof}

\begin{rem}\label{rem_no-ideal-connection} 
The statement and proof of Lemma \ref{l.no-ideal-connection} also holds for any geodesic $\beta$ that intersects the complementary region $C$, as long as $\beta$ does not cross any $L^\mp$-leaf. 
Indeed, if such a geodesic did not cross a geodesic side of $C$, then it would contradict the density or connectedness properties of $(L^+,L^-)$. 
\end{rem}

\begin{lemma} \label{lem:linkage_connected}
For any complementary region $C$ of $L^\pm$, the graph $\graph(C)$ is connected.
\end{lemma}
\begin{proof}
By Lemma~\ref{l.no-ideal-connection} every ideal vertex of $\Gamma(C)$ is connected to a geodesic vertex.  Thus it is enough to check that any two geodesic vertices $\alpha,\beta$ are connected by a path.  Each of $\alpha,\beta$ is either accumulated upon by leaves $\alpha_n, \beta_n\in L^\pm$ contained in the half disc bounded by the geodesic side $\alpha$ or $\beta$ and disjoint from $C$, or is already a leaf of $L^\pm$ (in which case we can write $\alpha_n=\alpha, \beta_n = \beta$ for all $n$).

By the connectedness property of fully transverse prelaminations, there is a polygonal path $\gamma$ joining $\alpha_n$ to $\beta_n$.  
The intersection of this path with $C$ consists of $L^\mp$-leaf segments joining two geodesic sides of $C$ and possibly $L^\pm$-leaf segments contained in geodesic sides of $C$.  Thus $\gamma\cap C$ induces a path in $\Gamma(C)$ connecting $\alpha$ with $\beta$.
\end{proof}

\begin{lemma} \label{lem:union_of_segments} 
Suppose that $C$ admits an ideal side $I$. Let $\{ s_i \}$ denote the set of ideal segments contained in  $I$. Then $\overline{\bigcup_i s_i} = I$.
\end{lemma} 

\begin{proof} 

Suppose for contradiction that there exists an open interval $U$ in $I\smallsetminus \cup_i s_i$. 
By the density property of FT prelaminations, there exist two leaves $\alpha_0$, $\alpha_1$  of $L^\mp$ with distinct  endpoints $a_0\neq a_1$ in $U$. By Lemma \ref{l.no-ideal-connection} the other endpoints of $\alpha_0$ and $\alpha_1$ are  contained in (possibly equal) geodesic sides $\beta_0$ and $\beta_1$ of $C$, respectively.

If $\beta_0=\beta_1$ then every $L^\mp$-leaf having an endpoint between $a_0$ and $a_1$ crosses $\beta_0$. Thus, the open interval $(a_0,a_1)\subset U$ is contained in an ideal segment, contradicting the hypothesis on $U$.

So we must have $\beta_0\neq \beta_1$.  Since $(L^+,L^-)$ is fully transverse, there is a polygonal path $\gamma_0$ from $\beta_0$ to $\beta_1$. By replacing all components of $\gamma_0\smallsetminus C$ by geodesic sides of $C$, we obtain a path $\gamma$ that stays in $C$ and has its (finitely many) sides alternatively on $L^\mp$-leaf segments and segments of geodesic boundary components of $C$.  This may not be a path in $L^+ \cup L^-$ because geodesics sides of $C$ are not necessarily leaves, however the path can be approximated arbitrarily well by paths in $L^+ \cup L^-$, and we may replace $\gamma$ by any sufficiently good approximation in the rest of the argument. 

Now every $L^\mp$-leaf having an endpoint between $a_0$ and $a_1$ crosses $\gamma$ at one of its finitely many sides on geodesic components of $C$.
Thus there is a non-empty open interval $V\subset (a_0,a_1)$ such that a dense subset of $V$ consists of endpoints of $L^\mp$-leaves crossing the same geodesic boundary component of $C$. In other words, $V\subset U$ is contained in an ideal segment, contradicting the hypothesis on $U$.
\end{proof}

We finish this introductory subsection with an additional lemma about valence $1$ vertices.
\begin{lemma}\label{lem:valence1}
Suppose $\alpha$ is a geodesic side of $C$ of valence $1$ in $\graph(C)$, and let $e$ be the edge starting at $\alpha$.
If the other vertex of $e$ is a geodesic side $\alpha_1$ (resp.~an ideal segment $s$) then every $L^\mp$-leaf 
crossing $\alpha$ crosses $\alpha_1$ (resp. has an endpoint in the interior of $s$).
\end{lemma}
  
\begin{proof}
Suppose $\beta \in L^-$  
crosses $\alpha$.   Since $\alpha$ has valence 1, either $\beta$ crosses $\alpha_1$ or ends in the interior of $s$, depending on the case (which is what we wanted to show) or $\beta$ meets $\partial C$ at an ideal point between two (ideal or geodesic) sides (by Lemma \ref{lem:union_of_segments}).  In the latter case, consider the connected component of $C \smallsetminus \beta$ that does not contain $e$.  No side of $C$ in this component can be connected to $\alpha$ (because $\alpha$ has valence 1), nor can an edge cross $\beta$.  This shows that $\graph(C)$ is disconnected, contradicting Lemma \ref{lem:linkage_connected}.  
\end{proof}

\subsection{Adding leaves to prelaminations}
By definition, the planar completion of a prelamination is obtained by adding additional leaves.  
In this and the next section, we gather results towards this process, but with a spirit of more generality. In particular, we treat two questions: 
\begin{enumerate} 
\item What leaves can be added to a pair of FT prelaminations while preserving the prelamination and fully transverse properties?
\item Which leaves may be added while preserving the no high-valence or simple cycle conditions?   More generally, how does inserting an extra leaf into a complementary region change its linkage graph? 
\end{enumerate} 

As a first step towards the first question we have the following fact, which will also be used in the Appendix.  

\begin{lemma} \label{lem:extend_lamination}
Let $(L^+,L^-)$ be a pair of FT prelaminations.  Define 
\[ \bar{L}^\pm := \{ (a, a') \in S^1 \times S^1 : a \neq a', \,  \exists (a_n, a_n') \in L^\pm \text{ converging to } (a, a')\} \]
Then, for any subset $\hat{L}^\pm$ with $L^\pm \subseteq \hat L^\pm \subseteq \bar L^\pm$, the pair $(\hat L^+, \hat L^-)$ is a pair of FT prelaminations. 
\end{lemma} 

\begin{proof}
We first show that $\bar{L}^\pm$ are both prelaminations (in fact actual laminations, since they are closed), and are transverse. 

Suppose $(a_n, a_n') \in L^+$ converges to $(a, a')$ and $a \neq a'$.  (The case for $L^-$ is the same).  We need to show that $(a, a')$ does not cross any leaf or limit of leaves of $L^+$, and that it does not agree with a leaf or limit of leaves of $L^-$.   For the first point, if $(a, a')$ crossed a limit, then it would cross a leaf $\alpha$, and thus $\alpha$ would cross $(a_n, a_n')$ for sufficiently large $n$, which is impossible.   For the second point, note that density and connectedness of fully transverse prelaminations implies that there must exist a leaf of either $L^+$ or $L^-$ crossing $(a, a')$, and as we just observed this is necessarily a leaf of $L^-$.  Thus, $(a, a')$ cannot be a limit of leaves in $L^-$. 

Now we show the other properties of FT prelaminations hold.  Density of the pair $(\hat L^-,\hat L^+)$ follows  from the density of $(L^-,L^+)$, as $L^\pm\subset \hat L^\pm $.
For the connectedness of $(\hat L^-,\hat L^+)$, note that the argument above implies that every leaf of $\hat L^\pm$  crosses a leaf of $L^\mp$.  With this, the connectedness of the pair $(\hat L^-,\hat L^+)$ follows from the connectedness of the pair $(L^-,L^+)$.
\end{proof} 

Adding leaves from $\bar L^\pm$ to $L^\pm$ changes complementary regions and linkage graphs in a predictable way:  
\begin{lemma} \label{lem:closure_same_graph}
Let $(L^+,L^-)$ be a pair of FT prelaminations.  Then 
\begin{enumerate}[label=(\roman*)]
\item the closure $\bar L^\pm$ has no trivial complementary regions
\item For any $L^\pm \subseteq \hat L^\pm \subseteq \bar L^\pm$, we have that $C$ is a nontrivial complementary region of $\hat{L}^\pm$ if  and only if it is a nontrivial complementary region of $L^\pm$. Moreover, $\graph(C; L^\mp)\cong\graph(C; \hat L^\mp)$. 
\end{enumerate}

In particular, $v$ is a high-valence side for $\graph(C; L^\mp)$ if and only if it is a high-valence side for $\graph(C; \hat L^\mp)$, and $(L^+,L^-)$ satisfies the simple cycle condition if and only if $(\hat L^+,\hat L^-)$ does too.
\end{lemma}

Note that prelaminations $\hat L^\pm \subseteq \bar L ^\pm$ may contain high-valence leaves even when $L^\pm$ does not: For instance, if $(L^+,L^-)$ is induced from a pA-bifoliation with some nonseparated leaves, then the closures $\bar L ^\pm$ will contain the ``roots" of the one-root complementary regions, which are high-valence.  

\begin{proof}[Proof of Lemma \ref{lem:closure_same_graph}]
A trivial complementary region of a prelamination is a geodesic $\alpha$ that is accumulated on both sides by leaves. Since $\bar L^\pm$ is closed, it has no trivial complementary regions. 

If $C$ is a nontrivial connected component of $L^\pm$, its boundary in the open disc is made up of leaves of $\bar L^\pm$, so it is also a complementary region of $\hat L^\pm$ (Recall that we defined complementary regions to be the \emph{closure} of connected components of the complement.)
We now argue that $\graph(C;\hat L^\mp) \cong \graph(C;L^\mp)$. Suppose a leaf $\alpha$ of $\hat L^\mp \smallsetminus L^\mp$ represents an edge of  $\graph(C;\hat L^\mp)$, thus crosses at least one geodesic side and perhaps meets an ideal segment of $C$.  Since crossing and meeting ideal segments (which were defined in terms of open intervals) are open conditions, this means that some $L^\mp$ leaves accumulating on $\alpha$ represents the same side.  This shows 
$\graph(C;\hat L^\mp) \subseteq \graph(C;L^\mp)$ and the other inclusion is immediate. 
\end{proof}

The other possible leaves that could be added to a prelamination are geodesics contained inside nontrivial complementary regions. Any nontrivial complementary region $C$ of $L^\pm$ is the convex hull of its extremal points.  Any geodesic $\alpha$ between two extremal points of $C$ can be added as a leaf to $L^\pm$ in such a way so that $\alpha\cup L^\pm$ is still a prelamination, and the fully transverse property is still preserved. 

However, adding such a leaf will change the linkage graphs, since it divides one complementary region in two.  More importantly, it may change the properties of linkage graphs that we care about, such as the simple cycle condition and high-valence leaves. In the next section, we introduce a natural class of geodesics called \emph{crossing geodesics} that can be added to FT prelaminations without adversely affecting these properties.   We then use the crossing geodesics in our construction of a completion in Section \ref{sec:proof_completion_thm}.  


\subsection{Crossing geodesics}

In this section again, $L^\pm$ are FT prelaminations and $C$ is a fixed nontrivial complementary region of $L^\pm$.

Since the complementary region $C$ we consider has non-empty interior, its boundary $\partial C$ is a simple closed curve. The vertices correspond to intervals on this curve, with non-empty, pairwise disjoint interiors. Thus these vertices inherit a natural cyclic order.
If we select a vertex $v_0$ in $\graph(C)$, the vertices in $\graph(C)\smallsetminus\{v_0\}$ all lie in a single component of $\partial C \smallsetminus\{v_0\}$ and therefore inherit a {\em linear order}; which also gives a linear order on $(\partial C \smallsetminus\{v_0\}) \cap S^1$. Thus, if $\Gamma \subset \graph(C)\smallsetminus\{v_0\}$ is a subgraph, we get an induced linear order on the subset of $S^1$ corresponding to endpoints of vertices in $\Gamma$.   Abusing notation, we denote this set by $\Gamma \cap S^1$.  

\begin{definition}\label{d.extrema} 
Let $v_0$ be a vertex of $\graph(C)$ and let $\Gamma\subset \left( \graph(C)\smallsetminus\{v_0\} \right)$ be a subgraph of $\graph(C)$.
The {\em extremal points of $\Gamma$ with respect to $v_0$} are the unique pair of points in $S^1$ obtained as the infimum and supremum of $\Gamma \cap S^1$ (as defined above) with respect to the inherited linear order.   
\end{definition}

When $e$ is an edge of $\graph(C)$ which is not in a cycle, then $\graph(C)\smallsetminus \{e\}$ has exactly two connected components, possibly reduced to a vertex. Similarly, if $e_1,e_2$ are two successive edges in a same cycle (that is, $e_1$ and $e_2$ are in the same cycle and share a vertex $v_0$) and neither $e_1$ nor $e_2$ are part of a different cycle, then $\graph(C)\smallsetminus\{e_1,e_2\}$ has exactly two connected components. Thus, 
we make the following definition
\begin{definition}
We call an edge $e$ not part of a cycle a \emph{disconnecting edge}\footnote{There is a similar notion of \emph{cut-edge} in graph theory, which is defined as an edge that separates the graph into two \emph{nontrivial} subgraphs. Thus a cut-edge is a disconnecting edge, but not vice-versa.}. Similarly, we call a pair of edges $(e_1,e_2)$ a \emph{disconnecting pair} if they share a vertex and are part of a unique cycle containing both $e_1$ and $e_2$. 
\end{definition}

We are now ready to define crossing geodesics. \footnote{
We could have also defined crossing geodesics for more general unions of edges that separates a linkage graph in exactly two connected components. However, we will not need to use a more general setting in our construction of a completion.}

\begin{definition}[Crossing geodesic]\label{d.crossing-geodesic}
Let $e$, respectively $(e_1,e_2)$, be a disconnecting edge and pair. Let $v_0$ be any vertex of $e$ or the shared vertex of $(e_1,e_2)$. 
Let $\Gamma_1$ be the connected component of $\graph(C)\smallsetminus \{e\}$ (resp.~$\graph(C)\smallsetminus\{e_1,e_2\}$) \emph{not} containing $v_0$.

The geodesic joining the two extremal points of $\Gamma_1$ with respect to $v_0$ is called \emph{the crossing geodesic} associated to the disconnecting edge $e$ (resp.~the disconnecting pair $(e_1,e_2)$) or simply \emph{$e$-crossing geodesic} (resp.~\emph{$(e_1,e_2)$-crossing geodesic}), and it is denoted by 
 $\alpha(e)$ (resp.~$\alpha(e_1,e_2)$).
\end{definition}

Some examples are shown in Figure \ref{fig:crossing}. 

\begin{rem}\label{rem_crossing_geodesics_dont_change_with_closure}
By Lemma \ref{lem:closure_same_graph}, the crossing geodesics of a pair of fully transverse prelamination $(L^+,L^-)$ are the same as the crossing geodesics of the prelaminations $(\hat L^+,\hat L^-)$ for any $L^\pm \subseteq \hat L^\pm \subseteq \bar L^\pm$.
\end{rem}

While the vertex $v_0$ in the above definition is uniquely defined for disconnecting pairs, it is not for disconnecting edges. However, the lemma below shows that the definition of crossing geodesic is independent of that choice, hence the notation $\alpha(e)$ makes sense.

\begin{lemma}\label{l.extrema-edge}
Let $C$ be a  complementary region of $L^\pm$ and $e$ be a disconnecting edge of $\graph(C)$, joining the vertices $v_0$ and $v_1$.
Let $\Gamma_0, \Gamma_1\subset \left( \graph(C)\smallsetminus\{e\} \right)$ be the subgraphs of $\graph(C)$ obtained by removing the edge  $e$, and containing respectively $v_0$ and $v_1$.

Then the  extremal points with respect to $v_0$, of $\Gamma_1$ coincide with the extremal points with respect to $v_1$ of $\Gamma_0$. 
\end{lemma}
\begin{proof}
As $\Gamma_0$ and $\Gamma_1$ are connected disjoint graphs, their vertices are contained in intervals of $\partial C$ with disjoint interiors. As $\graph(C)$ is connected,  $\Gamma_0\cup\Gamma_1$ contains every vertices of $\graph(C)$. Thus the union of these intervals is dense in $\partial C$, which means that these intervals have the same extremal points, which is precisely what we wanted.
\end{proof}

Crossing geodesics have the following simple characterization:
\begin{lemma}\label{l.crossing-geodesic} 
Let $e$ be a disconnecting edge of $\graph(C)$.
Then $\alpha(e)$  is the unique geodesic contained in $C$ and  satisfying the following property:
An $L^\mp$-leaf crosses $\alpha(e)$ if and only if it represents the edge $e$.

Similarly, if $(e_1,e_2)$ is a disconnecting-pair of $\graph(C)$, then  $\alpha(e_1,e_2)$ is the unique geodesic contained in $C$ and  satisfying the following property:
An $L^\mp$-leaf crosses $\alpha(e_1,e_2)$ if and only if it represents either $e_1$ or $e_2$.

Moreover, the same holds if $L^\pm$ is replaced with $\hat{L}^\pm$ if $L^\pm \subseteq \hat{L}^\pm \subseteq \bar{L}^\pm$. 
\end{lemma}
\begin{proof}
We do the proof for $\alpha(e)$, the case of a disconnecting-pair being essentially identical we leave it to the reader.

By definition, the crossing geodesic $\alpha(e)$ joins two extremal points of the convex region $C$ and thus is contained in $C$. 
First, we show the desired property holds.  Let $a, b$ be the endpoints of $\alpha(e)$.   By Lemma \ref{l.extrema-edge}, these are the extremal points of each component of $\Gamma(C) \smallsetminus \{e\}$. 
Suppose $\beta$ is a  leaf crossing $\alpha(e)$.  Then the endpoints of $\beta$ are in opposite connected components of $S^1 \smallsetminus \{a, b\}$.  By Lemma~\ref{l.no-ideal-connection}, $\beta$ crosses at least one geodesic side of $C$.  If it crosses also another geodesic side, or has endpoint in the interior of an ideal segment of $C$, then we are done: $\beta$ represents an edge with one vertex on either side of $S^1 \smallsetminus \{a, b\}$, so that edge must be $e$ since $e$ is disconnecting. 

Otherwise, let $x$ denote the ideal endpoint of $\beta$ in $\partial C\cap S^1$.    Let $I$ denote the open interval of $S^1$ bounded by the endpoints of the vertex of $e$ that lies on the same side of $\alpha(e)$ as $x$; up to reversing orientation and labeling $a, b$ we assume $I \subset (a, x)$ and $(x, b) \cap I = \emptyset$.  
There is necessarily another vertex of $C$ in $(x, b)$, which by definition of crossing edge, is connected to $I$ by a path.  Thus, some leaf of $L^\mp$ crosses $\beta$, which is a contradiction.  Thus, we have shown that $\beta$ represents $e$.

The fact that all geodesics representing $e$ cross $\alpha(e)$ is immediate from the definition.   For uniqueness, 
suppose that $\alpha$ is a geodesic contained in $C$ and crossing precisely the $L^\mp$-leaves representing $e$. If $\alpha$ has an endpoint $x$ in an ideal segment of $C$ then $x$ is accumulated on both sides by endpoints of $L^\mp$-leaves crossing the same geodesic side. 
Thus either $\alpha$ crosses these leaves (contradicting the assumption), or $\alpha$ crosses the same geodesic side contradicting again the assumptions.
Thus the endpoints of $\alpha$ are points in the ideal boundary of $C$, and not in the interiors of  ideal segments. These endpoints therefore partition the vertices of $\graph(C)$ into two sets,  which can be connected only by edges crossing $\alpha$, that is, by $e$.  By definition, this implies that $\alpha = \alpha(e)$. 

Finally, the same characterization holds for $\hat L^\pm$ by Remark \ref{rem_crossing_geodesics_dont_change_with_closure}.\qedhere
\end{proof}

Adding a single crossing geodesic to a prelamination obviously still gives a prelamination, since the crossing geodesic is contained in a complementary region.  But the situation is better than that: our next lemma shows that distinct crossing geodesics never intersect, hence they can \emph{all} be added simultaneously.

\begin{lemma} \label{lem_crossing_geodesics_dont_intersect}
Let $\alpha_0$ and $\alpha_1$ be crossing geodesics in $C$.
Then $\alpha_0$ and $\alpha_1$ do not cross each other.
\end{lemma}
\begin{proof}
We first consider the case where at least one of the crossing geodesics is associated to a disconnecting edge. That is, we suppose that $\alpha_0=\alpha(e_0)$ with $e_0$ a disconnecting edge. 
Since the definition of $\alpha(e)$ is independent of vertex chosen (Lemma \ref{l.extrema-edge}), we may choose vertices $v_i$, as in the definition of $\alpha_i$, such that $v_1 \neq v_0$.  Thus, $v_1$ is contained in some connected component (say $\Gamma_1$) of $\graph(C)\smallsetminus \{v_0\}$; and hence one component of $\graph(C) \smallsetminus \{e_1\}$ is a further subset. Therefore, its extremal points are contained in one connected component of $S^1 \smallsetminus \alpha_1$.  This means $\alpha_0$ and $\alpha_1$ do not cross. 

Now, suppose we are given disconnecting pairs $(e_0,e_1)$ sharing vertex $v_0$, and $(e_2,e_3)$ sharing vertex $v_1$, such that $\alpha_0= \alpha(e_0,e_1)$ and $\alpha_1=\alpha(e_2,e_3)$.
If $v_0 \neq v_1$, we may run the same argument as above, and deduce that the $\alpha_i$ do not cross.
Suppose now that $v_0 = v_1$. By definition of disconnecting pair, the edges $e_0$, $e_1$, $e_2$ and $e_3$ are all distinct. Thus, the other vertices of $(e_0,e_1)$ and of $(e_2,e_3)$ are in \emph{distinct} connected components of  $\graph(C) \smallsetminus \{v_0\}$. Call these $\Gamma_0$ and $\Gamma_1$ respectively. From this, it follows that the extremal points of $\Gamma_0$ do not separate the extremal points of $\Gamma_1$, and thus $\alpha_0$ and $\alpha_1$ do not cross.
\end{proof}

\begin{corollary} \label{cor:completion_FT}
Let $(L^+,L^-)$ be a pair of fully transverse prelaminations and suppose that $\hat L^\pm$ is obtained from $L^\pm$ by adding any collection of crossing geodesics as well as any collection of leaves in the closure $\bar L^\pm$.
Then $(\hat L^+,\hat L^-)$ is a pair of FT prelaminations.
\end{corollary}

\begin{proof}
Lemma \ref{lem_crossing_geodesics_dont_intersect} together with Lemma \ref{lem:extend_lamination} shows that $(\hat L^+,\hat L^-)$ are prelaminations. Moreover, a crossing geodesic for a complementary region of, say, $\hat L^+$ cannot correspond to a leaf of $\hat L^-$ by Lemma \ref{l.no-ideal-connection}. Therefore, $(\hat L^+,\hat L^-)$ are disjoint prelaminations. Then the connectedness and density properties follows as before from the connectedness and density properties of $(L^+,L^-)$.
\end{proof}

Now, we finally show that crossing geodesics behave well with respect to linkage graphs.  More precisely, we show that  if $\hat L^\pm$ is obtained from $L^\pm$ by adding crossing geodesics, then the linkage graphs of complementary regions of $L^\mp$ do not change (Lemma \ref{lem:same_graphs}), and the linkage graphs of complementary regions of $\hat L^\pm$ are either unchanged or obtained from linkage graphs of $L^\pm$ by ``splitting them in two'', in a way made precise in Lemmas \ref{lem_low_valence_crossing} and \ref{lem:divide_graph}.

\begin{lemma}  \label{lem:same_graphs}
Suppose $(L^+, L^-)$ is a pair of FT prelaminations and suppose $\hat{L}^+$ is obtained from $L^+$ by adding (any collection of) crossing geodesics.

Let $C$ be a (nontrivial) complementary region of $L^-$, then $\graph(C;\hat L^+) \cong \graph(C;L^+)$. More precisely, the vertices of $\graph(C;\hat L^+)$ are vertices of $\graph(C;L^+)$ (with the same cyclic order) and  each edge of $\graph(C;\hat L^+)$ is an edge of $\graph(C;L^+)$.
\end{lemma}

\begin{proof} 
Let $C = C^-$ be a (nontrivial) complementary region of $L^-$.  
Let $\alpha$ be a leaf of $\hat{L}^+ \smallsetminus L^+$ that intersects $C^-$, i.e., $\alpha$ is a crossing geodesic for some region $C^+$ of $L^+$.   We need to show $\alpha$ does not define a new edge.  
Remark~\ref{rem_no-ideal-connection} implies that $\alpha$ crosses at least one geodesic side of $C^-$. 

\noindent {\bf Case 1: $\alpha$ crosses only one geodesic side of $C$.}
In this case, the other vertex of edge $\alpha$ is a point $x$ on the ideal boundary of $C$.  If $x$ does not belong to the interior of an ideal segment, $\alpha$ does not define an edge.  If $x$ belongs to the interior of an ideal segment of $C^-$, then $x$ is accumulated on both sides by $L^+$-leaves crossing the same geodesic side of $C^-$ and hence $\alpha$ crosses the same geodesic side thus defines an already existing edge of $\graph(C^-;L^+)$.

\noindent {\bf Case 2: $\alpha$ crosses two geodesic sides $\beta_0,\beta_1$ of $C^-$, and $\alpha= \alpha(e)$.}
By Lemma~\ref{l.crossing-geodesic}, $\beta_0$ and $\beta_1$ either represent the edge $e$, or are accumulated by leaves that represent the edge $e$ (depending on whether they are in $L^-$ or not), in $\Gamma(C^+, L^-)$.
At least one vertex of $e$ is a geodesic segment $\gamma$, which by definition of $e$ must intersect both $\beta_0$ and $\beta_1$. If $v_0$ corresponds to a leaf of $L^+$ we are already done, otherwise, it is accumulated by leaves of $L^+$, and so some leaf of $L^+$ crosses $\beta_0,\beta_1$, as desired.

\noindent {\bf Case 3: $\alpha$ crosses two geodesic sides $\beta_0,\beta_1$ of $C^-$, and $\alpha= \alpha(e_1, e_2)$.}
By  Lemma~\ref{l.crossing-geodesic}, $\beta_0$ and $\beta_1$ represent (or are accumulated by leaves that represent) $e_1$ or $e_2$, which share a vertex $v_0$ of a cycle in $C^+$.  By  Lemma~\ref{l.no-ideal-connection}, ideal vertices have valence 1 so vertices of a cycle (in particular $v_0$) are geodesic sides.  By definition $v_0$ crosses $\beta_0$ and $\beta_1$.  As in the previous case, if $\gamma$ is a leaf of $L^+$ we are already done, otherwise, it is accumulated by leaves of $L^+$, and so some leaf of $L^+$ crosses $\beta_0,\beta_1$ as desired.  
\end{proof}

Collecting the results above, we summarize how adding a crossing geodesic to $L^\pm$ changes the linkage graphs of its complementary regions, depending on the type or valence of the vertices contained in the edge or edge pair.
In what follows, $C$ always denotes a nontrivial complementary region to $L^+$.  

   \begin{figure}[t]
     \centering
     \includegraphics[width=12cm]{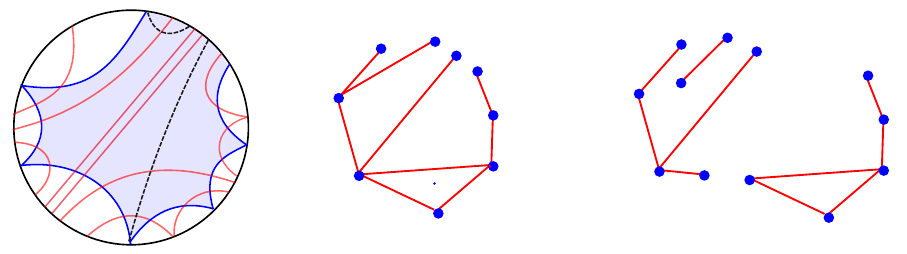}
     \caption{Crossing edges (dotted lines) associated to a valence 1 cut-edge, and to a cut-pair.  The original linkage graph (before adding crossing edges), and resulting three graphs (after) are shown.}
     \label{fig:crossing}
   \end{figure}

\begin{lemma}[Splitting $C$ by crossing an edge with ideal vertex] \label{lem:splitting_regions}
Let $e$ be a disconnecting edge of $\graph(C;L^-)$ between an ideal vertex $s$ and another vertex $v_1$.
Then the crossing geodesic $\alpha(e)$ is the geodesic between the two endpoints of $s$, so $C = C_s \cup \alpha(e) \cup C_1$ where $C_s, C_1$ are complementary regions of $L^+ \cup \{\alpha(e)\}$. 

Moreover, $\graph(C_s;L^-)$ is a degenerate graph consisting of one ideal vertex $s$ and one geodesic vertex $\alpha(e)$ joined by an edge, while $\graph(C_1;L^-)$ is isomorphic to $\graph(C;L^-)$ except that the ideal vertex corresponding to $s$ in $\graph(C;L^-)$ is replaced by the geodesic vertex corresponding to $\alpha(e)$. 
\end{lemma}

\begin{rem}
In the case above of an edge $e$ ending at an ideal segment $s$, while the valence of $\alpha(e)$ is $1$ in both $\graph(C_s;L^-)$ and $\graph(C_1;L^-)$, the leaf $\alpha(e)$ \emph{is} a high valence leaf of $L^+ \cup \{\alpha(e)\}$ since it satisfies condition \ref{item_high_val_degenerate} of Definition \ref{d.high-valence}. 
This is why we will not include such crossing geodesics in the completion of a pair of prelaminations (Definition \ref{def_completion}), so a reader might wish to skip this Lemma on a first pass. 
\end{rem}

\begin{proof}
The fact that $\alpha(e)$ is the geodesic between the endpoints of $s$ comes from the definition, see Lemma \ref{l.extrema-edge}. The description of the induced linkage graphs follows from the characterization of crossing geodesics given in Lemma \ref{l.crossing-geodesic}.
\end{proof}

The next lemma describes the somewhat degenerate case of a crossing geodesic for a vertex of very low valence.  We first observe an immediate consequence of Lemma \ref{l.crossing-geodesic}.  
\begin{observation} \label{obs:easy}
Let $e$ be a disconnecting edge of $\graph(C;L^-)$ between geodesic vertices $v_0,v_1$, such that $v_0$ has valence $1$. 
Then the crossing geodesic $\alpha(e)$ is the geodesic side $v_0$. 

Similarly, if $(e_1,e_2)$ is a disconnecting pair of $\graph(C;L^-)$ such that the shared vertex $v_0$ has valence $2$, then the crossing geodesic $\alpha(e_1,e_2)$ is the geodesic side $v_0$.
\end{observation}

As a further consequence we have: 
\begin{lemma}[Crossing low valence geodesic vertices]\label{lem_low_valence_crossing}
In the setting of Observation \ref{obs:easy}, $C$ is also a complementary region of  $L^+ \cup \{\alpha(e)\}$ (respectively, of $L^+ \cup \{\alpha(e_1,e_2)\}$), the linkage graph of $C$ is left unchanged, and $\alpha(e)$ (resp. $\alpha(e_1,e_2)$) has valence $1$ in both graphs.  

Moreover, $\alpha(e)$ is a high-valence leaf for $(L^+ \cup \{\alpha(e)\},L^-)$ if and only if $\alpha(e)\in L^+$ and it is a high-valence leaf of $(L^+,L^-)$\footnote{Necessarily, the only possibility in this case is that $\alpha(e)$ would satisfy condition \ref{item_high_val_degenerate} of Definition \ref{d.linkage-graph}.}.

Similarly, $\alpha(e_1,e_2)$ is a high-valence leaf for $(L^+ \cup \{\alpha(e_1,e_2)\},L^-)$ if and only if $\alpha(e_1,e_2)\in L^+$ is a high-valence leaf of $(L^+,L^-)$\footnote{Necessarily, the only possibility in this case is that $\alpha(e_1,e_2)$ would satisfy condition \ref{item_high_val_two_polygons} of Definition \ref{d.linkage-graph}.}
\end{lemma}

\begin{proof}
Observation \ref{obs:easy} implies that $C$ and its linkage graph is unchanged, since $\alpha(e)$ or $\alpha(e_1,e_2)$ is a side of the complementary region $C$.  
Thus, the crossing geodesic $\alpha(e)$ or $\alpha(e_1,e_2)$ cannot satisfy conditions \ref{item_high_val_cycle_greater_3} or \ref{item_high_val_greater_2} of Definition \ref{d.linkage-graph}, and thus is a high-valence leaf if and only if it satisfies either condition \ref{item_high_val_degenerate} or \ref{item_high_val_two_polygons} of Definition \ref{d.high-valence}. In particular, this means that it must bound complementary regions of $L^+$ on both sides, and hence must be an element of $L^+$ (which is necessarily a high-valence leaf).
\end{proof}

\begin{lemma}[Splitting $C$ by a crossing edge, general case]
\label{lem:divide_graph}
Let $\alpha$ be a crossing geodesic associated to a disconnecting-pair $(e_1,e_2)$ with shared vertex $v_0$ of valence $>2$, or to a disconnecting edge with vertices $v_0,v_1$ both of valence $>1$.

Then $C = C' \cup \alpha \cup C''$, with $C', C''$ complementary regions of $L^+ \cup \{\alpha\}$, and we have:
\begin{itemize}
\item If $\alpha= \alpha(e)$, then $\graph(C')$ is obtained from $\graph(C)$ by deleting one connected component of $\graph(C) \smallsetminus e$, and relabelling the vertex $v_i$ adjoining the deleted component by $\alpha$.  Similarly, $\graph(C'')$ is obtained from $\graph(C)$ by applying this procedure to the other connected component. 
\item If $\alpha= \alpha(e_1, e_2)$, where $(e_1,e_2)$ share the vertex $v$, then one of $\graph(C')$, $\graph(C'')$ contains $e_1$ and $e_2$ and is obtained by deleting the other connected component of $\graph(C) \smallsetminus \{v\}$ and relabelling $v$ by $\alpha$.  The other is obtained by deleting the connected component of $\graph(C) \smallsetminus (e_1 \cup e_2)$ containing the rest of the cycle, attaching an edge to $v$ and labelling its other vertex by $\alpha$.   
\end{itemize}
\end{lemma} 

\begin{proof}
In the situation of the lemma, neither connected component of the graph $\graph(C)\smallsetminus\{e\}$ (or $\graph(C)\smallsetminus\{e_1,e_2\}$) is trivial, so one deduces from Lemma \ref{l.crossing-geodesic} that $\alpha$ is contained in the interior of $C$. The description of the associated graphs of  $L^+ \cup \{\alpha\}$ also follow from Lemma \ref{l.crossing-geodesic}, which specifies exactly the edges crossed by $\alpha(e)$ or $\alpha(e_1, e_2)$.  
\end{proof}

The results above make clear that the impact of adding crossing geodesics on linkage graphs is different depending on whether the disconnecting edge or pair splits the graph into two nontrivial component or not. For ease of reference, we introduce the following terminology.
\begin{definition}\label{def_cut_edge_pair}
A \emph{cut-edge} $e$ of $\graph(C)$ is a disconnecting edge such that $\graph(C)\smallsetminus \{e\}$ has two \emph{nontrivial} components\footnote{This corresponds to the usual notion of cut-edge in graph theory.}.
A \emph{cut-pair} $(e_1,e_2)$ of $\graph(C)$ is a disconnecting pair such that $\graph(C)\smallsetminus \{e_1,e_2\}$ has two \emph{nontrivial} components.
\end{definition}

We note the following useful direct consequence of the three lemmas above. 
\begin{corollary} \label{cor:crossing_is_valence_1}
If $\alpha$ is a crossing geodesic associated to a cut pair or cut edge, and lies in the boundary of complementary regions $C'$ and $C''$, then $\alpha$ has valence 1 in at least one of $\graph(C')$ or $\graph(C'')$. 
\end{corollary}

\section{Proof of Theorem \ref{thm_completion}: Existence}\label{sec:proof_completion_thm}

Recall that a {\em planar completion} of $(L^+, L^-)$ is a pA-bifoliation  $(\cF^+, \cF^-)$ such that $L^+$ and $L^-$ are dense in the induced prelamination of $(\cF^+, \cF^-)$, meaning that the set of leaves of $\cF^\pm$ whose ends lie in $L^\pm$ are dense in the plane. (Note that density in the plane and density of pairs of endpoints on the circle are not equivalent: a leaf of a foliation that is nonseparated from others on both sides is isolated in the induced lamination.)

In this section we use the tools developed above to prove the existence of a planar completion under the assumptions of Theorem \ref{thm_completion}. 

In the rest of this section, we fix a pair of FT prelaminations $(L^+,L^-)$ satisfying the conditions of Theorem \ref{thm_completion}, namely: the simple cycle condition, no high-valence leaves, and if three $L^\pm$-leaves cross a common $L^\mp$-leaf, then they do not have a common endpoint.

Notice that the simple cycle condition implies that every edge in a linkage graph is either a disconnecting edge (if it is not included in a cycle) or part of a disconnecting pair (if it is part of a cycle).

\begin{definition}\label{def_completion}
The \emph{completion} of $(L^+,L^-)$ is the pair of laminations $(\tilde L^+,\tilde L^-)$ defined as follows: 
$\tilde L^\pm$ is the union of $L^\pm$ with
\begin{itemize}
\item all the geodesics which correspond to trivial complementary regions;
\item all the crossing geodesics $\alpha^\pm(e_1,e_2)$ associated to a disconnecting pair of $L^\pm$; and 
 \item all the crossing geodesics $\alpha^\pm(e)$ where $e$ is a disconnecting edge whose vertices are {\em both} geodesics.
\end{itemize}
\end{definition}

Note that the completion depends on the pair $(L^+, L^-)$, it does not make sense to ``complete" $L^+$ without reference to $L^-$.  

The main work in this section will be to show that the completion satisfies the conditions of Theorem \ref{thm:realization}. Most of this will follow from the preparatory work we did in section \ref{sec_preliminaries}.

We first give an equivalent characterization of leaves that are added in the completion:
\begin{observation}\label{obs_types_leaves_completion}
Let $\alpha\in \tilde L^\pm\smallsetminus L^\pm$. Then $\alpha$ is of one of three possible, mutually exclusive, types:
\begin{enumerate}
\item\label{item_added_trivial_complementary_regions}  $\alpha$ is a geodesic corresponding to a trivial complementary region of $L^\pm$;
\item \label{item_added_boundary_geodesic}  $\alpha$ is a {\em not high-valence} geodesic boundary component of a nontrivial complementary region $C$ of $L^\pm$ 
\item \label{item_added_crossing_geodesic_cutedge}
 there is a nontrivial complementary region $C$ of $L^\pm$ such that $\alpha$ is the crossing geodesic of a \emph{cut-edge} $e$, or \emph{cut-pair} $(e_1,e_2)$ (see Definition \ref{def_cut_edge_pair}).
\end{enumerate}
Conversely, any geodesic of one of the type above must be a leaf of $\tilde L^\pm$.

Moreover, there are only countably many leaves of types \ref{item_added_boundary_geodesic} and \ref{item_added_crossing_geodesic_cutedge}.
\end{observation}
\begin{proof}
By definition, if $\alpha$ is not associated to a trivial complementary region of $L^\pm$, then it is a crossing geodesic of some disconnecting edge or pair in $\gamma(C)$, for a nontrivial complementary region $C$. Then either Lemma \ref{lem_low_valence_crossing} applies, so that $\alpha$ is in case \ref{item_added_boundary_geodesic}, or Lemma \ref{lem:divide_graph} applies so $\alpha$ is in case \ref{item_added_crossing_geodesic_cutedge}.
The converse is also immediate given the definition of completion and Lemmas \ref{lem_low_valence_crossing} and \ref{lem:divide_graph}. Finally, since there are only countably many nontrivial complementary regions, the last claim follows.
\end{proof}

\begin{rem} \label{rem:equiv_def_completion}
From the observation above, we see that there is an equivalent way of defining the completion:
The completion, $(\tilde L^+,\tilde L^-)$, of $(L^+,L^-)$ is obtained by first taking all the leaves in the closure $\bar L^\pm$ which do not represent high valence vertices of complementary regions of $L^\pm$; 
and then adding all the crossing geodesics associated to cut-edges or cut-pairs of the linkage graphs $\Gamma(C^\pm;L^\mp)$.
\end{rem}

Leaves in $\tilde L^\pm$ of type \ref{item_added_trivial_complementary_regions} or \ref{item_added_boundary_geodesic} are not high-valence, and, by Corollary \ref{cor:crossing_is_valence_1}, neither are the leaves of type \ref{item_added_crossing_geodesic_cutedge}. Hence, from Observation \ref{obs_types_leaves_completion}, we also deduce
\begin{corollary}\label{cor_no_high_valence_leaf_in_completion}
The completion $(\tilde L^+,\tilde L^-)$ of $(L^+,L^-)$ does not contain any high-valence leaves.
\end{corollary}

\begin{rem}
We emphasize that we \emph{do not} take crossing geodesics associated to edges with one ideal vertex in the completion since otherwise the above result would fail (see Lemma \ref{lem:splitting_regions}). 
\end{rem}

Lemma \ref{lem:divide_graph} shows that  adding crossing geodesics never creates new cut-edges or cut-pairs.  Thus, we immediately conclude:
\begin{corollary} \label{cor:completion_stabilizes}
Suppose $(\tilde L^+,\tilde L^-)$ is the completion of $(L^+, L^-)$.  Then 
the completion of $(\tilde L^+,\tilde L^-)$ is equal to $(\tilde L^+,\tilde L^-)$.
\end{corollary}

As stated above, our main work in this section will be to prove the following: 
\begin{proposition}\label{p.completion}Let $(\tilde L^+,\tilde L^-)$ be the completion of $(L^+,L^-)$.  Then: 
 \begin{enumerate}[label=(\roman*)]
\item  $(\tilde L^+, \tilde L^-)$ is fully transverse
 \item\label{item_completion_countable}  For any $a\in S^1$, the set $\{b : \{a,b\}\in \tilde L^-\cup \tilde L^+\}$ is countable
 \item\label{item_completion_one_root_ideal_polygon}  Each complementary region of $\tilde L^\pm$ is an ideal polygon or one-root region.  Ideal polygons come in coupled pairs, and no leaf lies in the boundary of two ideal polygons.
 \end{enumerate}
\end{proposition}

Before proving Proposition \ref{p.completion}, we first show that the proposition implies the existence of a {\em planar completion} for $(L^+,L^-)$.

\begin{proof}[Proof of Theorem \ref{thm_completion}, Existence]
Let $(\tilde L^+,\tilde L^-)$ be the completion of $(L^+,L^-)$. By Proposition \ref{p.completion}, $(\tilde L^+,\tilde L^-)$ satisfies the condition of Theorem \ref{thm:realization}, hence $(\tilde L^+,\tilde L^-)$ is induced by a pA-bifoliation $(\cF^+,\cF^-)$. We need to show that $(\cF^+,\cF^-)$ is a planar completion of $(L^+,L^-)$.
Since $L^\pm \subset \tilde L^\pm$, we only have to show that $L^\pm$ is dense in $\cF^\pm$.

Consider a non-singular leaf, or face of singular leaf, $\ell^\pm$ of $\cF^\pm$.  It corresponds to a $\tilde L^\pm$-leaf $\alpha$.  Any $\tilde L^\pm$-leaf crosses a $L^\mp$-leaf, so let $\beta$ be a $L^\mp$-leaf crossing $\alpha$ and let $\ell_\beta^\mp$ be the corresponding leaf or face.

The set $U_\beta$ of non-singular leaves or faces of leaves of $\cF^\pm$ crossing $\ell_\beta^\mp$ fills a neighborhood of $\ell$. By Observation \ref{obs_types_leaves_completion}, there are at most countably many leaves of $\tilde L^\pm$ in this neighborhood that are not in the closure $\bar L^\pm$.  
Thus a dense subset of leaves in $U_\beta$ define leaves of $L^\pm$, ending the proof of the existence of planar completion, assuming Proposition \ref{p.completion}.
\end{proof}

\subsection{Proof of Proposition \ref{p.completion}}

Let $(\tilde L^+,\tilde L^-)$ be the completion of $(L^+,L^-)$. By Corollary \ref{cor:completion_FT}, $(\tilde L^+,\tilde L^-)$ is fully transverse. We need to show items \ref{item_completion_countable}-\ref{item_completion_one_root_ideal_polygon} of Proposition \ref{p.completion}.

\begin{lemma}
For any $a\in S^1$, the set $\{b : \{a,b\}\in \tilde L^-\cup \tilde L^+\}$ is countable, i.e., item \ref{item_completion_countable} is satisfied.
\end{lemma}
\begin{proof} 
Suppose for contradiction there is $a\in S^1$ so that $\{b : \{a,b\}\in \tilde L^-\cup \tilde L^+\}$ is uncountable.  Up to switching $+$ and $-$, we can assume without loss of generality that there are uncountably many leaves of the form $\{a, b\}$ in $\tilde L^+$.  Since only countably many crossing leaves are added in the construction of $\tilde L^+$, we in fact have uncountably many leaves of the form $\{a, b\}$ in $\tilde L^+ \cap \bar L^+$.  
Each such leaf is crossed by a leaf of $L^-$, so there exists 
 $\beta\in \hat L^-$ such that $\{b : \{a,b\}\in \hat L^+ \text{ and }\{a,b\}\text{ crosses } \beta \}$ is still uncountable. 
In particular, $\beta$ intersects three leaves of $\hat L^+$ ending at the same point in $S^1$. Since every leaf of $\hat L^+$ is approximated on both sides by leaves of $L^+$, we deduce that $\beta$ also intersect three leaves of $L^+$ all ending at $b$ and that $\beta$ can be assumed to be in $L^-$. This contradicts the assumption on $(L^+,L^-)$.
\end{proof}

Now we have to show that complementary regions of $\tilde L^\pm$ satisfy item \ref{item_completion_one_root_ideal_polygon}. We split this in several lemmas.
The first one is a direct application of the work done in section \ref{sec_preliminaries}.
\begin{lemma} \label{lem:star_or_cycle}
Let $C$ be a complementary region of $\tilde L^\pm$, then its linkage graph $\graph(C;\tilde L^\mp)$ is either a star\footnote{Recall that a star is a connected graph with at most one vertex of valence greater than $1$.} or a cycle.
\end{lemma}

\begin{proof}
By Observation \ref{obs_types_leaves_completion} and Lemma \ref{lem:divide_graph}, $\graph(C)$ has no cut-edge or cut-pairs 
so it must either be a cycle, or a star.
\end{proof}

\begin{lemma}
Let $C$ be a complementary region of of $\tilde L^\pm$ such that $\graph(C)$ is a star.  Then $C$ is a one-root region. 
\end{lemma}

\begin{proof}
As a first case, suppose $C$ is degenerate in the sense that $\graph(C)$ is a single edge with two vertices. Then 
only one of these vertices, call it $v_0$, is geodesic (since a complementary region is a convex domain bounded by geodesics and segments of $S^1$) and it is high-valence (Definition \ref{d.high-valence}), so by Observation \ref{obs_types_leaves_completion} $v_0$ is not a leaf of $\tilde L^\pm$. 
Finally, every leaf of $\tilde L^\mp$ that intersects $C$ intersects $v_0$, by Lemma \ref{l.no-ideal-connection}.
So $C$ is a one-root region.

Now we assume that $\graph(C)$ is non-degenerate. Call $v_0$ the unique vertex in $\graph(C)$ of valence $\geq 2$. So $v_0$ is high-valence and thus is not a leaf of $\tilde L^\pm$ (by Corollary \ref{cor_no_high_valence_leaf_in_completion}). By Lemma \ref{lem:divide_graph} (or Observation \ref{obs_types_leaves_completion}), every other geodesic side of $\graph(C)$ is a leaf of $\tilde L^\pm$. 
Let $\beta$ be a leaf intersecting $C$.  Then $\beta$ meets at least one geodesic side of $C$ by Lemma \ref{l.no-ideal-connection}.  Either this is $v_0$, and we are done, or it meets some other geodesic side, which is valence 1 since $\graph(C)$ is a start.  In this case, Lemma \ref{lem:valence1} implies $\beta$ intersects $v_0$ as well, which is what we needed to show.  Thus, $C$ is a one-root region.
\end{proof}

\begin{lemma}
Let $C$ be a complementary region of $\tilde L^\pm$ such that $\graph(C)$ is a cycle.  Then $C$ is an ideal polygon and part of a coupled pair.
\end{lemma}

\begin{proof}
Since $\graph(C)$ is a cycle and ideal vertices have valence 1 (Lemma \ref{l.no-ideal-connection}), $\partial C$ has no ideal segments.  Thus, $C$ has only geodesic sides. 

Let $e_1, e_2, \ldots e_k$ denote the edges of $\graph(C)$ in cyclic order.  Each choice $\{\alpha_1, \ldots \alpha_k\}$ of geodesics representing the edges bounds a convex region of $\bD^2$.    The intersection of all such regions is again a convex region bounded by geodesic segments, which are either leaves of $\tilde L^\mp$ or accumulated by $\tilde L^\mp$ leaves representing an edge of $\graph(C)$.   Let $\beta_1, \ldots \beta_k$ denote these segments.  

Because $\graph(C)$ is a cycle, no other leaves $\tilde L^\mp$ cross $C$, so the $\beta_i$ form the sides of a complementary region $C^\mp$ of $\tilde L^\mp$; moreover they are cyclically ordered, and any two consecutive segments meet a common boundary component of $C$.  Thus, the $\beta_i$ are vertices of a cycle of $\Gamma(C^\mp)$.    
By Lemma \ref{lem:star_or_cycle}, $\Gamma(C^\mp)$ is a cycle.  
We conclude, as above that  $\partial C^\mp$ has no ideal segments.  
By Observation \ref{obs_types_leaves_completion}, this implies that the $\beta_i$ are in $\tilde L^\mp$, and thus the sides of an ideal polygon, which is necessarily coupled with $C$.    
\end{proof}

The next result ends the proof of Proposition \ref{p.completion}.
\begin{lemma} 
There are no adjacent ideal polygons.
\end{lemma} 

\begin{proof} 
Assume that $\alpha$ is a $\tilde L^\pm$ leaf which belongs to two ideal polygons.
 In particular, $\alpha$ bounds two complementary regions of $\tilde L^\pm$, so is either a leaf of $L^\pm$ or a crossing geodesic associated with a cut-edge or cut-pair (i.e., of type \ref{item_added_crossing_geodesic_cutedge} in Observation \ref{obs_types_leaves_completion}). In the latter case, Corollary \ref{cor:crossing_is_valence_1} implies that $\alpha$ must have valence $1$ in at least one of these two ideal polygons which is absurd.
So we must have $\alpha\in L^\pm$. However, $\alpha$ has valence $\geq 2$ in the linkage graphs on both sides for $\tilde L^\mp$, so by Lemma \ref{lem:same_graphs}, $\alpha$ also has valence $\geq 2$ in the linkage graphs for the adjacent complementary regions of $L^\mp$, contradicting the no high valence hypothesis on $(L^+, L^-)$.
\end{proof}

\section{Proof of Theorem \ref{thm_completion} -- Necessity and Uniqueness} \label{sec:necessity_uniqueness}
In this section we will prove both the necessity and the uniqueness parts of Theorem~\ref{thm_completion}, as they both follow from the same preliminary lemmas.  

We begin by noting an elementary observation

\begin{observation} \label{obs:endpoints}
Let $l_n$ be a sequence of leaves of $\cF$ and let $\{a_n, b_n\}$ be the endpoints of $l_n$.  Suppose that $\{a_n, b_n\} \to \{a, b\}$.  If $l$ is a limit of $l_n$, but the endpoints of $l$ is not equal to the set $\{a, b\}$, then $\{a, b\}$ and $l$ form two sides of a complementary region.  
\end{observation}

\begin{lemma} \label{lem:density}
Let $(\cF^+, \cF^-)$ be a planar completion of $(L^+,L^-)$. Let $(L^+_{\cF}, L^-_{\cF})$ be the prelaminations induced by $(\cF^+, \cF^-)$.
If $\alpha \in L^+_{\cF}$, then either $\alpha \in \bar{L}^+$, or $\alpha$ bounds a complementary region to $L^+_{\cF}$ on each side; at least one of which must be one-root. 
\end{lemma} 

\begin{proof}
By density of $L^+$ in the plane, there exists a sequence of leaves $\alpha_i$ of $\cF^+$ approaching $\alpha$ from one side, and a sequence $\beta_i$ approaching from the other.  If the endpoints of either sequence approach the endpoints of $\alpha$, then $\alpha \in \bar{L}^+$.  Otherwise, by Observation \ref{obs:endpoints} there is a complementary region of $L^+_{\cF}$ on each side of $\alpha$. 
Finally, the necessity part of Theorem \ref{thm:realization}, gives that no two ideal polygons share sides and hence at least one of the complementary regions bounded by $\alpha$ must be one-root.  
\end{proof}

\begin{lemma} \label{lem:edges_exist}
Let $(\cF^+, \cF^-)$ be a planar completion of $(L^+,L^-)$. Let $(L^+_{\cF}, L^-_{\cF})$ be the prelaminations induced by $(\cF^+, \cF^-)$.
If $C^+$ is a complementary region to $L^+_{\cF}$, then $\graph(C^+; L_{\cF}^-) \cong \graph(C^+; L^-)$.  In other words, each edge of $\graph(C^+, L_{\cF}^-)$ is represented by a leaf of $L^-$.
\end{lemma} 

\begin{proof}
Suppose $\alpha \in L^-_{\cF}$ represents an edge of the linkage graph of $C^+$ between sides $s_1, s_2$ of $C^+$. We need to show some leaf of $L^-$ also intersects $s_1$ and $s_2$. 
 If $\alpha$ lies in $L^-$ or is accumulated upon by leaves of $L^-$ (in particular, if it is a boundary of at most one complementary region of $L^-$) then we are done.   Otherwise, by Lemma \ref{lem:density} $\alpha$ bounds two complementary regions of $L_{\cF}^-$ and one is a one-root region.  
If $s_1, s_2$ are leaves, then because they intersect $\alpha$, they must also intersect the root of the one-root region.  If either is not a leaf, then they are accumulated by leaves (which intersect $\alpha$) and so again must intersect the root of the one-root region.  
In either case, the geodesic representing the root lies in $\bar{L}^-$, so we conclude that there is a leaf of $L^-$ intersecting $s_1$ and $s_2$, as desired.  
\end{proof}

The next lemma says that, if a pair of prelaminations admits a planar completion, then the pair is FT.  
\begin{lemma}\label{lem_planar_completion_imply_FT}
Let $(\cF^+,\cF^-)$ be a pA-bifoliation with induced laminations $L^\pm_{\cF}$, and suppose $L^\pm \subset L^\pm_{\cF}$ are sublaminations that are both dense in the plane.  Then $(L^+, L^-)$ is fully transverse.  
\end{lemma}

\begin{proof}
$(L^+, L^-)$ satisfy the transversality because they are subsets of the prelaminations induced by $(\cF^+,\cF^-)$. 
For the density condition, note that each open interval $U \subset S^1$ contains ends of leaves of $\cF^+ \cup \cF^-$, and thus endpoints of $\cF^+_0 \cup \cF^-_0$ by density in the plane. Connectedness follows similarly.
\end{proof}

 \begin{lemma}\label{lem_planar_completion_imply_simple_cycle}
 Let $(\cF^+,\cF^-)$ be a pA-bifoliation with induced laminations $L^\pm_{\cF}$, and suppose $L^\pm \subset L^\pm_{\cF}$ are sublaminations that are both dense in the plane.  Then $(L^+, L^-)$ satisfies the simple cycle condition. 
 \end{lemma}
 
\begin{proof}
 From the previous lemma we have that $(L^+, L^-)$ is fully transverse.   
 Suppose for contradiction that there is a complementary region $C$ of $L^\pm$ and an edge $e$ of $\Gamma(C; L^\mp)$ in two different cycles.
 Let $\alpha\in L^\mp$ be a leaf representing $e$. Consider the set $S$ of leaves of $\tilde L^\pm$  that are contained in $C$ and cross $\alpha$.
 We claim that $S$ has at most two elements: If there are three elements $\beta_1,\beta_2, \beta_3$ in $S$, then they must correspond to faces of at least two different leaves in $\cF^+$, which intersect the same leaf of $\cF^-$ corresponding to $\alpha$. By density of $L^+$ in $\cF^+$ we deduce that there are leaves of $L^+$ that will cross $\alpha$ between $\beta_1$, $\beta_2$ and  $\beta_3$. But such leaves would have to intersect $C$, which is impossible.
Therefore either $S$ consists in a unique leaf $l$, or in two leaves $l_1,l_2$ that are the sides of an ideal polygon.

Since $e$ is contained in two cycles, every leaf of $S$ intersect at least three edges of $\graph(C; L^-)$, in the sense that they intersect $\alpha$ as well as at least the two other leaves of $L^-$ representing the edges of $\graph(C;L^-)$ in the two cycles containing $e$. But this implies that either the leaves of $S$ are valence at least $3$ in one of the linkage graphs for $L_\cF^+$, or valence at least $2$ in the linkage graphs on both sides. This implies that the element of $S$ are high-valence leaves, contradicting the fact that $(L_\cF^+,L_\cF^-)$ is induced by a bifoliation.  
\end{proof}

Both the necessity and uniqueness parts of Theorem \ref{thm_completion} will now follow easily from the following characterization of leaves of a planar completion.  It says that the leaves of a planar completion are exactly the ones we added in the construction in Section \ref{sec:proof_completion_thm}. (See Remark \ref{rem:equiv_def_completion}.)
\begin{lemma}\label{lem_charac_of_leaves}
Let $(\cF^+,\cF^-)$ be any planar completion of a pair of prelaminations $(L^+,L^-)$. 
Call $(L_\cF^+,L_\cF^-)$ the pair of prelaminations induced by $(\cF^+,\cF^-)$ on $S^1$. If $\alpha$ is a leaf of $L_\cF^\pm$ then one of the following is true:
\begin{enumerate}[label=(\roman*)]
\item $\alpha\in \bar{L}^\pm\smallsetminus\{\text{high-valence leaves of }\bar{L}^\pm\}$, or
\item $\alpha$ is a crossing geodesic of a cut-edge $e$ or cut-pair $(e_1,e_2)$ 
 in a linkage graph $\graph(K)$ of a complementary region to $\bar{L}^\pm$. That is $\alpha$ is contained in $K$ and a leaf of $L^\mp$ intersects $\alpha$ if and only if it represents the edge $e$, resp.~one of the edges $e_1,e_2$. 
\end{enumerate}
\end{lemma}

\begin{proof}
Recall that $(L^+,L^-)$ are automatically fully transverse by Lemma \ref{lem_planar_completion_imply_FT} and satisfy the simple cycle condition by Lemma \ref{lem_planar_completion_imply_simple_cycle}. In particular, every edge is either disconnecting or part of a disconnecting pair.

To fix notation, we do the proof for leaves of $L^+$.
Let $\alpha\in L_\cF^+$. Then $\alpha$ is not a high-valence leaf of $(L_\cF^+,L_{\cF}^-)$ (which has no high valence leaves) so by Lemma \ref{lem:edges_exist} it is also not high valence for $(L_\cF^+,L^-)$.  
By Lemma \ref{lem:density}, if $\alpha \notin \bar{L}^+$, then it has complementary regions of $L_\cF^+$ on each side; at least one of which is a one-root region.  Let $C_l,C_r$ denote these regions, and suppose $C_l$ is one-root.   Let $K$ be 
the complementary region of $L^+$ containing $C_l \cup C_r$. 
Let $s_l$ be the side of $C_l$ representing the high-valence vertex of $\graph(C_l)$. Then $s_l$ is in $\bar{L}^+$ by Lemma \ref{lem:density}, so it is also a side of $K$. 
Since $C_l$ is one-root, there is a leaf of $L_{\cF}^-$ through $s_l$, and any leaf through $\alpha$ meets $s_l$.  By Lemma \ref{lem:edges_exist} there is also a leaf of $L^-$ through $s_l$ and $\alpha$ (and since $L^- \subset L_{\cF}^-$ any leaf through $s_l$ intersects $\alpha$).  

If $C_r$ is also one-root, the same argument applies: There is a side $s_r$ of $K$ such that every leaf of $L^-$ through $\alpha$ intersects $s_r$, and some such leaf of $L^-$ exists.  We conclude that $\alpha$ is the crossing leaf of an edge from $s_r$ to $s_l$ in $K$, as desired.  

It remains to analyze the case where $C_r$ is an ideal polygon.  Suppose this is the case and let $\alpha = \alpha_0, \alpha_1, \ldots \alpha_k$ denote its sides in cyclic order.   Since $\graph(C_r)$ is a cycle, Lemma \ref{lem:edges_exist} says every adjacent pair $\alpha_i, \alpha_{i+1}$ (indexed cyclically) have a leaf $e_{i}$ of $L^-$ intersecting both of them.   If $\alpha_i \in \bar{L}^+$, then it is a side of $K$, and $e_i, e_{i+1}$ are edges incident to this vertex.  If $\alpha_i \notin \bar{L}^+$ then Lemma \ref{lem:density} implies that on the other side of $\alpha_i$ there is a one-root region; let $s_i$ denote the root side.  Thus, $e_i$ and $e_{i+1}$ intersect $s_i$, which is necessarily a side of $K$, and each leaf of $L^-$ passing through $\alpha_i$ must also intersect $s_i$. In particular, no leaves intersect both $s_i$ and $s_j$ if $|j-i| > 1$. 
This shows that $K$ contains a cycle with edges $e_i$, and $\alpha = \alpha_0$ is the crossing leaf of the pair $(e_0, e_1)$.  
\end{proof}

Using this, we now prove the necessity of the conditions in Theorem \ref{thm_completion}:
\begin{proposition}[Necessity of conditions of Theorem \ref{thm_completion}] \label{l.necessary}
Let $(\cF^+,\cF^-)$ be a planar completion of a pair of prelaminations $(L^+,L^-)$.
Then $(L^+, L^-)$ are FT, 
satisfy simple cycle, have no high valence leaves, and no three $L^\pm$ leaves sharing an endpoint cross a common $L^\mp$ leaf.  
\end{proposition} 

\begin{proof}
We check each condition separately \\
\noindent{\bf Fully transverse:} This is given by Lemma \ref{lem_planar_completion_imply_FT}.

\noindent{\bf Simple cycle condition:} This is given by Lemma \ref{lem_planar_completion_imply_simple_cycle}.

\noindent{\bf No high valence:} 
Suppose $\alpha\in L^\pm$ has valence $\geq 2$ for a linkage graph $\graph(K;L^\mp)$, with $K$ a complementary region of $L^\pm$, then, by Lemma \ref{lem_charac_of_leaves}, it has valence greater than two for $\graph(C;L_\cF^\mp)$ with $C$ a complementary region of $L_\cF^\pm$ contained in $K$. Since $L_\cF^\pm$ is induced by a pA-bifoliation, it implies that $C$ is an ideal polygon, and that either the other side of $\alpha$ is a one-root region or $\alpha$ only bounds one complementary region. In either case, $\alpha$ is not high-valence.

Similarly, $\alpha\in L^\pm$ cannot be the unique geodesic side of a degenerate linkage graph.
Thus, $(L^+,L^-)$ has no high-valence leaves.

\noindent{\bf Leaves sharing endpoints:}
	
Two distinct leaves $l_0,l_1$ of $\cF^+$ that have a shared endpoint cannot intersect the same leaf $l^-\in \cF^-$, otherwise every leaf of $\cF^+$ intersecting $l^-$ in between $l_0$ and $l_1$ (of which there are uncountably many) would have one shared endpoint, contradicting the countability condition in the definition of $S^1_\infty$. Now given three distinct leaves $\alpha_1,\alpha_2,\alpha_3$ of $L^+$ with one shared endpoint, at most two of them can represent faces of the same $\cF^+$-leaf, so no leaf of $L^-$ can intersect the three leaves $\alpha_i$. 
\end{proof}

Now we can prove the uniqueness of planar completions:
\begin{proposition}\label{prop_uniqueness_completion}
Let $(L^+,L^-)$ be a pair of prelaminations 
and $(\cF_i^+,\cF_i^-)$, $i=1,2$, be two planar completions of $(L^+,L^-)$. Then there exists a homeomorphism of the disc taking $\cF^\pm_1$ to $\cF^\pm_2$ which is the identity on $S^1$.
\end{proposition}

The proof below uses Lemma \ref{lem_charac_of_leaves}.  As an alternative approach, one could also imitate the strategy of Proposition \ref{prop:uniqueness}, using the rectangle topology.

\begin{proof}
Let $L_i^\pm$ be the prelaminations induced by $\cF_i^\pm$. By Proposition \ref{l.necessary}, the pair $(L^+,L^-)$ satisfies the conditions of Theorem \ref{thm_completion}, so we may assume without loss of generality, that $L_2^\pm$ is the completion $\tilde L^\pm$ that we constructed in section \ref{sec:proof_completion_thm}. By assumption, both $L_i^\pm$ contain the leaves of $L^\pm$. We need to show that $L_1^\pm=L_2^\pm$.

By Lemma \ref{lem_charac_of_leaves}, if $\alpha \in L_1^\pm$ then either $\alpha$ is a crossing geodesic or is in the closure of $L^\pm$, and not a high-valence leaf. By definition of the completion (see Observation \ref{obs_types_leaves_completion}), we deduce that $\alpha \in L_2^\pm$.
Thus, $L_1^\pm \subset L_2^\pm$.   The proof is now reduced to showing that, in the case of induced laminations, containment implies equality.  Since this is a general fact, we prove this in a separate lemma below.   Using this lemma (Lemma \ref{lem:containment_is_equality}), we have $L_1^\pm = L_2^\pm$.  Now the uniqueness part of Theorem \ref{thm:realization} implies that $\cF_1^\pm = \cF_2^\pm$, i.e., there is a homeomorphism of the disc taking one to the other and restricting to the identity on the boundary circle.  
\end{proof}

\begin{lemma} \label{lem:containment_is_equality}
Suppose $(L_i^+, L_i^-)$, $i=1,2$, are prelaminations induced by pA-bifoliations$(\cF_i^+, \cF_i^-)$. If $L_1^\pm \subset L_2^\pm$, and $L_1^\pm$ is dense in $\cF_2^\pm$ then $L_1^\pm = L_2^\pm$.  
\end{lemma} 

\begin{proof} 
Let $L^\pm_i$, $\cF^\pm_i$ be as in the statement of the Lemma. 
Define a map from the open disc to open disc as follows.  Each point in the disc is uniquely realizable as the intersection of a leaf of $\cF^+_1$ with one of $\cF^-_1$.  Suppose $x = l^+ \cap l^-$.  Consider the endpoints $e^+(l^+)$ and $e^-(l^-)$.  The set of endpoints of $l^\pm$ is a subset of $S^1_\infty(\cF_1^+, \cF_1^-)$, either two points or $k$ points depending on whether $l^\pm$ is singular or not.  In the singular case, the set of $k$ points are the ideal vertices of an ideal $k$-gon.  
Since $L_1^\pm \subset L_2^\pm$, $e^+(l^+)$ and $e^-(l^-)$ are also the endpoints of leaves of $\cF^+_2$ and $\cF^-_2$ respectively.  Send $x$ to the unique intersection point of these two leaves.  Note that they do indeed intersect because intersection is encoded by the boundary circle.   Thus, this map is well defined, and by construction it is continuous and injective.  
Since $L_1^\pm$ is dense in $\cF_2^\pm$, the image of this map is dense in the open disc; thus by the invariance of domain theorem, it is a homeomorphism onto the disc.   By construction it sends leaves to leaves; so we conclude $L_1^\pm = L_2^\pm$.
\end{proof}

We can now also prove Corollary \ref{cor:for_nontransitive} from the Introduction:
\begin{corollary} \label{cor:for_nontransitive_inthepaper}
Let $(P_i,\cF_i^+,\cF_i^-)$, $i=1,2$ be pA-bifoliations. Let $\cL_i^\pm$ be subsets of $\cF_i^\pm$ and call $L_i^\pm$ the prelaminations induced by $\cL_i^\pm$ in $\Pbound_i:=S^1_\infty(\cF^+_i, \cF^-_i)$.   

Let $h\colon \Pbound_1 \to \Pbound_2$ be a homeomorphism inducing the map $\hat{h}\{x,y\} = \{h(x), h(y)\}$ on pairs of points in $\Pbound_1$.  
Suppose that:
\begin{enumerate}[label=(\roman*)]
\item The subsets $\cL_i^\pm \subset \cF_i^\pm$ are dense in the plane, and 
\item $\hat{h}(L_1^+\cup L_1^-) = L_2^+\cup L_2^-$. 
\end{enumerate}
Then there exists a unique homeomorphism $H\colon P_1 \to P_2$ such that $H|_{\Pbound_1} = h$, $H(\cF^+_1) = \cF^\pm_2$, and $H(\cF^-_1) = \cF^\mp_2$.
\end{corollary}

\begin{proof}
By Lemma \ref{lem_planar_completion_imply_FT}, the pairs of prelaminations $(L_i^+,L_i^-)$ are fully transverse. 

Pick some $\alpha_0\in L_1^+$. By assumption, $\hat{h}(\alpha_0)$ is in $L_2^+$ or in $L_2^-$. Up to exchanging the names of $L_2^+$ and $L_2^-$, we will assume that $\hat{h}(\alpha_0)\in L_2^+$.
For any leaf $\beta\in L_1^-$ that intersects $\alpha$, the leaf $\hat{h}(\beta)$ intersects $\hat{h}(\alpha_0)$, so $\hat{h}(\beta)\in L_2^-$. Since $(L_i^+,L_i^-)$ are fully transverse, the connectedness axiom together with an inductive argument implies that we must have $\hat{h}(L_1^+) = L_2^+$ and $\hat{h}(L_1^-) = L_2^-$, i.e., $h$ preserves the individual laminations, not just their union.  

Let $g\colon \bD^2 \to \bD^2$ be a homeomorphism that extends $h$. Then $(g(\cF^+_1),g(\cF^-_1))$ and $(\cF^+_2,\cF^-_2)$ are two planar completions of $h(L_1^+,L_1^-) = (L_2^+,L_2^-)$.
Thus, by Proposition \ref{prop_uniqueness_completion}, $(g(\cF^+_1),g(\cF^-_1))$ and $(\cF^+_2,\cF^-_2)$ are homeomorphic via a map that is the identity on the boundary. Thus, there exists a (unique) homeomorphism $H\colon P_1 \to P_2$, extending $h$ and such that $H(\cF^\pm_1) = \cF^\pm_2$.
\end{proof}

The fact that $H$ is canonically defined shows that, as in Proposition \ref{prop:uniqueness}, the extension to the disc respects composition.  This gives the proof of Theorem \ref{cor:action_extends_2}.  

Finally, we prove Theorem \ref{thm_special_case_regular} from the Introduction:
\begin{corollary} \label{cor:restatmentofregularformintro}
Suppose that $(L^+, L^-)$ is a pair of {\em regular} fully transverse prelaminations.   There is a unique planar completion $(\cF^+, \cF^-)$ of $(L^+, L^-)$ into a {\em nonsingular} bifoliated plane if and only if each simple, closed, polygon in $\bD$ whose sides are (finite) leaf segments of $\geo(L^+) \cup \geo(L^-)$ is a rectangle.  

Equivalently, this occurs if and only if there are no cycles among linkage graphs.  
\end{corollary}

\begin{proof}
Leaves of regular FT prelaminations are accumulated on both sides, so do not bound any complementary regions. Thus no leaf of a regular FT prelamination can be high-valence, and condition \ref{item_no_high_valence} of Theorem \ref{thm_completion} is automatically verified. Similarly, no two leaves of a regular prelamination share a common endpoint so condition \ref{item_same_endpoint_condition} is also automatically satisfied.
Hence, to deduce Corollary \ref{cor:restatmentofregularformintro} from Theorem \ref{thm_completion}, it suffices to notice two facts: First, a pA-bifoliation $(\cF^+,\cF^-)$ is non-singular if and only if there are no cycles in the linkage graphs of its induced prelamination, or any prelaminations whose planar completion is  $(\cF^+,\cF^-)$. Second, there exists a cycle in a linkage graph if and only if there is a closed polygonal path with strictly more than 4 sides.
\end{proof}

\section{Unique versus non-unique embeddings of prelaminations} \label{sec:uniqueness}

In Theorem \ref{thm_completion}, uniqueness of the planar completion was due in large part to the requirement that leaves of $L^\pm$ be dense in the plane.
In this section, we relax this density requirement by considering more general ways a prelamination may sit inside a pA-bifoliation :
\begin{definition}[Embedding into a bifoliation, planar extension]
A pair of prelaminations $(L^+,L^-)$ \emph{embeds} in a pA-bifoliation $(\cF^+, \cF^-)$ if $L^+ \subset L_\cF^\pm$, where $L_\cF^\pm$ are the prelaminations induced by $\cF^\pm$. In this case we say $(\cF^+, \cF^-)$ is a \emph{(planar) extension} of $(L^+,L^-)$.
\end{definition}

We emphasize that, as opposed to a planar completion, the prelaminations $L^\pm$ are not necessarily dense in a planar extension. 

We do not address in this article the general question of when a pair of prelaminations admits a planar extension. Instead, we consider the following question of rigidity:
{\em For which FT prelaminations $(L^+, L^-)$ is their planar completion the unique planar {\em extension} of $(L^+, L^-)$? } 
In Section \ref{subsec_examples}, we provide examples of such prelaminations admitting (many) different extensions, and in Section \ref{subsec_uniqueness}, we  prove a rigidity result, characterizing prelaminations whose unique extension is their planar completion.

For this, we use two different strategies.  Roughly speaking, in Section \ref{subsec_examples}, we start with a bifoliation, then remove some leaves and show that what is left will still admit a planar completion.  By contrast, in Proposition \ref{p.killing-product} and \ref{prop_uncountable_ideal}, we start with a pair of FT prelaminations and add (a lot of) leaves in some complementary regions making sure that these added leaves are \emph{not} crossing geodesics, so that the resulting bifoliated planes are ``bigger" than the completion.  

\subsection{Examples} \label{subsec_examples}
As we aim to build examples of distinct extensions of the same pair of prelaminations, we start by making precise what we mean by distinct:
\begin{definition}
Let $(L^+,L^-)$ be a pair of FT prelaminations.
Two planar extensions $(\cF_i^+,\cF_i^-)$, $i=1,2$, of $(L^+,L^-)$ are \emph{weakly equivalent} if there is an orientation preserving homeomorphism $h$ of $\bD^2$ sending $\cF^\pm_1$ to $\cF_2^\pm$ and whose restriction to $S^1$  preserves (setwise) $L^+$ and $L^-$.
They are  \emph{strongly equivalent} if one can find such a homeomorphism $h$ whose restriction to $S^1$ is the identity.
\end{definition}

Notice that strong equivalence was the notion we used to define ``uniqueness'' in Theorems \ref{thm:realization} and \ref{thm_completion}.
The difference between the notions of strong and weak equivalences will appear in Example \ref{e.R-covered} (see Claim \ref{claim_exampleR-covered}). We now can state precisely the results proved below.   
For simplicity, we only build examples involving non-singular foliations. One could of course generalize these constructions to give singular examples.

\begin{proposition}\label{p.non-uniqueness}
There exists pairs $(L^+,L^-)$ of fully transverse prelaminations that admit a planar completion and also admit an uncountable family $(\cF^+_i,\cF^-_i)_{i\in \cI}$ of pairwise non-weakly equivalent extensions.
\end{proposition}
In the example we provide for Proposition~\ref{p.non-uniqueness} (see Example \ref{e.R-covered} below) the set of points in $S^1$ which are endpoints of leaves in $L^+$ is not dense.  However, this lack of density is only one of the phenomenon that prevents uniqueness. Indeed, we also build an example (see Example \ref{e.lattice}) such that the endpoints of $L^+$ as well as those of $L^-$ are dense in $S^1$:
\begin{proposition}\label{p.still-non-uniqueness}
There exists prelaminations $(L^+,L^-)$ admitting a planar completion and such that:
\begin{itemize}
\item there are two pairs $(\cF^+_i,\cF^-_i)_{i\in \{0,1\}}$ of pairwise non-equivalent extensions of $(L^+,L^-)$.
\item both sets $\{a\in S^1_{\infty}(\cF^+,\cF^-) : \exists b, (a,b)\in L^\pm\}$ are dense in $S^1_{\infty}(\cF^+,\cF^-)$.
\end{itemize}
\end{proposition}

As one tool to construct examples, we use the following lemma, which gives a condition for a subset of an induced prelamination to satisfy the conditions of Theorem \ref{thm_completion} (i.e., to admit a planar completion).  

\begin{lemma}\label{l.non-uniqueness}
Let $(\cF^+,\cF^-)$ be a pair of transverse foliations of the open disc $\mathring\bD^2$. Assume that the set of endpoints of leaves of $\cF^-$ is dense in $S^1_{\infty}(\cF^+,\cF^-)$.

Let $U\subsetneq \bD^2$ be a proper, open, $\cF^+$-saturated set such that every leaf of $\cF^-$ intersects $U$,
and let $L^+$ be the set of pair of endpoints of $\cF^+$-leaves contained  in $U$, and $L^-$ the set of pair of endpoints of $\cF^-$-leaves.
Then $(L^+,L^-)$ admits a (non-singular) planar completion $(\tilde \cF^+,\tilde\cF^-)$.

Moreover, $(\tilde \cF^+,\tilde\cF^-)$ is not weakly-equivalent to $(\cF^+,\cF^-)$, so $(L^+,L^-)$ admits at least two distinct extensions.
\end{lemma}

\begin{proof}
We will show that $(L^+,L^-)$ satisfies the conditions of Theorem \ref{thm_completion}, with the added property that the linkage graphs have no cycles at all, hence proving that it admits a non-singular planar completion $(\tilde \cF^+,\tilde\cF^-)$. Since $L^+$ is dense in $\tilde \cF^+$ but not dense in $\cF^+$, the bifoliations cannot be even weakly equivalent.

The fact that $(L^+,L^-)$ embeds in the non-singular bifoliation $(\cF^+,\cF^-)$ directly implies that $(L^+,L^-)$ are transverse, have only trees as linkage graphs and any three leaves of $L^\pm$ that intersect the same $L^\mp$ leaf cannot all share the same endpoint in $S_\infty^1$.

The prelaminations $(L^+,L^-)$ have no high-valence leaves because $U$ is open, hence any leaf of $L^+$ that is the boundary of a complementary region of $L^+$ is also a boundary of a complementary region of the lamination induced by all the leaves of $\cF^+$. In particular, they are not high-valence.

So we only have to show the density and connectedness properties to finish proving that the prelaminations are FT.
The density condition of $(L^+,L^-)$ follows from the hypothesis that the endpoints of $\cF^-$-leaves are dense (recall that $L^-$ contains the endpoints of all leaves of $\cF^-$).
For connectedness, 
since any leaf of $\cF^-$ intersects $U$ by hypothesis, it in particular intersects leaves of $\cF^+$ contained in $U$ (recall that $U$ is saturated by $\cF^+$). Thus, one can cover $\mathring \bD^2$ by a countable family of open sets which are each the union of the $\cF^-$-leaves crossing a given $\cF^+$-leaf contained in $U$. The connectedness property of $(L^+,L^-)$ follows.
\end{proof}

The next example gives the family used for Proposition~\ref{p.non-uniqueness}.

\begin{example}\label{e.R-covered}
Let $\cF^+,\cF^-$ be the horizontal and vertical foliations, respectively, of $\bR^2$ restricted to the strip $B=\{(x,y)\in\bR^2 : |x-y|<1\}$.   $B$ is homeomorphic to $\R^2$, so $(\cF^+, \cF^-)$ is a bifoliated plane.

Let $\Delta_n:=\{(x,y)\in\bR^2 :  |y-4n|\leq \frac{1}{2} \}$, for $n\in \bZ$.
Consider the open set 
\[U:= B\smallsetminus \cup_{n\in \bZ} \Delta_n.\]
Let $L^+$ be the pairs of endpoints of $\cF^+$-leaves in $U$, and $L^-$ be the set of endpoints of all $\cF^-$-leaves. 

We will now build an uncountable family of pairwise non-weakly equivalent extensions of $(L^+,L^-)$, by pairs of transverse $\bR$-covered foliations (i.e., such that each foliation is conjugated to a trivial one).

Let $\cE\subset \bZ$ be any subset and let $L^+_\cE $ be the prelamination obtained from $L^+$ by adding all the $\cF^+$-leaves in $\Delta_n$ if $n\notin \cE$. The prelaminations $(L^+_\cE,L^-)$ are as in Lemma \ref{l.non-uniqueness}, so in particular they admit planar completions $(\cF^+_\cE,\cF^-_\cE)$.

Moreover, one easily sees that the boundary circles $S^1_\infty(\cF^+_\cE,\cF^-_\cE)$ are all canonically identified with the boundary circle of $S^1_\infty(\cF^+,\cF^-)$ which is just the boundary of the strip $B$ together with two points at infinity (one for the top right infinity and one for the bottom left).
Thus the family $\left\{(\cF^+_\cE,\cF^-_\cE)\right\}_{\cE\subset \bZ}$ is an uncountable collection of extensions of $(L^+,L^-)$.
\end{example}

In order to finish the proof of Proposition \ref{p.non-uniqueness}, we only have to show that different choices of subsets of $\bZ$ in the example above gives different bifoliations.  This is shown in the following claim.  

\begin{claim} \label{claim_exampleR-covered}
Let $\cI,\cJ \subset \bZ$.  Then the bifoliations $(\cF^+_\cI,\cF^-_\cI)$ and $(\cF^+_\cJ,\cF^-_\cJ)$  are strongly equivalent if and only if $\cI=\cJ$, and weakly equivalent if and only if there is $m\in \bZ$ such that $\cJ=\pm \cI+m$.

In particular, the family $\left\{(\cF^+_\cE,\cF^-_\cE)\right\}_{\cE\subset \bZ}$ contains uncountably many non-weakly equivalent extensions of $(L^+,L^-)$.
\end{claim}
\begin{proof}
Let $\alpha_{n,\pm}$ denote the geodesics corresponding to the leaves $\{y=4n\pm\frac 12\}$.
The $\alpha_{n,\pm}$, $n\in \cI$ (resp.~$n\in \cJ$) are the boundary components of complementary regions of $L^+_\cI$ (resp. of $L^+_\cJ$) and are limit of $L^+$-leaves. If $(\cF^+_\cI,\cF^-_\cI)$ and $(\cF^+_\cJ,\cF^-_\cJ)$  are strongly equivalent via a homeomorphism $h$, then $h$ maps boundary components to boundary components and is the identity on $S^1$, so both prelaminations have the same boundary components. Thus $\cI=\cJ$.

We will show that the difference $s-r$ for two successive points in $\cI$ can be characterized by a sort of ``push-map'' defined using the prelaminations. This will imply that $(\cF^+_\cI,\cF^-_\cI)$ and $(\cF^+_\cJ,\cF^-_\cJ)$  are conjugated by a homeomorphism preserving the orientations of the foliations if and only if $\cJ=\cI+m$ finishing the proof of the claim.

Recall that we can identify the circle at infinity of $(\cF^+_\cI,\cF^-_\cI)$ with the boundary of the strip $B$ union the two infinite ends.  We fix this identification. 
The laminations  $L_\cI^+,L^-$ induce two maps $f^+_\cI,f^-$ on subsets of this circle, associating to an endpoint of a leaf its other endpoint. The map $f^-$ is defined on the whole circle minus the two infinite ends, while the map $f_\cI^+$ is defined away from these two infinite ends as well as the intervals corresponding to the ideal segments $\sigma_{n,\pm}$, $n\in \cI$.
Where it is defined, the composition $f^-\circ f^+_\cI$ coincides with the map $(x,y)\mapsto (x+1,y+1)$ on $\{y-x=1\}$ and with $(x,y)\mapsto (x-1,y-1)$ on $\{x-y=1\}$, which are the boundary component of the strip.
Then, if $s,r$ are two consecutive elements of $\cI$, the difference $s-r$ is characterized by the number of fundamental domains of $f^-\circ f^+_\cI$ between the two intervals $\sigma_{r,+}$ and $\sigma_{s,+}$. 

This difference is preserved by any orientation preserving homeomorphisms conjugating $(\cF^+_\cI,\cF^-_\cI)$ and $(\cF^+_\cJ,\cF^-_\cJ)$, leading to the announced result.
\end{proof}

The next example is used for Proposition~\ref{p.still-non-uniqueness}.

\begin{example}\label{e.lattice}
Let $\mu,\nu$ be two irrational numbers, independent over $\mathbb{Q}$, and $\Lambda= \bZ (1,\mu)\oplus \bZ (1,\nu)$ be the induced lattice on $\bR^2$.

Let $P$ be the universal cover of $\bR^2\smallsetminus \Lambda$ and $(\cF^+,\cF^-)$ the bifoliation on $P$ obtained by lifting the horizontal and vertical foliations of $\bR^2\smallsetminus \Lambda$.
Call $\pi\colon P \to \bR^2\smallsetminus \Lambda$ the projection.

Let $\Delta\subset P$ be the lift of the strip $\{ |y|\leq 1\} \cap (\bR^2\smallsetminus \Lambda)$, and $U= P\smallsetminus \Delta$. Notice that the closed set  $\Delta$ does not contain any leaf of $\cF^-$ and is saturated for $\cF^+$ (so that $U$ is also saturated for $\cF^+$).

Finally, let $(L^+,L^-)$ be the prelaminations such that  $L^-$ is the set of  pairs of endpoints of the  $\cF^-$-leaves and $L^+$ is the set of  pairs of endpoints of leaves contained in $U$.
\end{example}

\begin{proof}[Proof of Proposition~\ref{p.still-non-uniqueness}]
Let $(L^+,L^-)$ be as in the example constructed above.  We will show that they
satisfy the conditions of Lemma \ref{l.non-uniqueness}, so that $(\cF^+,\cF^-)$ is an extension that is non-weakly equivalent to its completion.
We will also show that the set of endpoints of $L^+$ leaves, and the set of endpoints of $L^-$ leaves are 
each dense in  $\partial P := S^1_{\infty}(\cF^+,\cF^-)$.

First, note that a leaf of $\cF^\pm$ in $P$ whose projection by $\pi$ is contained in a horizontal or vertical line that contain a point of the lattice $\Lambda$ is non-separated from above and below with other leaves. 
Thus, the set of $\cF^\pm$-leaves which are not separated from below is dense in $P$, and so are leaves non separated from above. According to \cite[Lemma 4.19]{Bonatti_boundary} this implies that any non-empty open interval $I$ of $\Pbound$ contains both endpoint of a leaf of $\cF^-$, and of $\cF^+$.  Then \cite[Proposition 5.7]{Bonatti_boundary} implies that the endpoints of leaves of $\cF^\pm$ are dense in $\Pbound$. In particular, we deduce that  $(L^+,L^-)$ satisfy the conditions of Lemma \ref{l.non-uniqueness} and that the endpoints of $L^-$-leaves are dense in $\Pbound$.

We have left to show that the endpoints of $L^+$ are also dense. 
Suppose $I$ is an non-empty open interval in $\Pbound$. By the above, we know that there is a leaf $\beta$ of $L^-$ having both ends in $I$.  As the pair $(L^+,L^-)$ is fully transverse, $\beta$ crosses a $L^+$-leaf $\alpha$, which must thus have an endpoint in $I$, which is what we needed to show.
\end{proof}

\subsection{A characterization of unique strong extensions} \label{subsec_uniqueness}
Having built examples of non-(weakly) equivalent extensions, we now describe for which prelaminations 
their {\em unique} extension up to strong equivalence is their planar completion. 
\begin{theorem}\label{thm_uniqueness_extension}
Suppose $(L^+,L^-)$ are prelaminations admitting a planar completion. This planar completion is the unique extension of $(L^+, L^-)$ up to strong equivalence if and only if for each complementary region $C$ of $L^\pm$, the ideal boundary of $C$ (i.e., the set $\partial C \cap S^1$) has countable cardinality.
\end{theorem}

\begin{rem}
In fact, we will prove slightly more: the countability condition implies the uniqueness of the extension up to strong equivalence, while the failure of the countability condition implies the existence of at least two non weakly equivalent extensions.

Deciding whether one can build weakly equivalent but non-strongly equivalent completions of a given pair of prelamination would involve a study of the symmetries of bifoliated planes which is beyond the scope of this article.
\end{rem}

We start by remarking that it is enough to prove Theorem \ref{thm_uniqueness_extension} for prelaminations \emph{induced} by a pA-bifoliation.

\begin{lemma} \label{lem:uncountable_boundary_iff}
Let $(L^+,L^-)$ be prelaminations admitting a planar completion and let $(L_\cF^+,L_\cF^-)$ be the prelaminations induced from the completion.
Then $L^\pm$ has a complementary region with uncountable ideal boundary if and only if $L_\cF^\pm$ also does.
\end{lemma}

\begin{proof}
By Lemma \ref{lem_charac_of_leaves}, leaves of $L_\cF^\pm$ are either leaves of the closure $\bar{L}^\pm$ (which does not affect the nontrivial complementary regions) or are crossing geodesics in a nontrivial complementary region $K$ of $L^\pm$.  Since there are only countably many crossing leaves, every nontrivial complementary region $K$ of $L^\pm$ is a countable union of complementary regions $C_i$ of $L_\cF^\pm$. In particular, $K$ has an uncountable ideal boundary if and only if at least one $C_i$ has uncountable ideal boundary.
\end{proof}

We start with the proof of the sufficient condition:
\begin{proposition}\label{p.unicité} 
Suppose $(L^+,L^-)$ are prelaminations admitting a planar completion.
Assume that the ideal boundary of each complementary region of $L^\pm$  is countable.  Then $(L^+,L^-)$ admits a unique, up to strong equivalence, extension which is its planar completion.
\end{proposition}
\begin{proof}
Let $(\cF^+, \cF^-)$ be an extension of $(L^+,L^-)$ and let $L_\cF^\pm$ be the prelaminations induced by $\cF^\pm$.
We will show $L^+$ and $L^-$ are is dense in the plane for $(\cF^+, \cF^-)$, hence $(\cF^+, \cF^-)$ is the (unique) planar completion of $(L^+,L^-)$ by Theorem \ref{thm_completion}.  

First, notice that each trivial complementary region of $L^\pm$ must be a leaf of $L_\cF^\pm\smallsetminus L^\pm$. Call $\hat L^\pm$ the prelamination $L^\pm$ together with all trivial complementary regions added. So $\hat L^\pm \subset L_\cF^\pm$.
Now, all the leaves of $L_\cF^\pm\smallsetminus \hat L^\pm$ are necessarily contained in nontrivial complementary regions of $L^\pm$. Since there are only countably many nontrivial complementary regions, that these regions have countable ideal boundaries, and that only countably many leaves of $L_\cF^\pm$ can share the same endpoint (because it is induced by a bifoliation), we deduce that there are only countably many leaves in $L_\cF^\pm\smallsetminus \hat L^\pm$.

Therefore $\hat L^\pm$, and thus $L^\pm$, are dense in the foliation $\cF^\pm$. So $(\cF^+, \cF^-)$ is the unique planar completion of $(L^+,L^-)$.
\end{proof}

To prove the other direction of Theorem \ref{thm_uniqueness_extension}, we need to show that if one complementary region has an uncountable ideal boundary, then we can build non-equivalent extensions.  
We start with a lemma for complementary regions that contain ideal segments or not.  

\begin{proposition}\label{p.killing-product} 
Suppose $(L^+,L^-)$ are prelaminations admitting a planar completion.  Assume that there is a complementary region $C$ of $L^\pm$ such that $\partial C$ contains a nontrivial interval of $S^1$.
Then $(L^+,L^-)$ admits an extension that is not weakly equivalent to the planar completion.
\end{proposition}
\begin{proof}
By assumption, the linkage graph $\graph(C)$ contains at least one ideal vertex $I$ corresponding to an ideal segment. Let $e$ be the edge containing $I$ and call $\gamma$ its other vertex. So $\gamma$ is a geodesic side of $C$ (Lemma \ref{l.no-ideal-connection}).

Consider two closed segments $U,V$ such that $U\subset \mathring V\subset V\subset \mathring I$, and let $V_0,V_1$ be the two connected components of $V\smallsetminus U$.

Pick $\phi\colon V_0\to V_1$ a decreasing homeomorphism.  We define
$$L^+_1:=L^+\cup\{(a,\phi(a)) : a\in V_0\}.$$

One easily sees that $(L^+_1,L^-)$ is a pair of FT prelamination that still satisfy all the conditions of Theorem \ref{thm_completion}.

Hence $(L^+_1,L^-)$ admits a planar completion $(\cF_1^+,\cF^-_1)$, which is an extension of $(L^+,L^-)$. Since $L^+$ is by construction not dense in $\cF_1^+$, we deduce that $(L^+,L^-)$ admit at least two non-weakly equivalent extensions.
\end{proof}

\begin{rem}
One can easily modify the construction in the proof above to get an uncountable family of non equivalent extensions. For example, it suffices to consider infinite families of pairs of intervals and homemorphisms $\phi_n\colon V_{0,n}\to V_{1,n}$ with compatible relative positions. It is obvious that the extensions obtained that way are not strongly equivalent, the fact that uncountably many such extensions are not they are not even weakly equivalent is more technical.
\end{rem}

The next proposition finishes the proof of Theorem \ref{thm_uniqueness_extension}, using Lemma \ref{lem:uncountable_boundary_iff}.

\begin{proposition}\label{prop_uncountable_ideal}
Let $(L^+,L^-)$ be prelaminations admitting a planar completion 
Suppose that a complementary region $C$ of $L^\pm$ has an uncountable ideal boundary. Then there are at least two non weakly equivalent extensions of $(L^+,L^-)$.
\end{proposition}

\begin{proof}
By Lemma \ref{lem:uncountable_boundary_iff}, if $(L^+,L^-)$ have such a complementary region, then 
the prelaminations induced by their planar completion also has this property.  Thus, using Corollary \ref{cor:completion_stabilizes} we can assume without loss of generality that $(L^+,L^-)$ are induced by a pA-bifoliation.

If one ideal boundary of a complementary region of $L^\pm$ has non empty interior, then the result follows from  Proposition~\ref{p.killing-product}. So from now one, we restrict to the case where no complementary regions have any ideal sides (equivalently, the endpoints of leaves of $L^+$ and $L^-$ are both dense).

To fix notation, we now suppose that there exists a complementary region $C$ of $L^+$ with uncountable ideal boundary. The case for $L^-$ is obviously symmetric.
Since $(L^+,L^-)$ are induced by a pA-bifoliation, and $C$ has infinitely many sides, it is not an ideal polygon so is one-root region (by the necessity part of Theorem \ref{thm:realization}).

Since the ideal boundary of $C$ is an uncountable compact set in $S^1$ with empty interior, the Cantor--Bendixson Theorem says that this ideal boundary is the union of a Cantor set $\cC$ and a countable set.
Let $U_0,U_1$ be disjoint open subsets of $\cC$ such that $\cC=\cC_0\cup\cC_1$ where $\cC_i := \cC\cap U_i$ are both homeomorphic to cantor sets.

Choose $\phi\colon\cC_0\to \cC_1$ a decreasing homeomorphism.
Let $R_0\subset\cC_0$ and $R_1\subset\cC_1$ be the subsets consisting of the points which are accumulated on both sides. 

Define
$$L^+_0=\{ (a,\phi(a) : a\in R_0\}.$$
Then $L^+_0$ is a \emph{regular} prelamination of $S^1$ whose leaves are contained in $C$ and thus do not cross any $L^+$-leaf.  Therefore, $L^+_1=L^+\cup L^+_0$ is a prelamination.
Moreover, by the connectedness property of $(L^+,L^-)$, every leaf of $L^+_0$ crosses a $L^-$-leaf. Thus $(L^+_1,L^-)$ are fully transverse.

We will show the following
\begin{claim}\label{claim_L_1_completable}
The pair $(L^+_1,L^-)$ satisfies the hypotheses of Theorem~\ref{thm_completion}.\footnote{It is essential in the proof of this claim that the leaves $L_0^+$ that we added are regular, otherwise this claim would fail.}
\end{claim}

With the claim in hand, one can conclude the proof of the proposition as previously: The planar completion of $(L^+_1,L^-)$ is an extension of $(L^+,L^-)$ and cannot be weakly equivalent to the planar completion of $(L^+,L^-)$.\qedhere
\end{proof}

\begin{proof}[Proof of Claim \ref{claim_L_1_completable}]
We have already proved that $(L^+_1,L^-)$ are fully transverse. As no two leaves of $L_0^+$ share the same endpoint (and $(L^+,L^-)$ are induced by a pA-bifoliation), condition \ref{item_same_endpoint_condition} of Theorem \ref{thm_completion} is also satisfied.

So we have left to show the simple cycle and no high valence leaves conditions. To do this we describe how the linkage graphs have changed.

Recall that the leaves in $L^+_1\smallsetminus L^+$ are contained in a fixed complementary region $C$ of $L^+$. So if $K^-$ is a complementary region of $L^-$ that does not intersect $C$, we trivially have that $\graph(K^-; L^+_1)=\graph(K^-;L^+)$. Similarly, if $K^+$ is a complementary region of $L_1^+$ that is not contained in $C$, then $K^+$ is also a complementary region of $L^+$ and the linkage graphs are the same.

Now suppose $K^-$ is a complementary region of $L^-$ that does intersect $C$.
By assumption, $K^-$ has only geodesic sides.  Since $C$ is a one-root region, every $L^-$ leaf that intersects $C$ must intersect the root, call it $\gamma$, of $C$. As the root is always accumulated upon by leaves of $L^+$, we deduce that if a leaf of $L^+_0$ crosses two sides of $K^-$, these two (geodesic) sides also cross a same $L^+$-leaf.
So in this case too, $\graph(K^-; L^+_1)=\graph(K^-;L^+)$.
We deduce that $L^-$ has no high-valence leaves and the simple cycle condition for the pair $(L_1^+,L^-)$.

We treat the last case now.  Suppose $K^+$ is a complementary region of $L_1^+$ contained in $C$.
Since the prelamination $L_0^+$ is regular, the sides of $K^+$ are either sides of $C$ or geodesics accumulated upon by, but not equal to, leaves of $L_0^+$.
Moreover, since the sets $R_0$, $R_1$ are contained in disjoint open intervals  and $\phi$ is decreasing, the set of geodesics $(a,\phi(a))$ is totally ordered in the following way: $(a,\phi(a))<(b,\phi(b))$ if $(a,\phi(a)$ separates $(b,\phi(b))$ from $\gamma$, the root of $C$.
Therefore, $K^+$ has a geodesic side which is its ``lowest side'' (which is either $\gamma$ or the limit of the increasing sequence of leaves in $L_0^+$ which separate $K^+$ from $\gamma$). Since $C$ is a one-root region of $(L^+,L^-)$, any leaf of $L^-$ that intersects $K^+$ intersects that lowest sides. In particular, $\graph(K^+;L^-)$ is a star, and $L_1^+$ has no high-valence leaves. This ends the proof of the claim.\qedhere
\end{proof}

\appendix

\section{Naturality of the geodesic realization}\label{sec.naturality}

Here we prove the naturality of geodesic realizations as follows:
\begin{theorem}\label{t.naturality}
Let $(L^+, L^-)$ be a pair of fully transverse prelaminations of $S^1$.   Given any $f\in \mathrm{Homeo}(S^1)$, there exists a homeomorphism $F$ of $\bD^2$ sending geodesics of $\geo(L^\pm)$ to geodesics of $\geo(f(L^\pm))$ and restricting to $f$ on $S^1$.  
\end{theorem}

We also show that naturality of geodesic realizations may fail without the fully transverse assumption (see example \ref{ex:non_natural}.  

Previously in this work, we freely passed between prelaminations of the circle and their geodesic realizations. Since the goal of this appendix is to justify that practice, we will be much more careful with notation.  
\begin{notation} 
Given $\alpha\in L^{\pm}$, we write $\bar \alpha$ for its geodesic realization, i.e., if $\alpha=\{a_1, a_2\}$ then $\bar{\alpha}$ is the straight line segment between $a_1$ and $a_2$.  
Inversely, given a geodesic $\ell \subset \bD$ we denote by $\partial \ell$ the pair of endpoints of $\ell$ in $S^1$. 
\end{notation}

Sometimes it will be convenient to refer to pairs of points explicitly. Thus, extending our previous notation, we also define
\begin{notation} 
Given $p, q \in \bD^2$, let $\geo(p,q)$ denote the straight line between $p$ and $q$.  
\end{notation}

\begin{proof}[Proof of Theorem \ref{t.naturality}]
By Lemma \ref{lem:extend_lamination}, we may (by replacing $L^\pm$ by their closures, if needed) assume without loss of generality that $L^\pm$ are fully transverse {\em laminations}. 
Let $f \in \Homeo(S^1)$, without loss of generality we may assume that $f$ preserves orientation of $S^1$.  For $\{a, a'\} = \alpha$ in $L^\pm$ we will write $f(\alpha)$ to mean $\{f(a), f(a')\}$.  

Define $F$ as follows.   First, set $F|_{S^1} = f$.  
Secondly, for a pair of crossing leaves $\alpha$ and $\beta$ in $L^{\pm}$, define $F(\bar \alpha \cap \bar \beta) = \overline{f(\alpha)} \cap \overline{f(\beta)}.$   Note that this is well defined since $f$ is a homeomorphism so preserves crossing, and distinct pairs of leaves always cross at different points.  

Let $X = X(L^\pm) \subset \bD^2$
denote the set of intersection points of pairs of crossing geodesics realization of leaves.  So $F$ is thus far defined on $X \cup S^1$.
Our first goal is to show that $F$ is {\em continuous} on $X\cup S^1$. This follows from the following two claims.

\begin{claim}  \label{claimA1}
Suppose that $\alpha_n, \beta_n$ are crossing leaves in $L^+, L^-$ respectively, and $\overline{\alpha_n} \cap \overline{\beta_n} \to p \in S^1$.  Then $\overline{f(\alpha_n)} \cap \overline{f(\beta_n)}$ converges to $f(p)$.
\end{claim} 

\begin{proof}As $\bD^2$ is strictly convex, the fact that $\overline{\alpha_n} \cap \overline{\beta_n}$ tends to a point $p$ in $S^1$ implies that at least one of the points of $\alpha_n$ tends to $p$, and the same for $\beta_n$.

Consider the limits of the sequences of points $\alpha_n$ and $\beta_n$.  As a first case, if at least three limit points coincide, then either $\alpha_n$ or $\beta_n$ limits to $p$, and thus its image under $f$ limits to $f(p)$, and hence $\overline{f(\alpha_n)} \cap \overline{f(\beta_n)}$ converges to $f(p)$. 

Otherwise, letting $\alpha_n = \{a_n, a_n'\}$ and $\beta_n = \{b_n, b_n'\}$, we may have that $a_n \to p$, $b_n \to p$, but the other endpoints do not.  Passing to any convergent subsequence of endpoints, we will have $a_n' \to q$, $b_n' \to r$, for some $r \neq p$ and $q\neq p$.  Furthermore $r\neq q$ because $\{p,r\}$ and $\{p,q\}$ are leaves of distinct laminations $L^\pm$ and $L^\mp$.

Then $\overline{f(\alpha_n)} \cap \overline{f(\beta_n)}$ is the intersection point of geodesics converging to the geodesics $\geo(f(q),f(p))$ and $\geo(f(r),f(p))$, and thus converges to $f(p)$.
\end{proof}

\begin{claim}  \label{claimA2}
Suppose that $\alpha_n, \beta_n$ are pairs of crossing leaves in $L^+$ and $L^-$ respectively, and $\overline{\alpha_n} \cap \overline{\beta_n} \to p \in \bD^2 \smallsetminus S^1$.  Then $p = \bar{\alpha} \cap \bar{\beta}$ for some $\alpha, \beta$ in $L^+, L^-$, and 
$\overline{f(\alpha_n)} \cap \overline{f(\beta_n)}$ converges to $\overline{f(\alpha)} \cap \overline{f(\beta)}$.  
\end{claim} 

\begin{proof}
Let $\alpha_n = \{a_n, a_n'\}$ and $\beta_n = \{b_n, b_n'\}$.  Pass to a subsequence so that the sequences $a_n$, $a'_n$, $b_n$ and $b'_n$ all converge, say to $a, a', b$, and $b'$ respectively.   Since $\bar{\alpha_n} \cap \bar{\beta_n} \to p \in \bD^2 \smallsetminus S^1$, all four limit points are distinct.  Since we assumed that $L^{\pm}$ were laminations (hence closed), we deduce that $\{a, a'\} = \alpha$ and $\{b, b'\} = \beta$ for some $\alpha, \beta \in L^+$ and $L^-$ respectively; and necessarily $p = \bar{\alpha} \cap \bar{\beta}$.  

Since $L^+$ and $L^-$ are laminations, $p$ cannot be the intersection point of any other pair of crossing leaves in the realization, so 
$a_n$, $a'_n$, $b_n$ and $b'_n$ all converge (without needing to pass to subsequences).  
Thus, $f(\alpha_n)$ and $f(\beta_n)$ converge to $f(\alpha)$ and $f(\beta)$, and so $\overline{f(\alpha_n)} \cap \overline{f(\beta_n)}$ converges to $\overline{f(\alpha)} \cap \overline{f(\beta)}$.  
\end{proof} 

Let $X(f(L^\pm))$ denote the set of intersection points of leaves of $\geo(f(L^+) \cup f(L^-))$.  Summarizing our work above, we have: 
\begin{corollary}
The set $X\cup S^1$ is compact, as is its image $X(f(L^\pm)) \cup S^1$, and $F\colon X\cup S^1 \to X(f(L^\pm))\cup S^1$ is a homeomorphism.
\end{corollary} 

\begin{proof}
Claims \ref{claimA1} and \ref{claimA2} show that the domain of $F$ is a closed subset of $\bD^2$ and that $F$ is continuous.  Using $f^{-1}$ instead of $f$, by analogous construction one obtains an inverse map to $F$ defined on $S^1 \cup X(f(L^\pm))$, and the same argument shows it is continuous.
\end{proof} 

Next we extend the domain of definition of $F$ to  $\geo(L^{\pm})$.  
Let $\alpha \in L^\pm$.  On each segment $I$ of a geodesic $\bar{\alpha}$ bounded by points $x, y \in X\cup S^1$, and containing no other points on $X\cup S^1$ (where  $F$ is already defined), we define $F|_I$ to be the unique affine map (with respect to the usual affine structure inherited from $\R^2$) sending $I$ to the straight line segment with endpoints $F(x), F(y)$.

\begin{lemma}\label{c.Fonlaminations} 
We have $F(\geo(L^\pm))=\geo(f(L^\pm))$.
Moreover, the restriction of $F$ to a leaf $\geo(a,b)$ of $\geo(L^\pm)$, oriented from $a$ to $b$, is an increasing homeomorphism onto the leaf $\geo(f(a),f(b))$ of  $\geo(f(L^\pm))$ oriented from $f(a)$ to $f(b)$.
\end{lemma}

\begin{proof} 
Let $\alpha=\{a,b\}\in L^\pm$ be a leaf and orient $\bar \alpha$ from $a$ to $b$ and $\overline{f(\alpha)}$ from $f(a)$ to $f(b)$.  
Consider $X\cap \bar\alpha$ and $X(f(L^\pm))\cap \geo(f(a),f(b))$. Each is a totally ordered set (a subset of an oriented straight line), and the restriction of $F$ to $X\cap \bar\alpha$ is a bijection onto $X(f(L^\pm))\cap \geo(f(a),f(b))$.

We claim this bijection is strictly increasing.  To see this, consider the connected components $I, J$ of $S^1 \smallsetminus \{a, b\}$.  Each leaf of $L^\pm$ that crosses $\alpha$ has one endpoint in $I$ and one in $J$; orient $I$ and $J$ from $a$ to $b$ so that the order of intersection points of leaves of $L^\mp$ along $\alpha$, $I$ and $J$ all agree.  Since $f$ preserves orientation, the same is true for leaves of $f(L^\mp)$ along $f(\alpha)$, $f(I)$ and $f(J)$. 
The desired conclusion now holds by definition of $F$ along $\alpha$ as a piecewise affine map.  
\end{proof}

Our next goal is to extend $F$ to the complement of $\geo(L^{\pm}) \cup S^1$ in $\bD^2$.

\begin{lemma}  \label{l.complementary_component}
Let $P$ be a connected component of $\bD^2 \setminus (\geo(L^-)\cup\geo(L^+))$.  Then $P$ is a convex region whose closure contains at most one point of $S^1$.   The boundary of $P$ consists of either a finite union of segments of leaves of $L^+$ and $L^-$, or a union of infinitely many such segments, which accumulate only at the unique point of $\overline{P} \cap S^1$. 
\end{lemma}

\begin{proof} 
Let $P$ be a connected component of $\bD^2 \setminus (\geo(L^-)\cup\geo(L^+))$.  
Then $P$ is the intersection of two convex sets (which are connected components of $\bD^2 \setminus \geo(L^-)$ and $\bD^2 \setminus \geo(L^+)$, respectively) so is convex.   We first show its closure contains at most one point of $S^1$.  

Suppose for contradiction that $\overline{P}$ contains two points $p,q$ in $S^1$.  Since $\overline{P}$ is convex in $\bD^2$, we have $\geo(p,q) \subset \overline{P}$. This implies that no $L^\pm$-leaf crosses $\{p,q\}$, contradicting the connectedness and density assumption of fully transverse prelaminations.

Now, every point in the interior of $\bD^2$ admits a neighborhood with a chart to $\R^2$ such that leaves of $\geo(L^\pm)$ are sent to horizontal and vertical segments.  
This implies that in any compact subset of the interior of $\bD^2$, the boundary of $P$ is a simple curve consisting of alternating segments of $\geo(L^+)$ and $\geo(L^-)$.  This implies the statement given above.  
\end{proof}

Borrowing (and extending) our previous terminology, we call the closure $\overline{P}$ of a connected component $P$ of $\bD^2 \setminus (\geo(L^-)\cup\geo(L^+))$ a {\em complementary region}.   
\begin{corollary} \label{c.triangulation}
Each complementary region $\overline{P}$ admits a triangulation whose vertices are in $(X \cup S^1) \cap \overline{P}$ and whose edges are geodesic segments.
\end{corollary}

\begin{proof} 
Lemma \ref{l.complementary_component} implies that $\overline{P}$ is either a finite sided convex polygon (whose sides are segments of $\geo(L^-)\cup\geo(L^+)$) or a convex region whose boundary consists of alternating segments of $\geo(L^-)\cup\geo(L^+)$ which accumulate only at a single point $p$ of $S^1$ --- call this an ``infinite sided polygon"; it is the convex hull of its ``vertices" i.e., the points of $(X \cap \overline{P}) \cup \{p\}$.

In the first case, $\overline{P}$ obviously admits a triangulation, in the second one may triangulate it by connecting $p$ to each of the other vertices. 
\end{proof}

As a consequence of Lemma \ref{l.complementary_component} and the definition of $F$, for each complementary region $\overline{P}$ of $\geo(L^-)\cup\geo(L^+)\cup S^1$, there exists a unique complementary region of $\geo(f(L^-)\cup\geo(L^+)) \cup S^1$ which is the convex hull of the images by $F$ of the vertices of $\overline{P}$. We call this polygon $F(P)$.

Now, fix a triangulation $\cT_P$ of each complementary region $\overline{P}$ as given by Corollary~\ref{c.triangulation}, and call $F(\cT_P)$ the triangulation of $F(\overline{P})$ which has the corresponding combinatorial data (i.e., two vertices $x_1,x_2$ of $\overline{P}$ are joined by an edge in $\cT_{P}$ if and only if the vertices $F(x_1),F(x_2)$ are joined by an edge in $F(\cT_P)$).

We extend $F$ to $\overline{P}$ by using the unique affine map mapping a triangle of $\cT_P$ to the corresponding triangle in $F(\cT_P)$; i.e., if $v_1, v_2, v_3$ are vertices of a triangle $T$, for each convex combination $\sum t_i v_i \in T$ we define $F(\sum t_i v_i):=\sum t_i F(v_i)$.  Note that this agrees with the previous definition of $F$ for any points on the boundary of $T$ where $F$ is already defined.
At this stage $F$ is defined on all of $\bD^2$, and we need only show that it is continuous, since continuity of the inverse follows from the same argument.  

\subsection*{Continuity of $F$}
Let $x_n \in \bD^2$ be a sequence of points converging to $x\in \bD^2$.   Since $F$ agrees with $f$ on $S^1$, we may without loss of generality assume $x_n$ are not in $S^1$.  
For each $n$, the point $x_n$ is either inside a triangle of a complementary region, or lies on a segment of a geodesic leaf bounded by points of $X \cup S^1$.  To streamline the proof, we think of these segments as (degenerate) triangles, with two vertices equal to each other.  Thus, each $x_n$ lies in some ``triangle" $T_n$ with vertices in $X \cup S^1$.  

Let $(v_{1,n}, v_{2,n}, v_{3,n})$ denote the vertices of $T_n$, thus $x_n = \sum_i t_{i,n} v_{i,n}$ where $0 \leq t_{i,n} \leq 1$ and $\sum_i t_{i,n}=1$.  
By compactness of $X \cup S^1$, we can pass to a subsequence $T_j$ so that we have convergence $v_{i,j} \to v_i \in X\cup S^1$. 
If all three points $v_1,v_2,v_3$ coincide, then they necessarily agree with $x$.  Now $F(v_{i,j})$ tends to $F(x)$ for all $i$, as $F$ is a homeomorphism on $X\cup S^1$. Thus  $F(x_j)$ converges to $F(x)$ in that case.

\begin{claim}  If $v_1 \neq v_2$, then 
the interior of the geodesic segment between $v_1$ and $v_2$ does not cross any leaf of $L^-\cup L^+$, nor does it cross the interior of a side of a triangle in the triangulation $\cup_P \cT_P$.  
\end{claim} 
\begin{proof} Crossing a leaf, and crossing the interior of a side of a triangle are both open conditions, and no segment $\geo(v_{1,j}, v_{3,j})$ crosses such a leaf or side.  Thus, the claim follows, and 
so the segment between $v_1$ and $v_2$ is either contained in a triangle or in a leaf. 
\end{proof} 

As an immediate consequence of this claim, $v_1, v_2$ and $v_3$ all lie in the closure of a single triangle or a leaf.  In either case, $F$ is affine on the convex hull of $v_1, v_2, v_3$ and so $F(x_j)$ converges to $F(x)$ as desired.  
Since this was true for any choice of subsequence provided that the vertices converge, we conclude $F(x_n)$ converges to $F(x)$ which finishes the proof Theorem \ref{t.naturality}.    
\end{proof} 

The following example, illustrated in Figure \ref{fig:appendix} shows that Theorem \ref{t.naturality} fails without the fully transverse assumption.
The example can easily be adapted to produce one where endpoints are dense, but the connectedness property is not satisfied.  

   \begin{figure}[h]
     \centering
     \includegraphics[width=7cm]{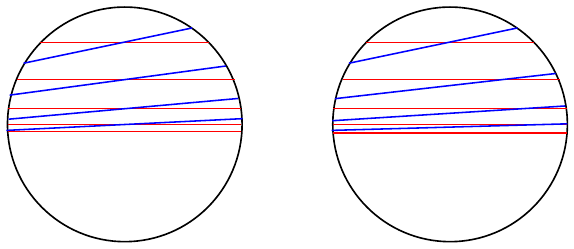}
     \caption{Two (partially defined) prelaminations that are homeomorphic on $S^1$ but with no extension to $\bD^2$.} 
     \label{fig:appendix}
   \end{figure}

\begin{example} \label{ex:non_natural}
We work with the unit disc $\bD^2$ in $\R^2$.  
Let $(a_n, a'_n)$ be the points of intersection of $S^1$ with the line $y = 1/n$ in $\bD^2$, for $n = 2, 3, \ldots$, so $a_n \to (-1,0)$ and $a'_n \to (1,0)$.  
Let $L^+$ be the collection of all such leaves $\{a_n, a'_n\}$.  
Consider sequences of points $b_n \in [a_n, a_{n+1}]$ and $b'_n \in [a'_n, a'_{n-1}]$.  We may choose these so that 
$\geo(a_n, a'_n) \cap \geo(b_n, b'_n)$ is any specified point on the line segment $\geo(a_n, a'_n)$.  To fix an example, choose $b_n$ and $b'_n$ such that $\geo(a_n, a'_n) \cap \geo(b_n, b'_n) = (0, 1/n)$.  Let $L^-$ be the collection of leaves $\{b_n, b'_n\}$.  

Now choose $c_n \in [a_n, a_{n+1}]$ and $c'_n  \in [a'_n, a'_{n-1}]$ such that $\geo(a_n, a'_n) \cap \geo(c_n, c'_n)$ converges to the point $(0,1)$.  There is a homeomorphism  $f$ of $S^1$ pointwise fixing $\{a_n, a'_n : n\geq 2 \in \bN\}$ and with $f(b_n) = c_n$, $f(b'_n) = c'_n$.   Obviously such a homeomorphism cannot extend continuously to a homeomorphism of $\bD^2$ taking lines of $\geo(L^\pm)$ to lines of $\geo(f(L)^\pm)$.   
\end{example}

\bibliographystyle{amsalpha}
\bibliography{refs_prelam}

\end{document}